\documentclass[fleqn,usenatbib,useAMS]{mnras}

\usepackage{graphicx}	
\usepackage{amsmath}	
\usepackage{amssymb}	
\usepackage{bm}		
\usepackage{pdflscape}	

\usepackage{ulem}
\usepackage{float}
\usepackage{caption}
\usepackage{subfigure}
\usepackage{epsfig}
\usepackage{epstopdf}
\usepackage{multirow}
\usepackage{booktabs}

\newtheorem{theorem}{Theorem}[section]
\newtheorem{remark}[theorem]{Remark}
\newtheorem{proposition}[theorem]{Proposition}

\newcommand{\RNum}[1]{\uppercase\expandafter{\romannumeral #1\relax}}
\numberwithin{equation}{section}
\graphicspath{{./figures/}} 

\usepackage[T1]{fontenc}
\usepackage{ae,aecompl}

\DeclareRobustCommand{\VAN}[3]{#2}
\let\VANthebibliography\thebibliography
\def\thebibliography{\DeclareRobustCommand{\VAN}[3]{##3}\VANthebibliography}
\usepackage{newtxtext,newtxmath}

\title[DG methods for the Euler--Poisson equations]{Energy conserving and well-balanced discontinuous Galerkin methods for the Euler--Poisson equations in spherical symmetry$^{\star}$}

\author[W. Zhang, Y. Xing, E. Endeve]{Weijie Zhang$^{1}$\thanks{Contact e-mail: \href{zhwj@mail.ustc.edu.cn}{zhwj@mail.ustc.edu.cn}}%
	, Yulong Xing$^{2}$\thanks{Contact e-mail: \href{xing.205@osu.edu}{xing.205@osu.edu}}%
	, Eirik Endeve$^{3,4}$\thanks{Contact e-mail: \href{endevee@ornl.gov}{endevee@ornl.gov}}%
	\\
	$^{1}$School of Mathematical Sciences, University
	of Science and Technology of China, Hefei, Anhui 230026, P.R. China
	\\
	$^{2}$Department of Mathematics, The Ohio State University, Columbus, OH 43210, USA
	\\
	$^{3}$Multiscale Methods and Dynamics Group, Oak Ridge National Laboratory, Oak Ridge, TN 37831, USA
	\\
	$^{4}$Department of Physics and Astronomy, University of Tennessee Knoxville, TN 37996, USA
	\\~\\
	$\star$This manuscript has been authored in part by UT-Battelle, LLC, under contract DE-AC05-00OR22725 with the US Department of Energy (DOE). \\
	The US government retains and the publisher, by accepting the article for publication, acknowledges that the US government retains a nonexclusive, \\
	paid-up, irrevocable, worldwide license to publish or reproduce the published form of this manuscript, or allow others to do so, for US government \\
	purposes. DOE will provide public access to these results of federally sponsored research in accordance with the \\
	DOE Public Access Plan (http://energy.gov/downloads/doe-public-access-plan).}

\date{Last updated xxx; in original form xxx}

\pubyear{2021}

\begin{document}
	\label{firstpage}
	\pagerange{\pageref{firstpage}--\pageref{lastpage}}
	\maketitle
	
	\begin{abstract}
		This paper presents high-order Runge--Kutta (RK) discontinuous Galerkin methods for the Euler--Poisson equations in spherical symmetry. The scheme can preserve a general polytropic equilibrium state and achieve total energy conservation up to machine precision with carefully designed spatial and temporal discretizations. To achieve the well-balanced property, the numerical solutions are decomposed into equilibrium and fluctuation components which are treated differently in the source term approximation. One non-trivial challenge encountered in the procedure is the complexity of the equilibrium state, which is governed by the Lane--Emden equation. For total energy conservation, we present second- and third-order RK time discretization, where different source term approximations are introduced in each stage of the RK method to ensure the conservation of total energy. A carefully designed slope limiter for spherical symmetry is also introduced to eliminate oscillations near discontinuities while maintaining the well-balanced and total-energy-conserving properties. Extensive numerical examples --- including a toy model of stellar core-collapse with a phenomenological equation of state that results in core-bounce and shock formation --- are provided to demonstrate the desired properties of the proposed methods, including the well-balanced property, high-order accuracy, shock capturing capability, and total energy conservation.
	\end{abstract}
	
	
	
	
	
	\section{Introduction}
\label{sec:Introduction}
In this paper, we present high-order discontinuous Galerkin (DG) methods for the Euler--Poisson equations in spherical symmetry, which have the well-balanced property to preserve hydrostatic equilibrium states exactly and total energy conservation property at the same time. 

The Euler equations with gravitation have wide applications in geophysical and astrophysical flow problems. In the case of a time-dependent gravitational potential, the model can be coupled with the Poisson equation to represent the self-gravity, which leads to the Euler--Poisson equations. They play an important role in many geophysical and astrophysical flows, for example, core-collapse supernova explosions \cite{mullerSteinmetz1995,couch2013improved,muller2020review}, star formation \cite{Ostriker_2001,doi:10.1146/annurev.astro.45.051806.110602}, planet formation \cite{armitage2011dynamics,simon2016mass}, and plasma physics applications \cite{guo1998smooth,suzuki2011asymptotic}. Self-gravitating astrophysical dynamics are often physically complex, and numerical methods are usually employed to simulate such complicated systems. 

The Euler equations with gravitation belong to the family of hyperbolic conservation laws with source terms. One of the most important features of such systems is that they admit non-trivial time-independent steady state solutions. Well-balanced schemes are introduced to preserve such steady states exactly on the discrete level and shown to be efficient and accurate for capturing small perturbations to such steady states. These perturbations may be at the level of the truncation error of standard numerical schemes and can be hard to capture with relatively coarse meshes. The well-balanced methods have been widely studied in the context of the shallow water equations over a non-flat bottom topology, see e.g., \cite{bermudez1994upwind,leveque1998balancing,audusse2004fast,XS2005,NXS2007,gallardo2007well,xing2010positivity}. In recent years, well-balanced methods for the Euler equations with static gravity have attracted much attention and have been developed within several different frameworks; see e.g., \cite{xu2010well,kappeli2014well,chandrashekar2015second,kappeli2016well,thomann2019second} for first- and second-order schemes, and \cite{xing2013high,li2016high,ghosh2016well,LX2016,chandrashekar2017well,klingenberg2019arbitrary,veiga2019capturing,grosheintz2019high,castro2020well} for high-order schemes. Some of these works assume that the desired equilibrium is explicitly known \cite{klingenberg2019arbitrary,wu2021uniformly}, while others only need a pre-description of the desired equilibrium \cite{li2018well}, and work for a class of equilibria. Recently, several works are established without any information of the desired equilibrium state \cite{kappeli2016well,franck2016finite,CiCP-30-666}. For the Euler--Poisson equations considered in this paper, the equilibrium states are more complicated due to the coupling with the Poisson equation. 

For the Euler--Poisson equations, another important feature is that they conserve the total energy, which is defined as the sum of the potential, internal, and kinetic energies.  In the standard formulation of the Euler--Poisson equations, the effect of gravity is included as source terms, and the total energy conservation statement is obtained in a non-trivial way.  Thus, conserving the total energy numerically becomes challenging. For some systems, e.g., in hydrostatic equilibrium, the total energy can be much smaller than either the potential or internal energies, which means that even a small truncation error in standard methods for the potential energy can lead to a large error in the total energy, and eventually the wrong numerical solution \citep{jiang2011star}.  
Fully conservative schemes for the Euler--Poisson equations, which conserve mass, momentum, and total energy, have been studied under the framework of finite difference methods in the last fifteen years. One popular technique is to transfer the energy equation to the equation for total energy and rewrite the governing equations in conservative form, see e.g., \cite{jiang2013new}. Another popular technique does not involve the reformulation of the unknown variables, but apply integration by parts and the mass conservation equation to discretize the source term in the energy equation, see e.g., \cite{mikami2008three,hanawa2019conservation,mullen2021extension}. With a careful approximation of the source term in the energy equation, one can carry out a rigorous proof to show the conservation of total energy. In this paper, we adopt the second technique and study it in the framework of high-order finite element DG methods. 
We note that we solve the Euler--Poisson equations in spherical symmetry, where we are unable to formulate the momentum equation in conservative form.  
For this reason we do not consider momentum conservation in this paper \citep[cf.][]{jiang2013new,mullen2021extension}.

The main objective of this paper is to develop high-order DG methods for the Euler--Poisson equations, which are well-balanced and at the same time have the total energy conservation property. The well-balanced DG scheme for the Euler equations with a time-independent gravitational potential was studied in \cite{li2018well}, where the key component to achieve the well-balanced property is to decompose the source into equilibrium and fluctuation components and treat them differently in the source term approximation. Here we consider the extension of this technique to the Euler--Poisson equations. One non-trivial difficulty encountered in the procedure is the complexity of the equilibrium state, which is now governed by the well known Lane--Emden equation. For total energy conservation, very recent work presented in \cite{mullen2021extension}, where a second-order finite difference, fully conservative scheme was proposed and studied. Here, the extension to the framework of DG methods is studied, which involves a special integration by parts and novel second- and third-order Runge--Kutta (RK) time discretization, where different source term approximations are introduced in each 
stage of RK method to ensure the conservation of total energy. A carefully designed slope limiter in spherical symmetry is also introduced to eliminate oscillations near discontinuities while still maintaining the well-balanced and total-energy-conserving properties. To the best of our knowledge, the design of well-balanced methods for the Euler--Poisson system has not been studied in the 
context of DG methods, and there are no existing Runge--Kutta discontinuous Galerkin (RKDG) schemes which can conserve the total energy for the Euler--Poisson equations. This is the first paper trying to tackle both challenges simultaneously. 

The main motivating astrophysical application for the present work is the simulation of core-collapse supernovae (CCSNe) in the context of non-relativistic, self-gravitating hydrodynamics with DG methods \citep[see also][]{pochik2021thornado}.  
After the collapse of the iron core of a massive star, the inner core settles into an approximate hydrostatic equilibrium, which is not easily captured by standard numerical methods, unless relatively high spatial resolution is used \citep{kappeli2016well}.  
Moreover, conserving the total energy in CCSN simulations with standard numerical methods and moderate spatial resolution is challenging \citep[e.g.,][]{muller2010new}.  
The kinetic energy of the explosion is a key quantity of interest targeted by CCSN simulation codes, and is typically on the order of $10^{51}$~erg \citep[or less; e.g.,][]{lentz_etal_2015,melson_etal_2015,burrows_etal_2020}.  
Thus, for reliable estimates of the explosion energy, the total energy should be conserved to well within this threshold.  
The use of high-order, well-balanced, and energy conserving numerical methods, as developed in this paper, may help provide reliable estimates for quantities of interest from CCSN simulations at a reduced computational cost.  

The rest of the paper is organized as follows. In Section \ref{sec:model}, we introduce the Euler--Poisson equations, their steady-state solutions, and discuss total energy conservation. In Section \ref{sec:methods}, we present the structure-preserving numerical methods for the Euler--Poisson equations. We start by introducing the conventional DG methods for the Euler--Poisson equations, and then discuss the well-balanced modifications and total-energy-conserving source term and time discretization, which leads to our well-balanced and total-energy-conserving fully discrete RKDG scheme. In Section \ref{sec:example}, numerical examples are given to verify the properties of our proposed methods. Concluding remarks are provided in Section \ref{sec:conclusion}.

\section{Mathematical model}\label{sec:model}

In this section, we introduce the Euler equations with self-gravity in spherical symmetry, and discuss the steady-state solutions and total energy conservation property of the model.

\subsection{Euler--Poisson equations}

The Euler equations in spherical symmetry take the form
\begin{align}
	&\frac{\partial \rho}{\partial t} +\frac{1}{r^{2}}\frac{\partial}{\partial r}\Big(\,r^{2}\,\rho u\,\Big)
	=0,\label{eq:mass}\\
	&\frac{\partial \rho u}{\partial t}
	+\frac{1}{r^{2}}\frac{\partial}{\partial r}\Big(\,r^{2}\,\big(\,\rho u^{2}+p\,\big)\,\Big)
	=\frac{2\,p}{r}-\rho\,\frac{\partial \Phi}{\partial r},\label{eq:momentum}\\
	&\frac{\partial E}{\partial t}
	+\frac{1}{r^{2}}\frac{\partial}{\partial r}\Big(\,r^{2}\,\big(\,E+p\,\big)\,u\,\Big)
	=-\rho u\,\frac{\partial \Phi}{\partial r},\label{eq:energy}
\end{align}
where $r$ is the radial coordinate, $\rho$ is the mass density, $u$ denotes the fluid velocity, $p$ is the pressure, and $E=\rho e+\frac{1}{2}\,\rho\,u^{2}$ is the total non-gravitational energy with $e$ being the specific internal energy. 
An additional thermodynamic equation to link $p$ with $(\rho,e)$, called the equation of state (EoS), is needed. For ideal gases, it is given by
\begin{equation}\label{eq:eos}
	p=(\gamma-1)\,\rho e,
\end{equation}
where $\gamma$ is the (constant) ratio of specific heats. The gravitational potential $\Phi$ can be obtained from the density $\rho$ via the Poisson equation
\begin{equation}
	\frac{1}{r^{2}}\frac{\partial}{\partial r}\Big(\,r^{2}\,\frac{\partial \Phi}{\partial r}\,\Big)=4\pi\,G\,\rho,
	\label{eq:poisson}
\end{equation}
where $G$ is the gravitational constant. 
The coupling of these two models yield the Euler--Poisson equations in spherical symmetry.

\subsection{Steady states and the Lane--Emden equation}\label{sec:lane-emden}

The Euler equations \eqref{eq:mass}-\eqref{eq:energy} admit the following zero-velocity steady states:
\begin{equation}
	\rho=\rho(r),\qquad u=0,\qquad \frac{\partial p}{\partial r}=-\rho\frac{\partial\Phi}{\partial r}.
	\label{eq:equilibrium}
\end{equation}
Considering the polytropic hydrostatic equilibrium characterized by
\begin{equation}
	p=\kappa\rho^\gamma,
	\label{eq:polytropic}
\end{equation}
we can combine \eqref{eq:poisson}, \eqref{eq:equilibrium} and \eqref{eq:polytropic} to obtain the steady-state equation
\begin{equation}
	\frac{1}{r^2}\frac{\partial}{\partial r}\left(\frac{r^2}{\rho}\kappa\gamma\rho^{\gamma-1}\frac{\partial\rho}{\partial r}\right)=-4\pi \,G\,\rho,
	\label{eq:lane-emden1}
\end{equation}
which is the equation satisfied by $\rho(r)$. By introducing the quantities $\theta$ and $n$ defined by 
\begin{align}
	\rho&\equiv\lambda\theta^n,\qquad
	\gamma\equiv\frac{n+1}{n},
\end{align}
with $\lambda\equiv\rho_c$ being the value of density $\rho$ at the center $r=0$, the equation \eqref{eq:lane-emden1} can be simplified as
\begin{equation}
	\frac{(n+1)\kappa\lambda^{\frac{1-n}{n}}}{4\pi\,G}\frac{1}{r^2}\frac{\partial}{\partial r} \left(r^2\frac{\partial\theta}{\partial r} \right)=-\theta^n.
\end{equation}	
Let us define the scaled radial coordinate $\xi$ as
\begin{align}\label{eq:lane-emden-coef}
	\xi\equiv&\frac{r}{\alpha},\qquad
	\alpha\equiv\sqrt{\frac{(n+1)\kappa\lambda^{\frac{1-n}{n}}}{4\pi\,G}},
\end{align}
and this equation can be non-dimensionalized into the well-known Lane--Emden equation for the polytropic hydrostatic equilibrium:
\begin{equation}
	\frac{1}{\xi^2}\frac{\partial}{\partial\xi}\left(\xi^2\frac{\partial \theta}{\partial\xi}\right)=-\theta^n.
	\label{eq:lane-emden2}
\end{equation}
As a second-order ordinary differential equation for $\theta(\xi)$, it requires two boundary conditions:
\begin{enumerate}
	\item Since $\lambda\equiv\rho_c=\left. \rho\right|_{\xi=0}$ and $\rho=\lambda\theta^n$, we have $\left. \theta\right|_{\xi=0}=1$ at the center $\xi=0$;
	\item The polytropic equilibrium \eqref{eq:polytropic} leads to 
	\begin{equation}\label{eq:pderivative}
		\frac{\partial p}{\partial r}=\kappa\gamma\rho^{\gamma-1}\frac{\partial\rho}{\partial r} ~~\propto ~~\frac{\partial\theta}{\partial\xi}.
	\end{equation}	
	We have ${\partial p}/{\partial r}=-\rho~ {\partial\Phi}/{\partial r}=0$ at $r=0$ (because there is no mass inside zero radius). 
	Therefore, we conclude that 
	\begin{equation}\label{eq:lane-emden-boundary2}
		\left. \frac{\partial\theta}{\partial\xi}\right|_{\xi=0}=0.
	\end{equation}
\end{enumerate}

\begin{remark}
	The methods presented in this paper are to preserve the steady state \eqref{eq:polytropic} for the ideal EoS \eqref{eq:eos} up to round-off errors, but can deal with problems for general EoS without preserving the steady states up to machine error.
\end{remark}

\subsection{Total energy conservation}\label{sec2.3}

The solutions of the Euler--Poisson system \eqref{eq:mass}-\eqref{eq:poisson} satisfy the following conservation law for the total energy:
\begin{equation}\label{eq:total-energy-continuous}
	\frac{\partial}{\partial t}\left(E+\frac12\,\rho\,\Phi\right)+\frac{1}{r^{2}}\frac{\partial}{\partial r}\Big(\,r^{2}\left(\big(\,E+p\,\big)\,u+F_g\right)\Big)=0,
\end{equation}
where
\begin{equation}\label{eq:consrve_energy_flux}
	F_g=\frac{1}{8\pi\,G}\left(\Phi\,\frac{\partial^2}{\partial r\partial t}\Phi-\frac{\partial}{\partial t}\Phi\,\frac{\partial}{\partial r}\Phi\right)+\rho u\,\Phi,
\end{equation}		
which leads to the total energy conservation
\begin{equation}\label{eq:total-energy}
	\frac{\partial}{\partial t}\int_{\Omega}\left(E+\frac12\,\rho\,\Phi\right)\,r^2\,\mathrm{d}r=0,
\end{equation}
if the boundary fluxes are zero. Here $\frac12\,\rho\,\Phi$ is the canonical gravitational energy density 
of a self-gravitating system. 

Below, we sketch the main derivation steps of \eqref{eq:total-energy-continuous}, which will be useful in the derivation of the total-energy-conserving numerical methods. Let us decompose the time derivative into two terms as
\begin{align}
	\frac{\partial}{\partial t}\left(E+\frac12\,\rho\,\Phi\right)r^2
	&=\left(\frac{\partial E}{\partial t}+\frac12\frac{\partial \rho}{\partial t}\Phi + \frac12\rho\,\frac{\partial \Phi}{\partial t}\right)r^2\nonumber
	\\
	&=\left(\frac{\partial E}{\partial t}+\frac{\partial \rho}{\partial t}\,\Phi\right)r^2+\frac12\left(\rho\,\frac{\partial \Phi}{\partial t}-\frac{\partial \rho}{\partial t}\,\Phi\right)r^2.
\end{align}
For the first term, we have
\begin{align}
	&\left(\frac{\partial E}{\partial t}+\frac{\partial \rho}{\partial t}\,\Phi\right)r^2\nonumber\\
	&=\left(-\frac{\partial}{\partial r}\Big(\,r^{2}\,\big(\,E+p\,\big)\,u\,\Big)
	-\rho u\,\frac{\partial \Phi}{\partial r}\,r^2-\frac{\partial}{\partial r}\Big(\,r^{2}\,\rho u\,\Big)\,\Phi\right)\nonumber\\
	&=\left(-\frac{\partial}{\partial r}\Big(\,r^{2}\,\big(\,E+p\,\big)\,u\,\Big)
	-\frac{\partial}{\partial r}\Big(\,r^{2}\,\rho u\,\Phi\Big)\right)\nonumber\\
	&=-\frac{\partial}{\partial r}\Big(\,r^{2}\left(\big(\,E+p\,\big)\,u+\rho u\,\Phi\,\Big)\right),
\end{align}
which follows from Eq. \eqref{eq:mass} and \eqref{eq:energy}. For the second term, we have
\begin{align}
	&\frac12\left(\rho\,\frac{\partial \Phi}{\partial t}-\frac{\partial \rho}{\partial t}\,\Phi\right)r^2\nonumber\\
	&=\frac{1}{8\pi\,G}\left(\frac{\partial}{\partial r}\Big(\,r^{2}\,\frac{\partial \Phi}{\partial r}\,\Big)\,\frac{\partial \Phi}{\partial t}-\frac{\partial}{\partial t}\frac{\partial}{\partial r}\Big(\,r^{2}\,\frac{\partial \Phi}{\partial r}\,\Big)\,\Phi\right)\nonumber\\
	&=\frac{1}{8\pi\,G}\left(\frac{\partial}{\partial r}\Big(\,r^{2}\,\frac{\partial \Phi}{\partial r}\frac{\partial \Phi}{\partial t}\Big)-r^{2}\,\frac{\partial \Phi}{\partial r}\frac{\partial^2\Phi}{\partial r\partial t}\right.\nonumber\\
	&\hskip1.2cm\left.-\frac{\partial}{\partial r}\Big(\,r^{2}\,\frac{\partial^2\Phi}{\partial r\partial t}\,\Phi\Big)+\frac{\partial}{\partial t}\Big(\,r^{2}\,\frac{\partial \Phi}{\partial r}\,\Big)\,\frac{\partial\Phi}{\partial r}\right)\nonumber\\
	&=\frac{1}{8\pi\,G}\left(\frac{\partial}{\partial r}\Big(\,r^{2}\,\frac{\partial \Phi}{\partial r}\frac{\partial \Phi}{\partial t}\Big)-\frac{\partial}{\partial r}\Big(\,r^{2}\,\frac{\partial^2\Phi}{\partial r\partial t}\,\Phi\Big)\right),
\end{align}
which follows from Eq. \eqref{eq:poisson} and integration by parts. The combination of these leads to the conservative form of the total energy \eqref{eq:total-energy-continuous}. 

\begin{remark}
    We note that the form of the energy flux in Eq. \eqref{eq:consrve_energy_flux} is not unique \citep{jiang2013new,mullen2021extension}. The different energy fluxes will not affect the numerical methods proposed in this paper, which will be derived based on the original form \eqref{eq:mass}-\eqref{eq:poisson}. The energy flux in Eq. \eqref{eq:consrve_energy_flux} is introduced only as a tool for the proof of the total energy conservation property.
\end{remark}

	\section{Numerical methods}\label{sec:methods}

In this section, we present the high-order, total-energy-conserving, and well-balanced DG scheme for the Euler--Poisson equations \eqref{eq:mass}-\eqref{eq:poisson}, which preserves the polytropic equilibrium \eqref{eq:lane-emden1}, and at the same time has the total energy conservation property  \eqref{eq:total-energy} on the discrete level.

\subsection{Notations}

Let us divide the computational domain $\Omega=\{r:r\in[0,R]\}$ into computational cells
\begin{equation}
K_j=\{r:r\in[r_{j-\frac{1}{2}},r_{j+\frac{1}{2}}]\}\quad\mbox{and}\quad\Delta r_j=r_{j+\frac{1}{2}}-r_{j-\frac{1}{2}},
\end{equation}
for $j=1,...,N$. We define the finite dimensional function space
\begin{equation}
	\mathcal{V}_h:=\{v\in L^2(\Omega):\,v|_{K_j}\in P^k(K_j),\,\forall\,1\le j\le N\},
\end{equation}
where $P^k$ denotes the polynomial space up to degree $k$, and let 
\begin{equation}
	\boldsymbol{\Pi}_h:=\{(\zeta,\psi,\delta)^T:\,\zeta,\psi,\delta\in\mathcal{V}_h\}.
\end{equation}
For any unknown variable $u$, we denote its numerical approximation in the DG method by $u_h$, which belongs to the piecewise polynomial space $\mathcal{V}_h$. For $\psi\in\mathcal{V}_h$, the limit values at the cell boundaries $r_{j+\frac{1}{2}}$ from the left and the right are defined by
\begin{equation}
	\psi_{j+\frac{1}{2}}^-:=\lim_{\epsilon\rightarrow 0^+}\psi(r_{j+\frac{1}{2}}-\epsilon),\quad \psi_{j+\frac{1}{2}}^+:=\lim_{\epsilon\rightarrow 0^+}\psi(r_{j+\frac{1}{2}}+\epsilon).
\end{equation}

We introduce the Gauss-Radau projection, to be used later in designing the well-balanced methods. For a function $u\in L^2(\Omega)$ and $k\ge 1$, we define its projection $Pu$ into the space $\mathcal{V}_h$ as 
\begin{equation} \label{projection_Radau1}
	\int_{K_j}Pu\,\psi\,\mathrm{d}r=\int_{K_j}u\,\psi\,\mathrm{d}r,\quad  \forall \psi |_{K_j}\in P^{k-1}(K_j),
\end{equation}
for every cell $K_j$ and
\begin{equation}\label{projection_Radau2}
	Pu(r_{j-\frac{1}{2}}^+)=u(r_{j-\frac{1}{2}}^+).
\end{equation}

\subsection{The approximation of the gravitational potential}

Compared with the Euler equations with static gravitational field studied in \cite{li2018well,wu2021uniformly}, the Euler--Poisson equations \eqref{eq:mass}-\eqref{eq:poisson} involve the additional Poisson equation \eqref{eq:poisson} which governs the relation between time dependent $\Phi$ and the density $\rho$. There are extensive numerical methods that could be used to solve the Poisson equation. Here, we present the following simple approach to compute $\Phi$ numerically. 

Note that the source terms in \eqref{eq:momentum} and \eqref{eq:energy} involve only the derivative $\partial \Phi/\partial r$, however, we will compute the numerical approximation of both $\partial \Phi/\partial r$ and $\Phi$ in this paper, denoted by $\partial \Phi_h/\partial r$ and $\Phi_h$ respectively, as the latter will be used in the design of total-energy-conserving methods.

We can integrate the Poisson equation \eqref{eq:poisson} directly and obtain
\begin{align}
	\frac{\partial\Phi_h}{\partial r}&=\frac{4\pi\,G}{r^2}\int_0^r \rho_h\tau^2\,\mathrm{d}\tau,
	\label{eq:DG-poisson}\\
	\Phi_h&=\Phi_h(R)-\int_r^R\frac{\partial\Phi_h}{\partial r}\,\mathrm{d}r,\label{eq:DG-poisson2}
\end{align}
with the boundary conditions ${\partial\Phi_h(0)}/{\partial r}=0$ and $\Phi_h(R)=\text{constant}$. The equations \eqref{eq:DG-poisson} and \eqref{eq:DG-poisson2} mean that we calculate $\frac{\partial\Phi_h}{\partial r}$ and $\Phi_h$ cell by cell that
\begin{align}
	\frac{\partial\Phi_h}{\partial r}(r)&=\frac{4\pi\,G}{r^2}\int_{r_{j-\frac12}}^r \rho_h\tau^2\,\mathrm{d}\tau+\frac{r_{j-\frac12}^2}{r^2}\frac{\partial\Phi_h}{\partial r}(r_{j-\frac12}),
	\label{eq:DG-poisson3}
\end{align}
for $r\in K_j$, $j=1,...,N$ and
\begin{align}
	\Phi_h(r)&=\Phi_h(r_{j+\frac12})-\int_r^{r_{j+\frac12}}\frac{\partial\Phi_h}{\partial r}\,\mathrm{d}r,\label{eq:DG-poisson4}
\end{align}
for $r\in K_j$, $j=N,...,1$. 
We set $\Phi_h(R)=0$ 
in the numerical tests of this paper to observe the total energy conservation up to round-off error. Note that $\rho_h$ is a piecewise polynomial of degree $k$, hence the integrals in \eqref{eq:DG-poisson3} and \eqref{eq:DG-poisson4} can be evaluated exactly over each computational cell $K_j$. The detailed procedure is summarized in the following steps.
\begin{enumerate}
	\item Assume $\rho_h$ is $P^k$ piecewise polynomial taking the form, for $r\in K_j$, $j=1,...,N$,
	\begin{equation}\label{eq:def-coefficient1}
		\rho_h(r)\Big|_{K_j}=\sum_{i=0}^{k}\rho_{j,i}\,r^i.
	\end{equation}
	\item Compute the integration in \eqref{eq:DG-poisson3} exactly and obtain $\partial \Phi_h/\partial r$ as
	\begin{align}
		\frac{\partial \Phi_h}{\partial r}(r)=&\frac{4\pi\,G}{r^2}\left.\sum_{i=0}^{k}\frac{\rho_{j,i}\,\tau^{i+3}}{i+3}\right|_{\tau=r_{j-\frac12}}^r+\frac{r_{j-\frac12}^2}{r^2}\frac{\partial \Phi_h}{\partial r}(r_{j-\frac12})\nonumber\\
		:=&\sum_{i=1}^{k+1}g_{j,i}\,r^i+\frac{g_{j,-2}}{r^2},\label{eq:integral-poisson}
	\end{align}
	for $r\in K_j$, $j=1,...,N$. 
	\item Compute the integration in \eqref{eq:DG-poisson4} exactly and obtain $\Phi_h$ as
	\begin{equation}
		\Phi_h(r)=\Phi_h(r_{j+\frac12})-\left.\left(\sum_{i=1}^{k+1}\frac{g_{j,i}\,\tau^{i+1}}{i+1}-\frac{g_{j,-2}}{\tau}\right)\right|_{\tau=r}^{r_{j+\frac12}},\label{eq:integral-poisson2}
	\end{equation}
	for $r\in K_j$, $j=N,...,1$. 
Here $\rho_{j,i}$ in \eqref{eq:def-coefficient1} and $g_{j,i}$ in \eqref{eq:integral-poisson} are the polynomial coefficients of degree $i$ ($i\ge0$) in the $j$-th cell for $\rho_h$ and ${\partial \Phi_h}/{\partial r}$ respectively. $g_{j,-2}$ in \eqref{eq:integral-poisson} are the coefficient of the term ${1}/{r^2}$ for ${\partial \Phi_h}/{\partial r}$.
\end{enumerate}


	\subsection{The standard DG scheme}\label{sec:convention}

In this subsection, we will briefly review the standard DG method for the Euler--Poisson equations \eqref{eq:mass}-\eqref{eq:poisson}, which will be used in the numerical section for comparison. For ease of presentation, we denote the equations \eqref{eq:mass}-\eqref{eq:energy} as:
\begin{equation}
	\label{eq:problem}
	\frac{\partial\boldsymbol{u}}{\partial t}+\frac{1}{r^2}\frac{\partial}{\partial r}(r^2\boldsymbol{f}(\boldsymbol{u}))=\boldsymbol{s}(\boldsymbol{u},\Phi),
\end{equation}
where
\begin{align}
	\boldsymbol{u}=\left(\begin{matrix}
		\rho\\\rho u\\E
	\end{matrix}\right),~\boldsymbol{f}(\boldsymbol{u})=\left(\begin{matrix}
		\rho u\\\rho u^2+p\\ (E+p)u
	\end{matrix}\right),~\boldsymbol{s}(\boldsymbol{u},\Phi)=\left(\begin{matrix}
		0\\\frac{2p}{r}-\rho\frac{\partial \Phi}{\partial r}\\-\rho u\frac{\partial \Phi}{\partial r}
	\end{matrix}\right).
\end{align}

To derive the semi-discrete DG scheme, we multiply the equations by $r^2$ and test functions, apply integration by parts and replace the boundary value by a monotone numerical flux, which leads to the following DG scheme: find $\boldsymbol{u}_h\in\boldsymbol{\Pi}_h$ such that for any test function $\boldsymbol{v}=(\zeta,\psi,\delta)^T\in\boldsymbol{\Pi}_h$, it holds that 
	\begin{align}
	&\partial_t\int_{K_j}\boldsymbol{u}_h\cdot\boldsymbol{v}\, r^2\mathrm{d}r+r_{j+\frac{1}{2}}^2\hat{\boldsymbol{f}}_{j+\frac{1}{2}}\cdot\boldsymbol{v}_{j+\frac{1}{2}}^--r_{j-\frac{1}{2}}^2\hat{\boldsymbol{f}}_{j-\frac{1}{2}}\cdot\boldsymbol{v}_{j-\frac{1}{2}}^+\nonumber\\
	&\hskip1.2cm-\int_{K_j}\boldsymbol{f}(\boldsymbol{u}_h)\cdot(\partial_r\boldsymbol{v})r^2\mathrm{d}r=\boldsymbol{s}_j, \label{eq:scheme0}
\end{align}
where $\boldsymbol{s}_j$ is the approximation of $\int_{K_j}\boldsymbol{s}(\boldsymbol{u},\Phi)\cdot\boldsymbol{v}r^2\mathrm{d}r$ taking the form
\begin{equation}\label{eq:source-term-standard}
	\boldsymbol{s}_j=\left(\begin{matrix}
	0\\s_j^{[2]}\\s_j^{[3]}
	\end{matrix}\right)=\left(\begin{matrix}
	0\\
	\int_{K_j}\left(\frac{2p_h}{r}-\rho_h\frac{\partial\Phi_h}{\partial r}\right)\psi\, r^2\mathrm{d}r\\
	\int_{K_j}-(\rho u)_h\frac{\partial\Phi_h}{\partial r}\delta\,r^2\mathrm{d}r
	\end{matrix}\right),
\end{equation}
and $\hat{\boldsymbol{f}}$ is the monotone numerical flux.
In this paper, to have good performance in capturing shocks and optimal error convergence rate, we consider the Harten-Lax-van Leer contact (HLLC) flux \citep{toro2013riemann} 
\begin{equation} \label{eq:hllc}
	\hat{\boldsymbol{f}}=\hat{\boldsymbol{f}}(\boldsymbol{u}_h^-,\boldsymbol{u}_h^+)=\begin{cases}
		\boldsymbol{f}(\boldsymbol{u}_h^-) & \text{ if }0\le S^-,\\
		\boldsymbol{f}(\boldsymbol{u}_h^-)+S^-\left(\boldsymbol{u}_*^--\boldsymbol{u}_h^-\right) & \text{ if }S^-\le0\le S^*,\\
		\boldsymbol{f}(\boldsymbol{u}_h^+)+S^+\left(\boldsymbol{u}_*^+-\boldsymbol{u}_h^+\right) & \text{ if }S^*\le0\le S^+,\\
		\boldsymbol{f}(\boldsymbol{u}_h^+) & \text{ if }0\ge S^+,
	\end{cases}
\end{equation}
where $S^-$, $S^+$ and $S^*$ are the signal speeds
\begin{align}
	&S^-=\min\{u_h^--c_h^-,~u_h^+-c_h^+\},\quad S^+=\max\{u_h^-+c_h^-,~u_h^++c_h^+\},\\
	&S^*=\frac{p_h^+-p_h^-+\rho_h^-u_h^-\left(S^--u_h^-\right)-\rho_h^+u_h^+\left(S^+-u_h^+\right)}{\rho_h^-\left(S^--u_h^-\right)-\rho_h^+\left(S^+-u_h^+\right)},
\end{align}
$c_h^-$, $c_h^+$ are the sound speeds calculated from $\boldsymbol{u}_h^-$, $\boldsymbol{u}_h^+$ respectively, and $\boldsymbol{u}_*^+$, $\boldsymbol{u}_*^+$ denote the intermediate states which can be computed via
\begin{equation}
	\boldsymbol{u}_*^\pm=\rho_h^\pm\left(\frac{S^\pm-u_h^\pm}{S^\pm-S^*}\right)\left(\begin{matrix}
		1\\
		S^*\\
		\frac{E_h^\pm}{\rho_h^\pm}+\left(S^*-u_h^\pm\right)\left(S^*-\frac{p_h^\pm}{\rho_h^\pm\left(S^\pm-u_h^\pm\right)}\right)
	\end{matrix}\right).
\end{equation}
The initial condition $\boldsymbol{u}_{h,0}\in\boldsymbol{\Pi}_h$ of the numerical method is given by
\begin{equation} \label{eq:initial}
	\boldsymbol{u}_{h,0}=P\boldsymbol{u}_{ex}(r,t=0),
\end{equation}
where $\boldsymbol{u}_{ex}(r,t=0)$ is the exact initial data, and $P$ stands for the Gauss-Radau projection \eqref{projection_Radau1}-\eqref{projection_Radau2}. 

\subsection{The well-balanced DG scheme}
In this subsection, we will introduce the well-balanced DG scheme which maintains the polytropic equilibrium \eqref{eq:lane-emden1}, or equivalently the Lane--Emden equation \eqref{eq:lane-emden2}. There are some recent works \citep{xing2014exactly,grosheintz2020high,pares2021well} on designing well-balanced methods for general steady states including non-zero equilibrium, which will be studied in future work.

\subsubsection{Solution of Lane--Emden equation}\label{sec:lane}

As illustrated in Section \ref{sec:lane-emden}, the polytropic equilibrium state of the Euler--Poisson equations is based on the solution of the Lane--Emden equation. The Lane--Emden equation can be analytically solved \citep{lane-emden} only for a few special integer values of the index $n$, as outlined below:
\begin{alignat}{2}
	&\text{\bf Analytical solution for n=0 (i.e., $\gamma=\infty$):} \quad &&\theta_0(\xi)=1-\frac{1}{6}\xi^2,  	\label{eq:lane_emden1}\\ 
	&\text{\bf Analytical solution for n=1 (i.e., $\gamma=2$):} \quad &&\theta_1(\xi)=\frac{\sin(\xi)}{\xi},  	\label{eq:lane_emden2}\\ 
	&\text{\bf Analytical solution for n=5 (i.e., $\gamma=\frac65$):} \quad &&\theta_5(\xi)=\frac{1}{\sqrt{1+\frac{1}{3}\xi^2}}.	\label{eq:lane_emden3}
\end{alignat}	

For all other values of $n$, we must resort to numerical solutions. Rewrite the equation \eqref{eq:lane-emden2} as
\begin{align}
	\frac{\partial\theta}{\mathrm{d}\xi}&=-\frac{\varphi}{\xi^2}, \qquad
	\frac{\partial\varphi}{\mathrm{d}\xi}=\theta^n\xi^2,
\end{align}
coupled with boundary conditions $\theta(0)=1$ and $\varphi(0)=0$. We denote them in the vector form by 
\begin{equation}
	\frac{\partial}{\partial \xi}\boldsymbol{y}=\boldsymbol{F}(\xi,\boldsymbol{y})\quad\text{with }\boldsymbol{y}=\left(\begin{aligned}
	\theta\\
	\varphi
	\end{aligned}\right)\quad\text{and}\quad \boldsymbol{F}(\xi,\boldsymbol{y})=\left(\begin{aligned}
	-\frac{\varphi}{\xi^2}\\
	\theta^n\xi^2
	\end{aligned}\right).\label{added1}
\end{equation}
Note that when $\xi=0$, we let $\boldsymbol{F}(0,\boldsymbol{y}(0))=0$ following the given boundary conditions. The equations \eqref{added1} is a system of ordinary differential equations, which can solved by various numerical methods. For example, we can use the fifth-order Runge-Kutta-Fehlberg technique in \cite{norsett1987solving}
\begin{equation}
	\boldsymbol{y}_{j+1}=\boldsymbol{y}_j+h\sum_{i=1}^sb_i\boldsymbol{k}_i,
	\label{eq:RKF45}
\end{equation}
where $\boldsymbol{y}_j$ denotes the numerical solution at the grid $\xi_j$, $h=\xi_{j+1}-\xi_j$ and $\boldsymbol{k}_i$, $i=1,2,\ldots,s$, is given by
\begin{align}
	\boldsymbol{k}_i&=\boldsymbol{F}(\xi_j+c_ih,\boldsymbol{y}_j+h(a_{i1}\boldsymbol{k}_1+a_{i2}\boldsymbol{k}_2+\cdots+a_{i,i-1}\boldsymbol{k}_{i-1})).
\end{align}
with the coefficients $a_{ij}$, $b_i$ and $c_i$ given in the following Butcher tableau:
\begin{align}
\renewcommand\arraystretch{1.2}
\begin{array}
	{c|cccccc}
	0 \\
	\frac{1}{4} & \frac{1}{4} \\
	\frac{3}{8} & \frac{3}{32} & \frac{9}{32} \\
	\frac{12}{13} & \frac{1932}{2197} & -\frac{7200}{2197} & \frac{7296}{2196} \\
	1 & \frac{439}{216} & -8 & \frac{3680}{513} & -\frac{845}{4104} \\
	\frac{1}{2} & -\frac{8}{27} & 2 & -\frac{3544}{2565} & \frac{1859}{4104} & -\frac{11}{40} \\
	\hline
	& \frac{16}{135} & 0 & \frac{6656}{12825} & \frac{28561}{56430} & -\frac{9}{50} & \frac{2}{55}
\end{array}
\end{align}
\noindent The numerical solution of \eqref{eq:lane-emden2} can be solved with enough accuracy by taking small enough $h$.
We note that the solution of the Lane--Emden equation only depends on $n$ (i.e. $\gamma$). For each computational example, $\gamma$ is fixed, hence we can pre-calculate and save the numerical solution $\theta_n$ at the beginning of the simulation.

\subsubsection{Decomposition of the numerical solutions}\label{sec:recovery}

To design the well-balanced method, we follow the approach in \cite{xing2014exactly} where well-balanced methods for the moving water equilibrium of the shallow water equations are designed. The first step is to separate the numerical solutions into the well-balanced equilibrium component $\boldsymbol{u}_h^e\in\boldsymbol{\Pi}_h$ and the fluctuation part $\boldsymbol{u}_h^f \in \boldsymbol{\Pi}_h$ at each time step, which will be elaborated below.

We start by recovering the desired equilibrium state $\boldsymbol{u}^d$ which satisfies the polytropic equilibrium \eqref{eq:lane-emden1} and usually does not belong to $\boldsymbol{\Pi}_h$. For the given $\gamma$ (or $n$), the solution $\theta_n$ of Lane--Emden equation \eqref{eq:lane-emden2} can be pre-computed. Then we evaluate the density and pressure of the numerical solution $\boldsymbol{u}_h(r,t)$ at the center $r=0$ and denote them by $\rho_0$ and $p_0$. By setting $\kappa=p_0/\rho_0^\gamma$ and $\alpha=\sqrt{\frac{\gamma}{\gamma-1}\kappa\rho_0^{\gamma-2}/(4\pi\,G)}$ in \eqref{eq:lane-emden-coef}, we can define the desired equilibrium state $\boldsymbol{u}^d$ as
\begin{equation}
	\label{eq:reference_operator}
	\boldsymbol{u}^d(r)=\left(\begin{matrix}
		\rho_0\left(\theta_n\left(\frac{r}{\alpha}\right)\right)^{\frac{1}{\gamma-1}}\\
		0\\
		\frac{\kappa}{\gamma-1}\rho_0^\gamma\left(\theta_n\left(\frac{r}{\alpha}\right)\right)^{\frac{\gamma}{\gamma-1}}
	\end{matrix}\right).
\end{equation}
Suppose the initial condition is in the equilibrium state, i.e., $\boldsymbol{u}_{ex}(r,0)$ satisfies the polytropic equilibrium \eqref{eq:lane-emden1}. Note that although $\boldsymbol{u}_{h,0}\in\boldsymbol{\Pi}_h$ defined in \eqref{eq:initial} is not in perfect equilibrium, the above procedure can recover the exact equilibrium, i.e., we can compute $\boldsymbol{u}^d$ from $\boldsymbol{u}_{h,0}$ with $\boldsymbol{u}^d=\boldsymbol{u}_{ex}(r,0)$.

Next we can define $\boldsymbol{u}_h^e\in\boldsymbol{\Pi}_h$ as the projection of $\boldsymbol{u}^d$ into the DG solution space:
\begin{equation}
	\boldsymbol{u}^e_h=P\boldsymbol{u}^d, 
	\label{eq:ue}
\end{equation}
and also define the fluctuation term $\boldsymbol{u}^f\in\boldsymbol{\Pi}_h$ as:
\begin{equation}
	\boldsymbol{u}^f_h=\boldsymbol{u}_h-\boldsymbol{u}^e_h.
	\label{eq:uf}
\end{equation}
For the $\theta_n$ explicitly given in \eqref{eq:lane_emden1}-\eqref{eq:lane_emden3}, the integration in the definition of the projection in Eq. \eqref{eq:ue} can be evaluated exactly. Otherwise, the integration is computed by using the values at the Gaussian quadrature points which can be obtained from interpolation.

\begin{remark}
	When recovering the desired equilibrium state $\boldsymbol{u}^d$, two practical issues in the implementation are noted. First, since the density is positive, $\theta(\xi)$ should also be positive for robustness of the simulation, and one should pay attention to the range of the solution of $\theta(\xi)$. If the analytical solution of the Lane-Emden equation is used, there is a constraint on the range of $\xi$ for $n=0,1$. For example, $\theta_0(\xi)>0$ for $\xi\in[0,\sqrt6)$ and $\theta_1(\xi)>0$ for $\xi\in[0,\pi)$. If the numerical solution of the Lane-Emden equation is used, $\theta(\xi)$ may become negative due to numerical integration errors. Therefore, if there is a range constraint on $\theta(\xi)$ and a cell $K_j$ where the value of $\theta(\xi)$ is outside of this range constraint, we set $\left.\boldsymbol{u}^d\right|_{K_j}=0$ for robustness of the simulation.
	Second, if the solution is too far away from the equilibrium state, for example, for the cells $K_j$ with
	\begin{equation}
		\rho^d(r_{j-\frac12})>2\rho_{h,j-\frac12}^+~\text{ or }~p^d(r_{j-\frac12})>2p_{h,j-\frac12}^+,
	\end{equation}
	we set $\left.\boldsymbol{u}^d\right|_{K_j}=0$ to avoid the accumulation of error since $\boldsymbol{u}^d$ is calculated globally. 
\end{remark}

\subsubsection{Well-balanced numerical flux and source term approximation}

With the decomposition of the numerical solutions into the equilibrium component $\boldsymbol{u}_h^e$ and the fluctuation part $\boldsymbol{u}_h^f $ at each time step, we can now present the well-balanced numerical fluxes and the well-balanced source term approximation. 

We can define the modified cell boundary values of $\boldsymbol{u}_h$ as
\begin{equation}
	\boldsymbol{u}_{h,j+\frac{1}{2}}^{*,-}=\boldsymbol{u}^d\left(r_{j+\frac{1}{2}}\right)+\boldsymbol{u}_{h,j+\frac{1}{2}}^{f,-},
	\quad\boldsymbol{u}_{h,j+\frac{1}{2}}^{*,+}=\boldsymbol{u}^d\left(r_{j+\frac{1}{2}}\right)+\boldsymbol{u}_{h,j+\frac{1}{2}}^{f,+},
	\label{eq:u*}
\end{equation}
where $\boldsymbol{u}^d$ is continuous over the whole computational domain 
and defined in \eqref{eq:reference_operator}, and $\boldsymbol{u}^f_h$ is defined in \eqref{eq:uf}.
The well-balanced numerical flux $\hat{\boldsymbol{f}}^*$ can be evaluated by
\begin{align}
	\hat{\boldsymbol{f}}^*&=\hat{\boldsymbol{f}}(\boldsymbol{u}_h^{*,-},\boldsymbol{u}_h^{*,+}),\label{eq:flux1}
\end{align}
		with $\hat{\boldsymbol{f}}$ being the HLLC flux defined in \eqref{eq:hllc}.
		
		For the well-balanced source term approximation, we follow the main idea in \cite{xing2014exactly,li2018well}, but with some modifications introduced below. As $s_j^{[3]}$ in \eqref{eq:source-term-standard} equals to zero automatically at the equilibrium state, we focus only on the term $s_j^{[2]}$. Since $\boldsymbol{u}^d$ is the equilibrium solution and continuous, we have 
		\begin{align}
			&r_{j+\frac{1}{2}}^2\boldsymbol{f}\left(\boldsymbol{u}^d\left(r_{j+\frac{1}{2}}\right)\right)\cdot\boldsymbol{v}_{j+\frac{1}{2}}^--r_{j-\frac{1}{2}}^2\boldsymbol{f}\left(\boldsymbol{u}^d\left(r_{j-\frac{1}{2}}\right)\right)\cdot\boldsymbol{v}_{j-\frac{1}{2}}^+\nonumber\\
			&\qquad
			-\int_{K_j}\boldsymbol{f}\left(\boldsymbol{u}^d\right)\cdot(\partial_r\boldsymbol{v})\,r^2\mathrm{d}r-\int_{K_j}\boldsymbol{s}\left(\boldsymbol{u}^d,\Phi^d\right)\cdot\boldsymbol{v}\,r^2\mathrm{d}r=0,\label{eq:source-term-ud}
		\end{align}
		where $\Phi^d$ is solved exactly from $\rho^d$ in \eqref{eq:poisson}. 
		Because $\boldsymbol{u}_h^e\in\boldsymbol{\Pi}_h$ is the projection of $\boldsymbol{u}^d$ with high-order accuracy, and $\boldsymbol{u}^d$ is continuous at the cell interfaces, we have
		\begin{align}
			&r_{j+\frac{1}{2}}^2f^{[2]}\left(\boldsymbol{u}^d\left(r_{j+\frac{1}{2}}\right)\right)\psi_{j+\frac{1}{2}}^-
			-r_{j-\frac{1}{2}}^2f^{[2]}\left(\boldsymbol{u}^d\left(r_{j-\frac{1}{2}}\right)\right)\psi_{j-\frac{1}{2}}^+\notag
			\\
			&\quad-\int_{K_j}f^{[2]}(\boldsymbol{u}^e_h)\,(\partial_r\psi)\,r^2\mathrm{d}r-\int_{K_j}\left(\frac{2p^e_h}{r}-\rho_h^e\frac{\partial\Phi^e_h}{\partial r}\right)\psi\,r^2\mathrm{d}r\notag\\
			&=\mathcal{O}((\Delta r_j)^{k+1}),
		\end{align}
		where $f^{[2]}$ denotes the second component of $\boldsymbol{f}$ and $\partial\Phi^e_h/\partial r$ is evaluated as in \eqref{eq:DG-poisson}:
		\begin{equation}
			\frac{\partial\Phi^e_h}{\partial r}=\frac{4\pi\,G}{r^2}\int_0^r\rho^e_h\tau^2\mathrm{d}\tau,
		\end{equation}
		with $\frac{\partial\Phi^e_h(0)}{\partial r}=0$. 
		The approximation of the source term $\boldsymbol{s}_j^{\text{wb}}$ is then defined as
		\begin{equation}
			\boldsymbol{s}_j^{\text{wb}}=\left[0,s_j^{[2],\text{wb}},s_j^{[3]}\right]^T, \qquad
			s_j^{[2],\text{wb}}=s_j^{[2]}+s_j^{[2],\text{cor}},
			\label{eq:source-term}
		\end{equation}
		where $s_j^{[2]}$ and $s_j^{[3]}$ are defined in \eqref{eq:source-term-standard} and the correction term $s_j^{[2],\text{cor}}$ takes the form
		\begin{align}\label{eq:source-term-cor}
			s_j^{[2],\text{cor}}=&~r_{j+\frac{1}{2}}^2 p^d\left(r_{j+\frac{1}{2}}\right)\psi_{j+\frac{1}{2}}^-
			-r_{j-\frac{1}{2}}^2 p^d\left(r_{j-\frac{1}{2}}\right)\psi_{j-\frac{1}{2}}^+
			\nonumber\\
			&-\int_{K_j} p^e_h\,(\partial_r\psi)\,r^2\mathrm{d}r-\int_{K_j}\left(\frac{2p^e_h}{r}-\rho_h^e\frac{\partial\Phi^e_h}{\partial r}\right)\psi\,r^2\mathrm{d}r,
		\end{align}
		which will play an important role in the well-balanced proof. 
		
		%

		\subsubsection{Well-balanced semi-discrete DG scheme}
		
		The well-balanced semi-discrete DG scheme can be written as:  find $\boldsymbol{u}_h\in\boldsymbol{\Pi}_h$ such that for any test function $\boldsymbol{v}=(\zeta,\psi,\delta)^T\in\boldsymbol{\Pi}_h$, it holds that 
		\begin{equation}
			\label{eq:scheme}
			\partial_t\int_{K_j}\boldsymbol{u}_h\cdot\boldsymbol{v}\, r^2\mathrm{d}r=\mathcal{L}_j(\boldsymbol{u}_h,\boldsymbol{v})=\mathcal{F}_j(\boldsymbol{u}_h,\boldsymbol{v})+ \boldsymbol{s}_j^{\text{wb}}(\boldsymbol{u}_h,\boldsymbol{v}),
		\end{equation}
		with $\boldsymbol{s}_j^{\text{wb}}$ defined in \eqref{eq:source-term} and
		\begin{align} 
			\mathcal{F}_j(\boldsymbol{u}_h,\boldsymbol{v})=&-r_{j+\frac{1}{2}}^2\hat{\boldsymbol{f}}^*_{j+\frac{1}{2}}\cdot\boldsymbol{v}_{j+\frac{1}{2}}^-+r_{j-\frac{1}{2}}^2\hat{\boldsymbol{f}}_{j-\frac{1}{2}}^*\cdot\boldsymbol{v}_{j-\frac{1}{2}}^+\nonumber\\
			&+\int_{K_j}\boldsymbol{f}(\boldsymbol{u}_h)\cdot(\partial_r\boldsymbol{v})r^2\mathrm{d}r,
		\end{align}
		with the source term approximation $\boldsymbol{s}_j^{\text{wb}}$ defined in \eqref{eq:source-term}, and the numerical flux $\hat{\boldsymbol{f}}^*$
		defined in \eqref{eq:flux1}. 
		We have the following result on its well-balanced property.
		\begin{proposition}
			\label{prop1}
			The semi-discrete DG scheme \eqref{eq:scheme}, with initial condition defined in \eqref{eq:initial}, maintains the equilibrium state \eqref{eq:lane-emden1} exactly.
		\end{proposition}
		\begin{proof}
			Suppose the initial condition is at the equilibrium state \eqref{eq:lane-emden1}. We will complete the well-balanced proof in three steps. First, we will show that $\boldsymbol{u}_h=\boldsymbol{u}^e_h$ and $\boldsymbol{f}^e_h=0$. By the definition of $\boldsymbol{u}^d$ in Eq. \eqref{eq:reference_operator}, we can conclude that $\boldsymbol{u}^d=\boldsymbol{u}_{ex}$ as both are the stationary solutions of \eqref{eq:lane-emden1} and share the same value at the center $r=0$. It then follows from \eqref{eq:ue} and \eqref{eq:uf} that $\boldsymbol{u}^e_h=\boldsymbol{u}_h$ and $\boldsymbol{u}^f_h=0$. Moreover, we conclude that $\partial\Phi_h/\partial r=\partial\Phi_h^e/\partial r$, because $\partial\Phi_h/\partial r$ and $\partial\Phi_h^e/\partial r$ are calculated from $\rho_h$ and $\rho_h^e$, respectively, using \eqref{eq:integral-poisson}, and $\rho_h=\rho_h^e$.
			
			Second, we would like to show that 
			$\hat{f}_{j+\frac{1}{2}}^{*,[2]}=p^d\left(r_{j+\frac{1}{2}}\right)$.
			Since $\boldsymbol{u}^f=0$, we have that
			$\boldsymbol{u}_h^{*,-}=\boldsymbol{u}_h^{*,+}=\boldsymbol{u}^d$
			at the interface $r_{j+1/2}$, following the definition \eqref{eq:u*}.
			In Eq. \eqref{eq:flux1}, we have
			\begin{align}\label{eq:prop-wb-0}
				\hat{\boldsymbol{f}}_{j+\frac12}^*&=\hat{\boldsymbol{f}}(\boldsymbol{u}_{h,j+\frac12}^{*,-},\boldsymbol{u}_{h,j+\frac12}^{*,+})=\boldsymbol{f}(\boldsymbol{u}_{h,j+\frac12}^{*,\pm})\nonumber\\
				&=\boldsymbol{f}\left(\boldsymbol{u}^d\left(r_{j+\frac12}\right)\right)=\left(\begin{matrix}
					0\\
					p^d\left(r_{j+\frac12}\right)\\
					0
				\end{matrix}\right),
			\end{align}
			where the last equality follows from the zero velocity in the vector $\boldsymbol{u}^d$.
			
			Lastly, it is easy to observe that the first and third components of $\mathcal{L}_j$ in Eq. \eqref{eq:scheme} are zero. With the source term defined in \eqref{eq:source-term}-\eqref{eq:source-term-cor}, the second component of of $\mathcal{L}_j$ can be simplified as
			\begin{align}
				&\mathcal{L}^{[2]}_j(\boldsymbol{u}_h,\boldsymbol{v})=\int_{K_j}f^{[2]}(\boldsymbol{u}_h)(\partial_r\psi)r^2\mathrm{d}r-r_{j+\frac{1}{2}}^2 \hat{f}_{j+\frac{1}{2}}^{*,[2]} \psi_{j+\frac{1}{2}}^-\nonumber\\
				&\qquad+r_{j-\frac{1}{2}}^2 \hat{f}_{j-\frac{1}{2}}^{*,[2]} \psi_{j-\frac{1}{2}}^++s_j^{[2],\text{wb}}\nonumber\\
				&\quad=\uline{\int_{K_j}f^{[2]}(\boldsymbol{u}_h)(\partial_r\psi)r^2\mathrm{d}r}
				-\dotuline{r_{j+\frac{1}{2}}^2 p^d\left(r_{j+\frac{1}{2}}\right)\psi_{j+\frac{1}{2}}^-}\nonumber\\
				&\qquad+\dashuline{r_{j-\frac{1}{2}}^2 p^d\left(r_{j-\frac{1}{2}}\right)\psi_{j-\frac{1}{2}}^+}+\uwave{\int_{K_j}\left(\frac{2p_h}{r}-\rho_h\frac{\partial\Phi_h}{\partial r}\right)\psi\, r^2\mathrm{d}r}\nonumber\\
				&\qquad-\uwave{\int_{K_j}\left(\frac{2p^e_h}{r}-\rho_h^e\frac{\partial\Phi^e_h}{\partial r}\right)\psi\,r^2\mathrm{d}r}-\uline{\int_{K_j} p^e_h\,(\partial_r\psi)\,r^2\mathrm{d}r}\nonumber\\
				&\qquad+\dotuline{r_{j+\frac{1}{2}}^2 p^d\left(r_{j+\frac{1}{2}}\right)\psi_{j+\frac{1}{2}}^-}
				-\dashuline{r_{j-\frac{1}{2}}^2 p^d\left(r_{j-\frac{1}{2}}\right)\psi_{j-\frac{1}{2}}^+}\nonumber\\
				&\quad=0,
			\end{align}
			where different underlines are used in the last equality to highlight the terms that cancel each other.
			Therefore, we can conclude that the semi-discrete scheme \eqref{eq:scheme} maintains the equilibrium state \eqref{eq:lane-emden1} exactly.
		\end{proof}
	
	\subsection{The well-balanced total-energy-conserving RKDG scheme}\label{sec:scheme-total-energy}

In this subsection, we present the approach to design a total-energy-conserving fully discrete DG method to ensure the scheme has the total energy conservation property \eqref{eq:total-energy} on the discrete level. This will involve two components: the approximation $s_j^{[3]}$ of the source term in the energy equation \eqref{eq:energy}, and the temporal discretization. To illustrate the idea, we will start with the semi-discrete method to explain the approximation $s_j^{[3]}$, followed by the forward Euler time discretization, and the high-order Runge--Kutta method at the end. 

\subsubsection{Semi-discrete total-energy-conserving method}
The key idea of designing the total-energy-conserving scheme is on the approximation of the source term in the energy equation \eqref{eq:energy}. Let us apply integration by parts on the source term approximation $s_j^{[3]}$ in \eqref{eq:source-term-standard}, which leads to
\begin{align}
	s_j^{[3]}=&\int_{K_j}-(\rho u)_h\,\frac{\partial\Phi_h}{\partial r}\,\delta\,r^2\mathrm{d}r\nonumber\\
	=&-\left((\rho u)_h\,\Phi_h\,\delta\,r^2\right)\Big|_{r_{j-\frac12}^+}^{r_{j+\frac12}^-}+\int_{K_j}\frac{\partial}{\partial r}\left((\rho u)_h\,r^2\right)\,\Phi_h\,\delta\mathrm{d}r\nonumber\\
	&+\int_{K_j}(\rho u)_h\,\Phi_h\,\frac{\partial \delta}{\partial  r}\,r^2\mathrm{d}r\nonumber\\
	\approx&-\left(\hat{f}^{*,[1]}\,\Phi_h\,\delta\,r^2\right)\Big|_{r_{j-\frac12}^+}^{r_{j+\frac12}^-}-\int_{K_j}\frac{\partial\rho_h}{\partial t}\,\Phi_h\,\delta\,r^2\mathrm{d}r\nonumber\\
	&+\int_{K_j}(\rho u)_h\,\Phi_h\,\frac{\partial \delta}{\partial  r}\,r^2\mathrm{d}r\nonumber\\
	:=&s_j^{[3],\text{tec}}\left(\boldsymbol{u}_h,\hat{\boldsymbol{f}}^{*},\frac{\partial\rho_h}{\partial t},\Phi_h,\delta\right),\label{eq:conservative-source-term-3}
\end{align}
where the superscript `{\rm tec}' stands for total-energy-conserving, $\delta$ is the test function and $\hat{f}^{*,[1]}$ is the first component of the numerical flux in \eqref{eq:flux1}. 	Equation \eqref{eq:mass} is used to replace $\frac{\partial}{\partial r}\left((\rho u)_h\,r^2\right)$ by $-r^{2}\partial \rho_h/\partial t$ (approximately). 

With this reformulation of the source term, we can now modify the semi-discrete well-balanced method \eqref{eq:scheme} slightly, and obtain the semi-discrete well-balanced and total-energy-conserving scheme: find $\boldsymbol{u}_h\in\boldsymbol{\Pi}_h$ such that for any test function $\boldsymbol{v}=(\zeta,\psi,\delta)^T\in\boldsymbol{\Pi}_h$, it holds that
\begin{align}
	&\partial_t\int_{K_j}\boldsymbol{u}_h\cdot\boldsymbol{v}\, r^2\mathrm{d}r\nonumber\\
	=&\mathcal{F}_j(\boldsymbol{u}_h,\boldsymbol{v})+\mathcal{S}^{[2],\text{wb}}_j(\boldsymbol{u}_h,\boldsymbol{v})+\mathcal{S}_j^{[3],\text{tec}}\left(\boldsymbol{u}_h,\hat{\boldsymbol{f}}^{*},\frac{\partial\rho_h}{\partial t},\Phi_h,\delta\right),\label{eq:semi-discrete-scheme}
\end{align}
where 
\begin{equation}
	\mathcal{S}^{[2],\text{wb}}_j=\left[0,s_j^{[2],\text{wb}},0\right]^T, \qquad
	\mathcal{S}^{[3],\text{tec}}_j=\left[0,0,s_j^{[3],\text{tec}}\right]^T.
	\label{eq:source-term-new}
\end{equation}

\begin{proposition}\label{prop-semi}
	For the semi-discrete scheme \eqref{eq:semi-discrete-scheme}, we have the following total energy conservation property
	\begin{align}
		&\frac{\partial}{\partial t}\int_{K_j}\left(E_h+\frac12\rho_h\,\Phi_h\right)r^2\,\mathrm{d}r+\Bigg(\hat{f}^{*,[3]}+\hat{f}^{*,[1]}\Phi_h\nonumber\\
		&\qquad\quad\left.-\frac{1}{8\pi\,G}\left(\Phi_h\frac{\partial}{\partial t}\left(\frac{\partial\Phi_h}{\partial r}\right)-\frac{\partial\Phi_h}{\partial t}\frac{\partial\Phi_h}{\partial r}\right)\Bigg)r^2\,\right|_{r_{j-\frac12}^+}^{r_{j+\frac12}^-}=0, 
	\end{align}
	which is consistent with the continuous result in Eq. \eqref{eq:total-energy-continuous}, and leads to the conservation of total energy $\int_\Omega (E_h+\frac12\rho_h\,\Phi_h)r^2\,\mathrm{d}r$.
\end{proposition}
\begin{proof}
	Following the approach used in the proof of \eqref{eq:total-energy-continuous}, we decompose the first term into two parts:
	\begin{equation}
		\frac{\partial}{\partial t}\int_{K_j}\left(E_h+\frac12\rho_h\,\Phi_h\right)r^2\,\mathrm{d}r=\text{\RNum{1}}+\text{\RNum{2}},
	\end{equation}
	with
	\begin{align}
		\text{\RNum{1}}&=\int_{K_j}\left(\frac{\partial E_h}{\partial t}+\frac{\partial\rho_h}{\partial t}\Phi_h\right)r^2\,\mathrm{d}r,\\
		\text{\RNum{2}}&=\int_{K_j}\frac12\left(\rho_h\frac{\partial\Phi_h}{\partial t}-\frac{\partial\rho_h}{\partial t}\Phi_h\right)r^2\,\mathrm{d}r.
	\end{align}
	We set the test function $\boldsymbol{v}$ as $(0,0,1)^T$ in \eqref{eq:semi-discrete-scheme} to obtain
	\begin{align}
		&\int_{K_j}\frac{\partial E_h}{\partial t}\,r^2\,\mathrm{d}r\nonumber\\
		&=-\left.\left(\hat{\boldsymbol{f}}^{*,[3]}+\hat{\boldsymbol{f}}^{*,[1]}\Phi_h\right)r^2\,\right|_{r_{j-\frac12}^+}^{r_{j+\frac12}^-}-\int_{K_j}\frac{\partial\rho_h}{\partial t}\Phi_h\,r^2\,\mathrm{d}r,
	\end{align}
	which leads to the simplification of part \RNum{1} as
	\begin{equation}
		\text{\RNum{1}}=-\left.\left(\hat{\boldsymbol{f}}^{*,[3]}+\hat{\boldsymbol{f}}^{*,[1]}\Phi_h\right)r^2\,\right|_{r_{j-\frac12}^+}^{r_{j+\frac12}^-}.
	\end{equation}
	Next, note that the evaluation of $\Phi_h$ in \eqref{eq:DG-poisson} and \eqref{eq:DG-poisson2} are exact, i.e., 
	\begin{align}\label{eq:prop-ef-1}
		4\pi\,G\,\rho_h\,r^2=\frac{\partial}{\partial r}\left(\,r^{2}\,\frac{\partial \Phi_h}{\partial r}\,\right),
	\end{align}
	therefore, following the exact same step in the proof of \eqref{eq:total-energy-continuous} in Section \ref{sec2.3}, we have
	\begin{align}
		\text{\RNum{2}}&
		=\frac{1}{8\pi\,G}\left.\left(\frac{\partial\Phi_h}{\partial t}\frac{\partial\Phi_h}{\partial r}-\Phi_h\frac{\partial}{\partial t}\left(\frac{\partial\Phi_h}{\partial r}\right)\right)r^2\,\right|_{r_{j-\frac12}^+}^{r_{j+\frac12}^-}.
	\end{align} 
	The combination of these two equations leads to the total energy conservation property, which finishes the proof.  
\end{proof}

\subsubsection{Forward Euler time discretization and total energy conservation}

The extension of the total energy conservation property in Proposition \ref{prop-semi} to fully discrete schemes coupled with high-order RK methods is a non-trivial task. Let us start with the simpler first-order Euler method, and use it as an example to illustrate how to obtain the fully discrete second- and third-order total-energy-conserving schemes. 

The straightforward application of the forward Euler method to the semi-discrete well-balanced and total-energy-conserving scheme \eqref{eq:semi-discrete-scheme} may not conserve the total energy automatically. The only term that needs extra care is the approximation of $\mathcal{S}_j^{[3],\text{tec}}$ in \eqref{eq:conservative-source-term-3}, \eqref{eq:source-term-new}, and
the fully discrete scheme with forward Euler discretization is given by
\begin{align}
	&\int_{K_j}\boldsymbol{u}_h^{n+1}\cdot\boldsymbol{v}\,r^2\mathrm{d}r\nonumber\\
	&\quad=\int_{K_j}\boldsymbol{u}_h^n\cdot\boldsymbol{v}\,r^2\mathrm{d}r+\Delta t\Bigg(\mathcal{F}_j(\boldsymbol{u}_h^n,\boldsymbol{v})+\mathcal{S}_j^{[2],\text{wb}}(\boldsymbol{u}_h^n,\boldsymbol{v})\nonumber\\
	&\qquad+\mathcal{S}_j^{[3],\text{tec}}\left(\boldsymbol{u}_h^n,\hat{\boldsymbol{f}}^{*,n},\frac{\rho^{n+1}-\rho^n}{\Delta t},\frac{\Phi_h^{n+1}+\Phi_h^n}{2},\delta\right)\Bigg).\label{eq:fully-scheme-ef}
\end{align}
Note that although the right-hand side of \eqref{eq:fully-scheme-ef} contains $\rho_h^{n+1}$ and $\Phi_h^{n+1}$, the proposed scheme is still an explicit scheme as outlined below. First we can use the density equation to explicitly evaluate $\rho_h^{n+1}$, and obtain $\Phi_h^{n+1}$ following \eqref{eq:integral-poisson}-\eqref{eq:integral-poisson2}. Next the momentum equation is solved to update $(\rho u)_h^{n+1}$. Finally, with the available $\rho_h^{n+1}$ and $\Phi_h^{n+1}$, we can solve the energy equation to compute $E_h^{n+1}$ explicitly.

\begin{proposition}
	The fully discrete forward Euler DG scheme \eqref{eq:fully-scheme-ef} conserves total energy:
	\begin{equation}\label{eq:total-energy-ef}
		\int_{\Omega}\left(E_h^{n+1}+\frac12\rho_h^{n+1}\,\Phi_h^{n+1}\right)r^2\mathrm{d}r=\int_{\Omega}\left(E_h^n+\frac12\rho_h^n\,\Phi_h^n\right)r^2\mathrm{d}r,
	\end{equation}
	with outer boundary conditions $\Phi_h^n(R)=\Phi_h^{n+1}(R)=0$ and $\hat{\boldsymbol{f}}_{N+\frac12}^{*,n,[3]}=0$.
\end{proposition}
\begin{proof}
	The main structure of the proof is similar to that of the semi-discrete method in Proposition \ref{prop-semi}, with more terms due to the temporal discretization. In each cell $K_j$, we take the difference of the total energy in \eqref{eq:total-energy-ef} and separate it into two parts:
	\begin{align}
		&\int_{K_j}\left(\frac12\rho_h^{n+1}\,\Phi_h^{n+1}-\frac12\rho_h^{n}\,\Phi_h^{n}\right)r^2\mathrm{d}r+\int_{K_j}\left(E_h^{n+1}-E_h^n\right)r^2\mathrm{d}r\nonumber\\
		&\qquad:=\text{\RNum{1}}+\text{\RNum{2}}\label{eq:prop-ef--1},
	\end{align}
	with
	\begin{align}
		&\text{\RNum{1}}=\int_{K_j}\frac12\left(\rho_h^{n+1}-\rho_h^n\right)\left(\Phi_h^{n+1}+\Phi_h^{n}\right)r^2\mathrm{d}r+\int_{K_j}\left(E_h^{n+1}-E_h^n\right)r^2\mathrm{d}r, \\
		&\text{\RNum{2}}=\int_{K_j}\frac12\left(-\rho_h^{n+1}\,\Phi_h^{n}+\rho_h^n\Phi_h^{n+1}\right)r^2\mathrm{d}r.
	\end{align}
	
	Let us introduce the notation:
	\begin{equation}
		\Phi_h^{n+\frac12}=\frac{\Phi_h^{n+1}+\Phi_h^n}{2}.
	\end{equation}
	We note that $\hat{\boldsymbol{f}}^{*,n}$, $\frac{\partial \Phi_h^n}{\partial r}$, and $\Phi_h^n$ are single-valued in our schemes. By setting the test function $\boldsymbol{v}=(0,0,1)^T$ in \eqref{eq:fully-scheme-ef}, we can derive 
	\begin{align}
		\int_{K_j}E_h^{n+1}\,r^2\mathrm{d}r 
		&=\int_{K_j}E_h^{n}\,r^2\mathrm{d}r
		-\int_{K_j}\left(\rho_h^{n+1}-\rho_h^n\right)\Phi_h^{n+\frac12}\,r^2\mathrm{d}r \notag \\
		&\quad
		-\Delta t\left.\left(r^2\left(\hat{f}^{*,n,[3]}+\hat{f}^{*,n,[1]}\Phi_h^{n+\frac12}\right)\right)\right|_{r_{j-\frac12}}^{r_{j+\frac12}}, 
	\end{align}
	where $\hat{f}^{*,[i]}$ is the $i$-th component of the numerical flux $\hat{\boldsymbol{f}}^*$.		
	Therefore, we can simplify the term \RNum{1} as 
	\begin{align}
		& \text{\RNum{1}}
		=-\Delta t\left.\left(r^2\left(\hat{\boldsymbol{f}}^{*,n,[3]}+\hat{\boldsymbol{f}}^{*,n,[1]}\Phi_h^{n+\frac12}\right)\right)\right|_{r_{j-\frac12}}^{r_{j+\frac12}}\label{eq:prop-ef-4}.
	\end{align}
	
	Following the equality \eqref{eq:prop-ef-1} in the evaluation of $\Phi_h$, we have
	\begin{align}
		&4\pi\,G\int_{K_j}\rho_h^{n+1}\,\Phi_h^n\,r^2\mathrm{d}r=\int_{K_j}\frac{\partial}{\partial r}\left(r^2\frac{\partial \Phi_h^{n+1}}{\partial r}\right)\Phi_h^n\mathrm{d}r\nonumber\\
		&\hskip1.2cm=\left.\left(r^2\frac{\partial \Phi_h^{n+1}}{\partial r}\Phi_h^n\right)\right|_{r_{j-\frac12}}^{r_{j+\frac12}}-\int_{K_j}\frac{\partial \Phi_h^{n+1}}{\partial r}\frac{\partial \Phi_h^{n}}{\partial r}\,r^2\mathrm{d}r, \\
		&4\pi\,G\,\int_{K_j}\rho_h^{n}\,\Phi_h^{n+1}\,r^2\mathrm{d}r=\int_{K_j}\frac{\partial}{\partial r}\left(r^2\frac{\partial \Phi_h^{n}}{\partial r}\right)\Phi_h^{n+1}\mathrm{d}r\nonumber\\
		&\hskip1.2cm=\left.\left(r^2\frac{\partial \Phi_h^{n}}{\partial r}\Phi_h^{n+1}\right)\right|_{r_{j-\frac12}}^{r_{j+\frac12}}-\int_{K_j}\frac{\partial \Phi_h^{n}}{\partial r}\frac{\partial \Phi_h^{n+1}}{\partial r}\,r^2\mathrm{d}r.
	\end{align}
	Therefore, we can simplify term \RNum{2} as 
	\begin{equation}\label{eq:prop-ef-3}
		\text{\RNum{2}}
		=\frac{1}{8\pi\,G}\left.\left(r^2\frac{\partial \Phi_h^{n}}{\partial r}\Phi_h^{n+1}
		-r^2\frac{\partial \Phi_h^{n+1}}{\partial r}\Phi_h^n\right)\right|_{r_{j-\frac12}}^{r_{j+\frac12}}.
	\end{equation}
	
	We combine Eqs. \eqref{eq:prop-ef--1}-\eqref{eq:prop-ef-4} and sum over all the cells $K_j$ to obtain
	\begin{align}
		&\int_{\Omega}\left(E_h^{n+1}+\frac12\rho_h^{n+1}\,\Phi_h^{n+1}\right)r^2\mathrm{d}r-\int_{\Omega}\left(E_h^{n}+\frac12\rho_h^{n}\,\Phi_h^{n}\right)r^2\mathrm{d}r\nonumber\\
		&\qquad=\sum_{j=1}^{N} \frac{1}{8\pi\,G}\left.\left(r^2\frac{\partial \Phi_h^{n}}{\partial r}\Phi_h^{n+1}
		-r^2\frac{\partial \Phi_h^{n+1}}{\partial r}\Phi_h^n\right)\right|_{r_{j-\frac12}}^{r_{j+\frac12}}\nonumber\\
		&\qquad\quad-\Delta t\left.\left(r^2\left(\hat{\boldsymbol{f}}^{*,n,[3]}+\hat{\boldsymbol{f}}^{*,n,[1]}\Phi_h^{n+\frac12}\right)\right)\right|_{r_{j-\frac12}}^{r_{j+\frac12}}\nonumber\\
		&\qquad=\left.\frac{1}{8\pi\,G}\left(r^2\frac{\partial \Phi_h^{n}}{\partial r}\Phi_h^{n+1}-r^2\frac{\partial \Phi_h^{n+1}}{\partial r}\Phi_h^n\right)\right|_{0}^{R}\nonumber\\
		&\qquad\quad\left.-\Delta t\left(r^2\left(\hat{\boldsymbol{f}}^{*,n,[3]}+\hat{\boldsymbol{f}}^{*,n,[1]}\Phi_h^{n+\frac12}\right)\right)\right|_{0}^{R}\nonumber\\
		&\qquad=~0,
	\end{align}
	where the last equality is due to the outer boundary condition $\Phi_h^n(R)=\Phi_h^{n+1}(R)=\Phi_h^{n+\frac12}(R)=0$ and $\hat{f}_{N+\frac12}^{*,n,[3]}=0$. Therefore, the fully discrete forward Euler DG scheme \eqref{eq:fully-scheme-ef} has the total energy conservation property.  
\end{proof}

\begin{remark}	
	The assumptions on the outer boundary condition (i.e., $\Phi_h^n(R)=\Phi_h^{n+1}(R)=0$ and $\hat{\boldsymbol{f}}_{N+\frac12}^{*,n,[3]}=0$) are only used in the last equality of the proof. We use these assumptions for ease of presentation. The total energy conservation property of our numerical methods does not depend on these assumptions. In Section \ref{exam_yahil}, we consider a numerical example without the assumption $\hat{\boldsymbol{f}}_{N+\frac12}^{*,n,[3]}=0$, and observe conservation of total energy, after adding correction terms due to the outer boundary.  We can deal with the case without the assumption $\Phi_h^n(R)=\Phi_h^{n+1}(R)=0$ in a similar way by adding correction term. We refer to Section \ref{exam_yahil} for the details on these correction terms and the numerical observation.
\end{remark}

\begin{remark}
	We note that our proposed scheme \eqref{eq:fully-scheme-ef} still has the well-balanced property. The only thing to check is that the source term approximation $\mathcal{S}_j^{[3],\text{\rm tec}}=0$ holds at the steady state. This holds due to the fact that $\hat{f}^{*,n,[1]}=0$, $u_h^n=0$, and also $\rho_h^n=\rho_h^{n+1}$ by updating the density equation with the well-balanced DG method at the steady state.  
\end{remark}

\subsubsection{High-order Runge-Kutta time discretization}\label{sec:full-high-conserve}

In this section, we will extend the well-balanced and total-energy-conserving method \eqref{eq:fully-scheme-ef} coupled with forward Euler discretization to high-order RK discretization. In \cite{mullen2021extension}, the fully discrete energy conserving schemes with second- and third-order RK time discretization are introduced in the context of finite difference methods. The key idea is to use different source term approximations for each 
stage of the Runge--Kutta method, and a similar idea will be explored here. Comparing with the RK methods in \cite{mullen2021extension} and this paper, the main difference is that we involve additional terms, such as the approximation of $\frac{\partial \rho}{\partial t}$. This is because our DG schemes include test functions and the relationship between the variables $\boldsymbol{u}$ is more complicated.

Let us start with the second-order RK method. For the differential equation of the general form
$w_t = \mathcal{L}(w)$, a second-order RK method can be formulated as	
\begin{align} 
	w^{(1)} &= w^n + \Delta t\,\mathcal{L}(w^n), \notag \\
	w^{n+1}&=w^n + \frac12\left( w^{(1)} +\Delta t\,\mathcal{L}(w^{(1)})\right)\nonumber\\
	&=w^n + \Delta t\left(\frac{\mathcal{L}(w^n)+\mathcal{L}(w^{(1)})}{2}\right). \label{eq:2ndRK}
\end{align}
Starting from the forward Euler method \eqref{eq:fully-scheme-ef}, the fully discrete total-energy conserving scheme with second-order RK method \eqref{eq:2ndRK} is given by 
\begin{align}
	&\int_{K_j}\boldsymbol{u}_h^{(1)}\cdot\boldsymbol{v}\,r^2\mathrm{d}r\nonumber\\
	&\qquad=\int_{K_j}\boldsymbol{u}_h^n\cdot\boldsymbol{v}\,r^2\mathrm{d}r+\Delta t\Bigg(\mathcal{F}_j(\boldsymbol{u}_h^n,\boldsymbol{v})+\mathcal{S}_j^{[2],\text{wb}}(\boldsymbol{u}_h^n,\boldsymbol{v})\nonumber\\
	&\qquad\quad+\mathcal{S}_j^{[3],\text{tec}}\left(\boldsymbol{u}_h^n,\hat{\boldsymbol{f}}^{*,n},\frac{\rho^{(1)}-\rho^n}{\Delta t},\Phi_h^{(0,1)},\delta\right)\Bigg),\label{eq:fully-scheme-2nd-0}\\
	&\int_{K_j}\boldsymbol{u}_h^{n+1}\cdot\boldsymbol{v}\,r^2\mathrm{d}r\nonumber\\
	&\qquad=\int_{K_j}\boldsymbol{u}_h^n\cdot\boldsymbol{v}\,r^2\mathrm{d}r+\Delta t\Bigg(\frac{\mathcal{F}_j(\boldsymbol{u}_h^n,\boldsymbol{v})+\mathcal{F}_j(\boldsymbol{u}_h^{(1)},\boldsymbol{v})}{2}\nonumber\\
	&\qquad\quad+\frac{\mathcal{S}_j^{[2],\text{wb}}(\boldsymbol{u}_h^n,\boldsymbol{v})+\mathcal{S}_j^{[2],\text{wb}}(\boldsymbol{u}_h^{(1)},\boldsymbol{v})}{2}\nonumber\\
	&\qquad\quad+\mathcal{S}_j^{[3],\text{tec}}\left(\boldsymbol{u}_h^{(0,1)},\hat{\boldsymbol{f}}^{*,(0,1)},\frac{\rho^{n+1}-\rho^n}{\Delta t},\Phi_h^{(0,2)},\delta\right)\Bigg),\label{eq:fully-scheme-2nd}
\end{align}
where we introduced the following notations 
\begin{align}
	&\hat{\boldsymbol{f}}^{*,(0,1)}=\frac12\left(\hat{\boldsymbol{f}}^{*,n}+\hat{\boldsymbol{f}}^{*,(1)}\right), \qquad 
	\boldsymbol{u}^{(0,1)}=\frac12\left(\boldsymbol{u}_h^n+\boldsymbol{u}_h^{(1)}\right),\nonumber \\
	&\Phi_h^{(0,1)}=\frac12\left(\Phi_h^n+\Phi_h^{(1)}\right), \qquad \Phi_h^{(0,2)}=\frac12\left(\Phi_h^n+\Phi_h^{n+1}\right).
\end{align}

The third-order strong-stability-preserving RK method for $w_t = \mathcal{L}(w)$ can be formulated as
\begin{align}
	w^{(1)} &= w^n + \Delta t\,\mathcal{L}(w^n),\nonumber\\
	w^{(2)} &= \frac34 w^n + \frac14\left(w^{(1)} + \Delta t\,\mathcal{L}(w^{(1)})\right)\nonumber\\
	&=w^n + \frac{\Delta t}{2}\left(\frac{\mathcal{L}(w^n)+\mathcal{L}(w^{(1)})}{2}\right),\nonumber\\
	w^{n+1} &=\frac13 w^n + \frac23\left(w^{(2)}+\Delta t\,\mathcal{L}(w^{(2)})\right)\nonumber\\
	&=w^n+\Delta t\left(\frac{\mathcal{L}(w^n)+\mathcal{L}(w^{(1)})+4\mathcal{L}(w^{(2)})}{6}\right).\label{eq:standard-rk3}
\end{align}
The fully discrete total-energy conserving scheme with this third-order RK method is given by 
\begin{align}
	&\int_{K_j}\boldsymbol{u}_h^{(1)}\cdot\boldsymbol{v}\,r^2\mathrm{d}r=\int_{K_j}\boldsymbol{u}_h^n\cdot\boldsymbol{v}\,r^2\mathrm{d}r\nonumber\\
	&\qquad\quad+\Delta t\Bigg(\mathcal{F}_j(\boldsymbol{u}_h^n,\boldsymbol{v})+\mathcal{S}_j^{[2],\text{wb}}(\boldsymbol{u}_h^n,\boldsymbol{v})\nonumber\\
	&\qquad\quad+\mathcal{S}_j^{[3],\text{tec}}\left(\boldsymbol{u}_h^n,\hat{\boldsymbol{f}}^{*,n},\frac{\rho^{(1)}-\rho^n}{\Delta t},\Phi_h^{n+\frac12},\delta\right)\Bigg),\label{eq:fully-scheme-3rd-0}\\
	&\int_{K_j}\boldsymbol{u}_h^{(2)}\cdot\boldsymbol{v}\,r^2\mathrm{d}r=\int_{K_j}\boldsymbol{u}_h^n\cdot\boldsymbol{v}\,r^2\mathrm{d}r\nonumber\\
	&\qquad\quad+\frac{\Delta t}{2}\Bigg(\frac{\mathcal{F}_j(\boldsymbol{u}_h^n,\boldsymbol{v})+\mathcal{F}_j(\boldsymbol{u}_h^{(1)},\boldsymbol{v})}{2}\nonumber\\
	&\qquad\quad+\frac{\mathcal{S}_j^{[2],\text{wb}}(\boldsymbol{u}_h^n,\boldsymbol{v})+\mathcal{S}_j^{[2],\text{wb}}(\boldsymbol{u}_h^{(1)},\boldsymbol{v})}{2}\nonumber\\
	&\qquad\quad+\mathcal{S}_j^{[3],\text{tec}}\left(\boldsymbol{u}_h^{(0,1)},\hat{\boldsymbol{f}}^{*,(0,1)},\frac{\rho^{(2)}-\rho^n}{\Delta t/2},\Phi_h^{(0,2)},\delta\right)\Bigg),\\
	&\int_{K_j}\boldsymbol{u}_h^{n+1}\cdot\boldsymbol{v}\,r^2\mathrm{d}r=\int_{K_j}\boldsymbol{u}_h^n\cdot\boldsymbol{v}\,r^2\mathrm{d}r\nonumber\\
	&\qquad\quad+\Delta t\Bigg(\frac{\mathcal{F}_j(\boldsymbol{u}_h^n,\boldsymbol{v})+\mathcal{F}_j(\boldsymbol{u}_h^{(1)},\boldsymbol{v})+4\mathcal{F}_j(\boldsymbol{u}_h^{(2)},\boldsymbol{v})}{6}\nonumber\\
	&\qquad\quad+\frac{\mathcal{S}_j^{[2],\text{wb}}(\boldsymbol{u}_h^n,\boldsymbol{v})+\mathcal{S}_j^{[2],\text{wb}}(\boldsymbol{u}_h^{(1)},\boldsymbol{v})+4\mathcal{S}_j^{[2],\text{wb}}(\boldsymbol{u}_h^{(2)},\boldsymbol{v})}{6}\nonumber\\
	&\qquad\quad+\mathcal{S}_j^{[3],\text{tec}}\left(\boldsymbol{u}_h^{(0,2)},\hat{\boldsymbol{f}}^{*,(0,2)},\frac{\rho^{n+1}-\rho^n}{\Delta t},\Phi_h^{(0,3)},\delta\right)\Bigg),\label{eq:fully-scheme-3rd}
\end{align}
with the following notations 
\begin{align}
	&\hat{\boldsymbol{f}}^{*,(0,2)}=\frac16\left(\hat{\boldsymbol{f}}^{*,n}+\hat{\boldsymbol{f}}^{*,(1)}+4\hat{\boldsymbol{f}}^{*,(2)}\right), \\
	&\boldsymbol{u}^{(0,2)}=\frac16\left(\boldsymbol{u}_h^n+\boldsymbol{u}_h^{(1)}+4\boldsymbol{u}_h^{(2)}\right), \qquad 
	\Phi_h^{(0,3)}=\frac12\left(\Phi_h^n+\Phi_h^{n+1}\right).\nonumber
\end{align}

Note that different source term approximations of $\mathcal{S}_j^{[3],\text{tec}}$ are employed in the each 
stage of the RK method, in order to simultaneously achieve the total energy conservation property and high-order accuracy. The proofs of the well-balanced property and total energy conservation of the high-order RKDG methods \eqref{eq:fully-scheme-2nd-0}-\eqref{eq:fully-scheme-2nd} and \eqref{eq:fully-scheme-3rd-0}-\eqref{eq:fully-scheme-3rd} follow the exact same approach as that of the forward Euler DG scheme \eqref{eq:fully-scheme-ef}, and is omitted here to save space. 

\subsection{TVB limiter}

For problems containing strong discontinuities, oscillations may develop in the solutions obtained with DG methods, and in this case nonlinear limiters are needed after each 
stage of the Runge--Kutta methods to control these oscillations. One popular choice is the total variation bounded (TVB) limiter \cite{cockburn1989tvb}. Its extension to the system in spherically symmetrical coordinates has been considered in \cite{pochik2021thornado}, and will be employed here, provided some modifications to ensure the total-energy-conserving property. 

We start by defining two different cell averages of $\boldsymbol{u}_h$ in cell $K_j$: the standard and weighted cell averages given by
\begin{equation}\label{eq:tvd0}
	\bar{\boldsymbol{u}}_j=\frac{\int_{K_j}\boldsymbol{u}_h\,\mathrm{d}r}{\int_{K_j}1\,\mathrm{d}r}, \qquad\qquad 
	\tilde{\boldsymbol{u}}_j=\frac{\int_{K_j}\boldsymbol{u}_h\,r^2\,\mathrm{d}r}{\int_{K_j}r^2\,\mathrm{d}r},
\end{equation}
respectively. 
In cell $K_j$, the forward and backward slopes are defined as
\begin{align}
	\Delta\boldsymbol{u}_j^F=\frac{\bar{\boldsymbol{u}}_{j+1}-\bar{\boldsymbol{u}}_{j}}{r_{j+1}-r_j},&\qquad\qquad
	\Delta\boldsymbol{u}_{j}^B=\frac{\bar{\boldsymbol{u}}_{j}-\bar{\boldsymbol{u}}_{j-1}}{r_{j}-r_{j-1}},
\end{align}
where $r_j=(r_{j+\frac12}+r_{j-\frac12})/2$ denotes the midpoint of $K_j$. Then we apply the minmod function in \cite{cockburn1989tvb} to obtain
\begin{align}\label{eq:tvd0.5}
	&\tilde{\Delta}\boldsymbol{u}_{j}=\text{minmod}\left(\Delta\boldsymbol{u}_j,~\beta\Delta\boldsymbol{u}_{j}^F,~\beta\Delta\boldsymbol{u}_{j}^B\right),
\end{align}
where 
\begin{equation}
	\Delta\boldsymbol{u}_j=\frac{\boldsymbol{u}_{h,j+\frac12}^--\boldsymbol{u}_{h,j-\frac12}^+}{r_{j+\frac12}-r_{j-\frac12}},
\end{equation}
with $\beta$ being a constant to be specified. In \cite{pochik2021thornado}, it was shown that $\beta=1.75$ yields good results for a range of problems, and this value will also be used in this paper. If $\tilde{\Delta}\boldsymbol{u}_{j}$ and $\Delta\boldsymbol{u}_j$ are the same, this indicates that a limiter is not needed in this cell. When they are different, we mark this cell $K_j$ as a troubled cell. In such cell, we define a new linear polynomial $\tilde{\boldsymbol{u}}_{h,j}$ as
\begin{equation}\label{eq:tvd1}
	\tilde{\boldsymbol{u}}_{h,j}=\tilde{\boldsymbol{u}}_{j}^0+\tilde{\Delta}\boldsymbol{u}_{j}(r-r_j),\qquad 
	\tilde{\boldsymbol{u}}_{j}^0=\tilde{\boldsymbol{u}}_j-\tilde{\Delta}\boldsymbol{u}_{j}\frac{\int_{K_j}(r-r_j)\,r^2\,\mathrm{d}r}{\int_{K_j}r^2\,\mathrm{d}r},
\end{equation}
which has the updated slope $\tilde{\Delta}\boldsymbol{u}_{j}$ while keeping the same weighted cell average as $\tilde{\boldsymbol{u}}_j$. In the cells which are not marked as troubled cells, we simply set $\tilde{\boldsymbol{u}}_{h,j}=\boldsymbol{u}_{h,j}$. 
Finally, we replace the solution $\boldsymbol{u}_h$ by the updated solution $\tilde{\boldsymbol{u}}_{h}$ and continue the computation with the updated solution. This finishes the TVB limiter procedure. One can easily verify that the weighted cell average of $\tilde{\boldsymbol{u}}_{h,j}$ are the same as $\boldsymbol{u}_h$ in each computational cell, which yields the mass conservation property of the limiter procedure. 

Since the total energy depends nonlinearly on the variable $\rho_h$, this TVB limiter may destroy the total energy conservation property, which is satisfied by the proposed fully discrete method. To ensure the total-energy-conserving property, we slightly modify the TVB limiter on the variable $E_h$ as outlined below. Since the Euler--Poisson system does not conserve the non-gravitational energy $E$ in the PDE level, we propose an additional correction of $\tilde{E}_{h,j}$ as follows
\begin{equation}\label{eq:tvd2}
	\tilde{\tilde{E}}_{h,j}=\tilde{E}_{h,j}+\frac{\int_{K_j}\frac12(\rho_h\phi_h-\tilde{\rho}_h\tilde{\Phi}_h)\,r^2\,\mathrm{d}r}{\int_{K_j}r^2\,\mathrm{d}r},
\end{equation}
to ensure that the total energy $\int_{K_j} (E_h+\frac12\rho_h\Phi_h )~dr$ is not changed by the limiting procedure. 
Here $\tilde{\tilde{E}}_{h,j}$ is the updated numerical solution of $E$, $\tilde{E}_{h,j}$ is obtained in \eqref{eq:tvd1}, $\rho_h$ is the numerical solution before limiting, $\tilde{\rho}_h$ is the numerical solution after limiting, $\Phi_h$ and $\tilde{\Phi}_h$ are the gravitational potential calculated from $\rho_h$ and $\tilde{\rho}_h$ respectively. Note that $\tilde{\Phi}_h$ is evaluated after $\tilde{\rho}_h$ is available in all the cells, hence even though a cell $K_j$ is not marked as troubled cell, the value of $\tilde{\Phi}_h$ in this cell may be different from the original $\Phi_h$ due to modified $\tilde{\rho}_h$ in troubled cells in other locations. Therefore, this correction \eqref{eq:tvd2} will be applied for every cell regardless of being marked as troubled cells or not. 


The procedure of applying TVB limiter in each 
stage of Runge-Kutta method is summarized below, where the forward Euler time discretization is used for ease of presentation.
\begin{enumerate}
	\item At each time level $t^n$ (or 
	every intermediate stage of Runge-Kutta method), compute $\rho_h^{n+1}, (\rho u)_h^{n+1}$ for all cells $K_j$;
	\item Apply the TVB limiter to obtain $\tilde{\rho}_h^{n+1}, \widetilde{\rho u}^{n+1}$;
	\item Evaluate $\tilde{\Phi}_h^{n+1}$ based on the limited $\tilde{\rho}_h^{n+1}$;
	\item Compute $E_h^{n+1}$ (which employs the limited $\tilde{\rho}_h^{n+1}$ and $\tilde{\Phi}_h^{n+1}$) and apply TVB limiter with total-energy-conserving correction to $\tilde{\tilde{E}}_h^{n+1}$ (which involves both $\rho_h^{n+1}$, $\Phi_h^{n+1}$ and $\tilde{\rho}_h^{n+1}$, $\tilde{\Phi}_h^{n+1}$).
\end{enumerate}

\begin{remark}
	For the purpose of the well-balanced property, we use $\boldsymbol{u}_h-\boldsymbol{u}_h^e$ instead of $\boldsymbol{u}_h$ as an indicator to identify the troubled cells \cite{xing2014exactly}. If a cell is marked as a troubled cell, the update procedure is still applied on $\boldsymbol{u}_h$ as mentioned above. In the steady state, we have $\boldsymbol{u}_h-\boldsymbol{u}_h^e=0$, hence the TVB limiter will not take effect, and the well-balanced property will not be affected by the limiter. 
\end{remark}

\section{Numerical examples}\label{sec:example}


In this section, numerical examples will be provided to verify the properties of our proposed scheme, including the well-balanced property, total energy conservation properties and high-order accuracy. We use $P^2$ piecewise polynomial in the DG method and the third-order RK method \eqref{eq:fully-scheme-3rd-0}-\eqref{eq:fully-scheme-3rd} in the numerical tests, unless otherwise stated. The CFL number is set as 0.16 to determine the time step size. 

\subsection{Well-balanced and small perturbation tests}
\label{exam_well-balanced}

In this example, we consider a simple polytropic equilibrium and verify that our proposed scheme has the well-balanced property to maintain this equilibrium up to round-off error. We set $G=1/(4\pi)$ in this example, and choose two cases, $\gamma=2$ and $\gamma=1.2$, along with $\rho_0=1$ and $\kappa=1$. We have the following initial data
\begin{equation}
	\rho(r,0)=\frac{\sqrt{2}\sin(\frac{r}{\sqrt{2}})}{r},\quad \rho u(r,0)=0,\quad p(r,0)=\frac{2\sin^2(\frac{r}{\sqrt{2}})}{r^2},
\end{equation}
if $\gamma=2$, and 
\begin{equation}
	\rho(r,0)=(1+\frac{1}{18}r^2)^{-2.5},\quad \rho u(r,0)=0,\quad p(r,0)=(1+\frac{1}{18}r^2)^{-3},
\end{equation}
if $\gamma=1.2$, on the domain $\Omega=[0,1]$. The reflecting boundary condition is considered for the inner boundary and we set $\boldsymbol{u}^+(1)=\boldsymbol{u}^-(1)$ at the outer boundary. 
We set the stopping time $t = 4$ on the mesh with 200 uniform cells, and present the $L^1$ errors of the numerical solutions in Table \ref{table:1}, where both single and double precisions have been considered in the simulation. We can see that errors stay at the level of round-off errors for different precision, which verify the desired well-balanced property.
\begin{table}[h]
	\centering
	\caption{Example \ref{exam_well-balanced}, $L^1$ error of the numerical solutions for different precision in the well-balanced test.}
\begin{tabular}{c c c c c}
	\toprule
	Case & Precision & $\rho$ & $\rho u$ & $E$\\
	\midrule
	\multirow{2}{*}{$\gamma=2$} & double & 3.89E-13 & 2.70E-15 & 6.52E-14 \\
								& quad & 3.55E-31 & 3.44E-33 & 5.94E-32\\
	\midrule
	\multirow{2}{*}{$\gamma=1.2$} & double & 6.75E-13 & 8.00E-15 & 6.31E-13 \\
								  & quad & 6.04E-31 & 8.00E-33 & 5.74E-31 \\
	\bottomrule
\end{tabular}
	
	\label{table:1}
\end{table}

%

Next, we show the advantage of our proposed scheme in capturing a small perturbation to the equilibrium state. The initial data is given by imposing a pressure perturbation to the $\gamma=2$ equilibrium 
\begin{align}
	&\rho(r,0)=\frac{\sqrt{2}\sin(\frac{r}{\sqrt{2}})}{r},\quad \rho u(r,0)=0,\nonumber\\
	&p(r,0)=\frac{2\sin^2(\frac{r}{\sqrt{2}})}{r^2}+A\exp(-100r^2),\label{eq:exam1}
\end{align}
on the domain $\Omega=[0,0.5]$. The pressure is perturbed by a Gaussian bump of amplitude $A=10^{-6}$ in this test. We compute the solutions until $t=0.2$. A reference solution is computed with $N=400$ for comparison. We plot the velocity and pressure perturbation for $N=100$ in Figure \ref{fig:small-perturbation}, compared with the numerical solution of the non-well-balanced DG scheme from Section \ref{sec:convention}, and the reference solution. From the figures, we can see that the well-balanced scheme resolves the perturbation much better on a relatively coarse mesh. Similar test under the framework of finite difference methods in three dimensions can also be found in \cite{kappeli2014well}.

\begin{figure*}
	\centering
	\begin{subfigure}[pressure perturbation of wb]{\label{fig:wb-p}	
			\includegraphics[width = .45\linewidth]{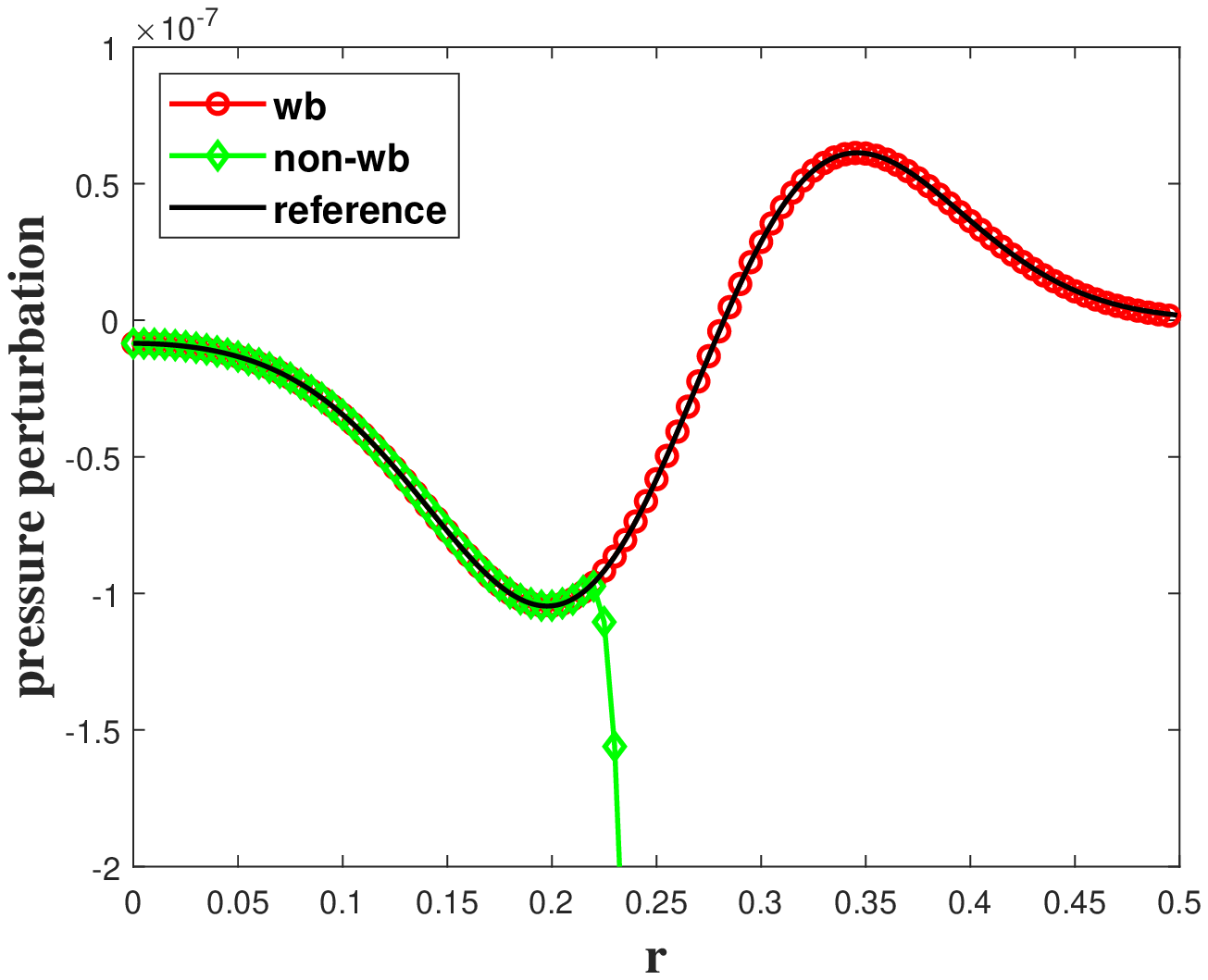}}
	\end{subfigure}
	\begin{subfigure}[pressure perturbation of non-wb]{\label{fig:non-wb-p}	
			\includegraphics[width = .45\linewidth]{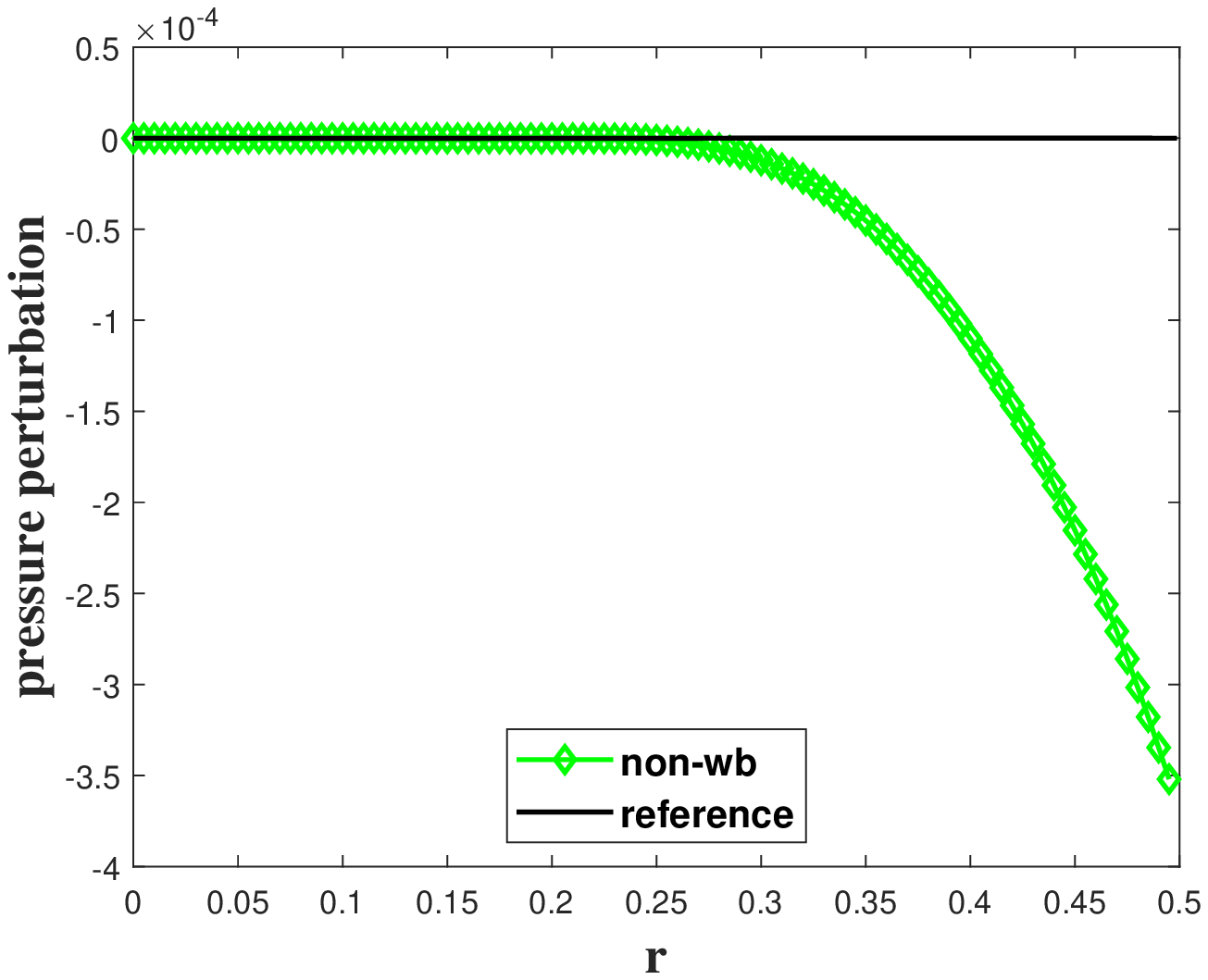}}
	\end{subfigure} \\
	\begin{subfigure}[velocity of wb]{\label{fig:wb-u}	
			\includegraphics[width = .45\linewidth]{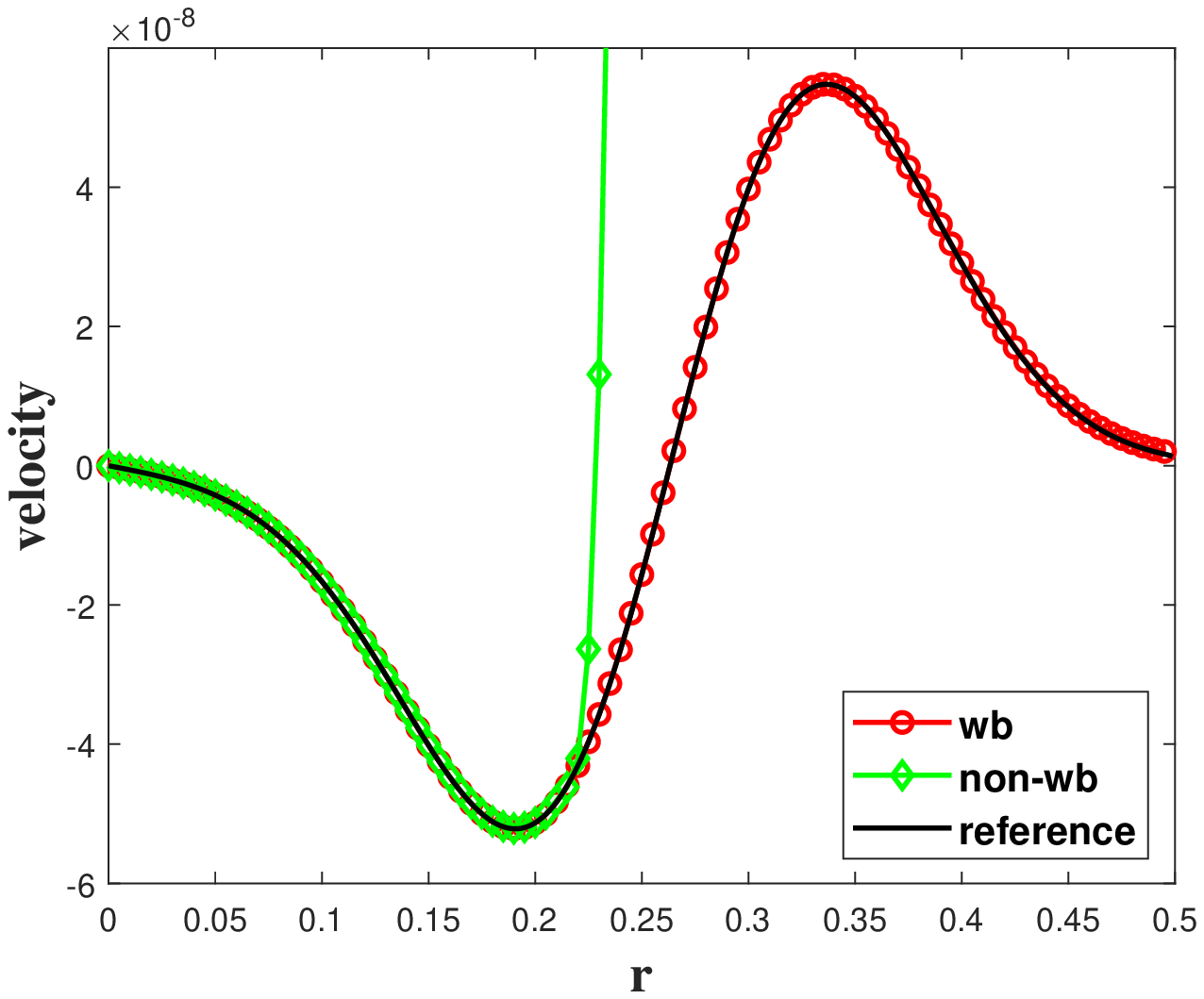}}
	\end{subfigure}
	\begin{subfigure}[velocity of non-wb]{\label{fig:non-wb-u}	
			\includegraphics[width = .45\linewidth]{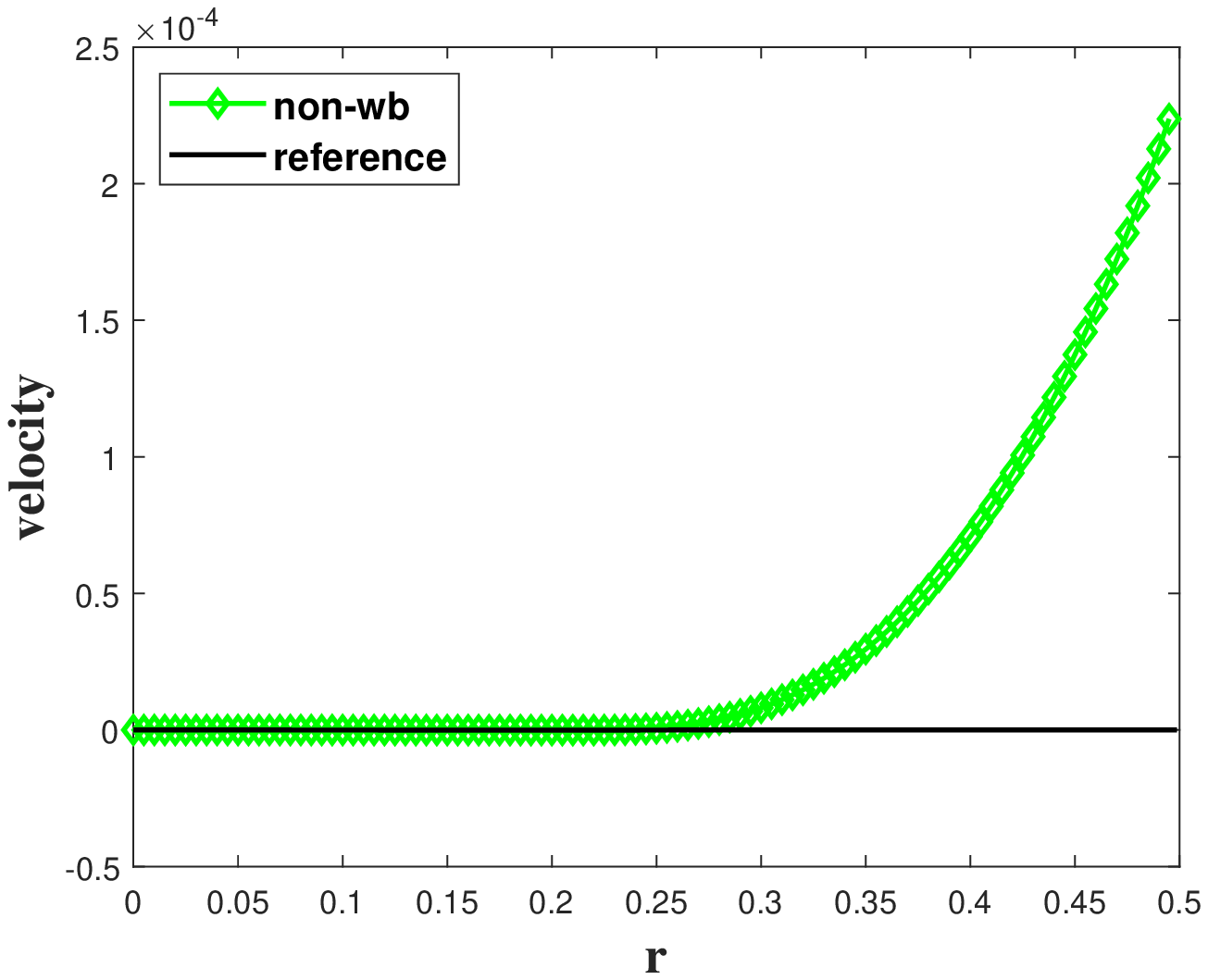}}
	\end{subfigure}
	\caption{Example \ref{exam_well-balanced}, numerical results at time $t=0.2$ for the small perturbation test. ``wb'' denotes the proposed DG scheme and ``non-wb'' denotes the standard DG scheme. The wb result is compared with non-wb result and the reference solution.}
	\label{fig:small-perturbation}
\end{figure*}

\subsection{Accuracy test}
\label{exam_accuracy}

\begin{enumerate}
    \item The accuracy test near the equilibrium state.
\end{enumerate}
In this example, we test the accuracy of the numerical solution near the equilibrium state and use the same initial condition in \eqref{eq:exam1} with parameter $A=0.001$. We set the domain $\Omega=[0,0.5]$, polynomial degree $k=2$ and stopping time $t=0.2$, same as those in Section \ref{exam_well-balanced}. Since the exact solution is unknown, we use the numerical solution of $N=640$ as a reference solution. The error table are shown in Table \ref{table:0}. We can observe the optimal convergence rate for all the variables. In addition, we also list the errors of the standard DG scheme \eqref{eq:scheme0} in Table \ref{table:00} for comparison. We observe that although both schemes have the optimal convergence order, the errors of our proposed scheme are much smaller than those of the standard scheme.

\begin{table}
	\caption{Example \ref{exam_accuracy}, accuracy test near the equilibrium state for $k=2$ with our proposed third-order RKDG scheme \eqref{eq:fully-scheme-3rd-0}-\eqref{eq:fully-scheme-3rd}.}
\small
\begin{tabular}{c c c c c c c}
	\toprule
	$N$ & \multicolumn{2}{c}{$\rho$} & \multicolumn{2}{c}{$\rho u$} & \multicolumn{2}{c}{$E$}\\
	\midrule
	10 & 2.62E-07 & - & 1.63E-07 & - & 2.23E-07 & - \\
	20 & 3.09E-08 & 3.08 & 1.71E-08 & 3.25 & 2.41E-08 & 3.21 \\
	40 & 3.73E-09 & 3.05 & 2.16E-09 & 2.98 & 3.08E-09 & 2.97 \\
	80 & 4.48E-10 & 3.06 & 2.97E-10 & 2.86 & 4.24E-10 & 2.86 \\
   	\bottomrule
\end{tabular}
	\label{table:0}
\end{table}

\begin{table}
	\caption{Example \ref{exam_accuracy}, accuracy test near the equilibrium state for $k=2$ with the standard DG scheme \eqref{eq:scheme0} and third-order RKDG time discretization \eqref{eq:standard-rk3}}
\small
\begin{tabular}{c c c c c c c}
	\toprule
	$N$ & \multicolumn{2}{c}{$\rho$} & \multicolumn{2}{c}{$\rho u$} & \multicolumn{2}{c}{$E$}\\
	\midrule
	10 & 1.84E-04 & - & 1.48E-04 & - & 2.19E-04 & - \\
	20 & 2.62E-05 & 2.81 & 2.03E-05 & 2.87 & 2.16E-05 & 3.34 \\
	40 & 3.35E-06 & 2.97 & 2.56E-06 & 2.99 & 3.96E-06 & 2.87 \\
	80 & 4.34E-07 & 2.95 & 3.33E-07 & 2.94 & 4.25E-07 & 2.80 \\
   	\bottomrule
\end{tabular}
	\label{table:00}
\end{table}

\begin{enumerate}
    \item[(ii)] The accuracy test far away from the equilibrium state.
\end{enumerate}
In this example, we provide an accuracy test for solutions far away from the equilibrium state, to test the high-order convergence rate of the DG methods. We consider the following ``manufactured'' exact solutions 
\begin{equation}
\rho(r,t)=\frac{\exp(t-r)}{r^2},\quad u(r,t)=1,\quad p(r,t)=\frac{1}{r^2}.
\end{equation}
As a result, the Euler--Poisson equations \eqref{eq:problem} becomes
\begin{equation}
	\label{eq:change}
	\frac{\partial\boldsymbol{u}}{\partial t}+\frac{1}{r^2}\frac{\partial}{\partial r}(r^2\boldsymbol{f}(\boldsymbol{u}))=\boldsymbol{s}(\boldsymbol{u},\Phi)+\boldsymbol{w}(r),
\end{equation}
with an additional source term $\boldsymbol{w}(r)$ given by
\begin{equation}
	\boldsymbol{w}(r)=\left(0,-\frac{\exp(2(t-r))+2r}{r^4},-\frac{\exp(2(t-r))}{r^4}\right)^T.
\end{equation}
In this test, we set $\gamma=2$, $G=1/(4\pi)$, the computational domain is $\Omega=[0.5,1]$, and the stopping time is set to $t=0.1$. The exact solution is used to provide the boundary condition for the Euler equations, and the boundary condition for the Poisson equation is set as
\begin{equation}
	\frac{\partial \Phi_h}{\partial r}(0.5)=-4\exp(t-0.5),\quad\Phi_h(0.5)=0.
\end{equation}
Since our computational domain does not contain the origin $r=0$, our approach of recovering the reference equilibrium state $\boldsymbol{u}^d$ needs an additional boundary condition instead of \eqref{eq:lane-emden-boundary2}. For simplicity, we skip the steps of recovering the reference state in Section \ref{sec:recovery} and set a global steady state $\boldsymbol{u}^d$ explicitly for all cells without using \eqref{eq:reference_operator}:
\begin{equation}
	\rho^d(r)=\frac{\sqrt{2}\sin(\frac{r}{\sqrt{2}})}{r},\quad u^d(r)=0,\quad p^d(r)=\frac{2\sin^2(\frac{r}{\sqrt{2}})}{r^2}.
\end{equation}
We have performed the simulations for various mesh size $N$. The results for $k=1$ with the second-order RKDG scheme \eqref{eq:fully-scheme-2nd-0}-\eqref{eq:fully-scheme-2nd} and $k=2$ with the third-order RKDG scheme \eqref{eq:fully-scheme-3rd-0}-\eqref{eq:fully-scheme-3rd} are shown in Table \ref{table:2}. We can observe the optimal convergence rate for all the variables and $k=1,2$, which confirms the high-order accuracy of the proposed RKDG method. More specifically, the different source term approximations in each 
stage of the third-order RK method \eqref{eq:fully-scheme-3rd-0}-\eqref{eq:fully-scheme-3rd} yields the desired third-order accuracy.
\begin{table}[!ht]
	\caption{Example \ref{exam_accuracy}, accuracy test far away from the equilibrium state for $k=1,2$ with equations \eqref{eq:change}.}
\small
\begin{tabular}{c c c c c c c c}
	\toprule
	Case & $N$ & \multicolumn{2}{c}{$\rho$} & \multicolumn{2}{c}{$\rho u$} & \multicolumn{2}{c}{$E$}\\
	\midrule
	\multirow{4}{*}{$k=1$} & 25 & 4.12E-04 & - & 5.17E-04 & - & 6.46E-04 & - \\
						   & 50 & 1.04E-04 & 1.98 & 1.31E-04 & 1.98 & 1.63E-04 & 1.99 \\
						   & 100 & 2.63E-05 & 1.99 & 3.29E-05 & 1.99 & 4.10E-05 & 1.99 \\
						   & 200 & 6.60E-06 & 1.99 & 8.59E-06 & 2.00 & 1.03E-05 & 2.00 \\
	\midrule
	\multirow{4}{*}{$k=2$} & 25 & 1.29E-05 & - & 1.75E-05 & - & 9.69E-06 & - \\
						   & 50 & 1.82E-06 & 2.82 & 2.41E-06 & 2.86 & 1.33E-06 & 2.87 \\
                           & 100 & 2.44E-07 & 2.90 & 3.17E-07 & 2.92 & 1.75E-07& 2.92 \\
                           & 200 & 3.16E-08 & 2.95 & 4.08E-08 & 2.96 & 2.25E-08 & 2.96 \\
   	\bottomrule
\end{tabular}
	\label{table:2}
\end{table}

\subsection{Explosion}
\label{exam_explosion}
In this example, we validate the shock capturing and total energy conservation properties of our proposed scheme. The initial data is given by
\begin{align}
	&\rho(r,0)=\frac{\sin(\sqrt{2\pi/\kappa}r)}{\sqrt{2\pi/\kappa}r},\quad \rho u(r,0)=0,\nonumber\\
	&p(r,0)=\begin{cases}\alpha\kappa\rho(r,0)^2, &r\le r_1 \cr \kappa\rho(r,0)^2, &r>r_1\end{cases},
\end{align}
where we set $\kappa=1$, $\gamma=2$, $G=1$ and increase the equilibrium pressure by a factor $\alpha=10$ for $r\le r_1=0.1$. The computational domain is set as $\Omega=[0,0.5]$, and discretized with $N=200$ cells. We use $P^2$ piecewise polynomial and the third-order RK method \eqref{eq:fully-scheme-3rd-0}-\eqref{eq:fully-scheme-3rd}. We set the boundary condition of the velocity $u(0.5,t)=0$ at the outer domain boundary. We perform the simulation up to time $t=0.15$, and the numerical results are shown in Figure \ref{fig:exp2}. Both the well-balanced scheme and the standard DG scheme perform similarly in capturing shocks, which means our proposed scheme does not diminish the robustness of the shock capturing capability. Moreover, we can observe that our proposed scheme conserves total energy up to machine precision, while the standard DG scheme produces an error of about $3.5\times10^{-6}$ at $t=0.15$.
\begin{figure*}
	\centering
	\includegraphics[width=0.45\linewidth]{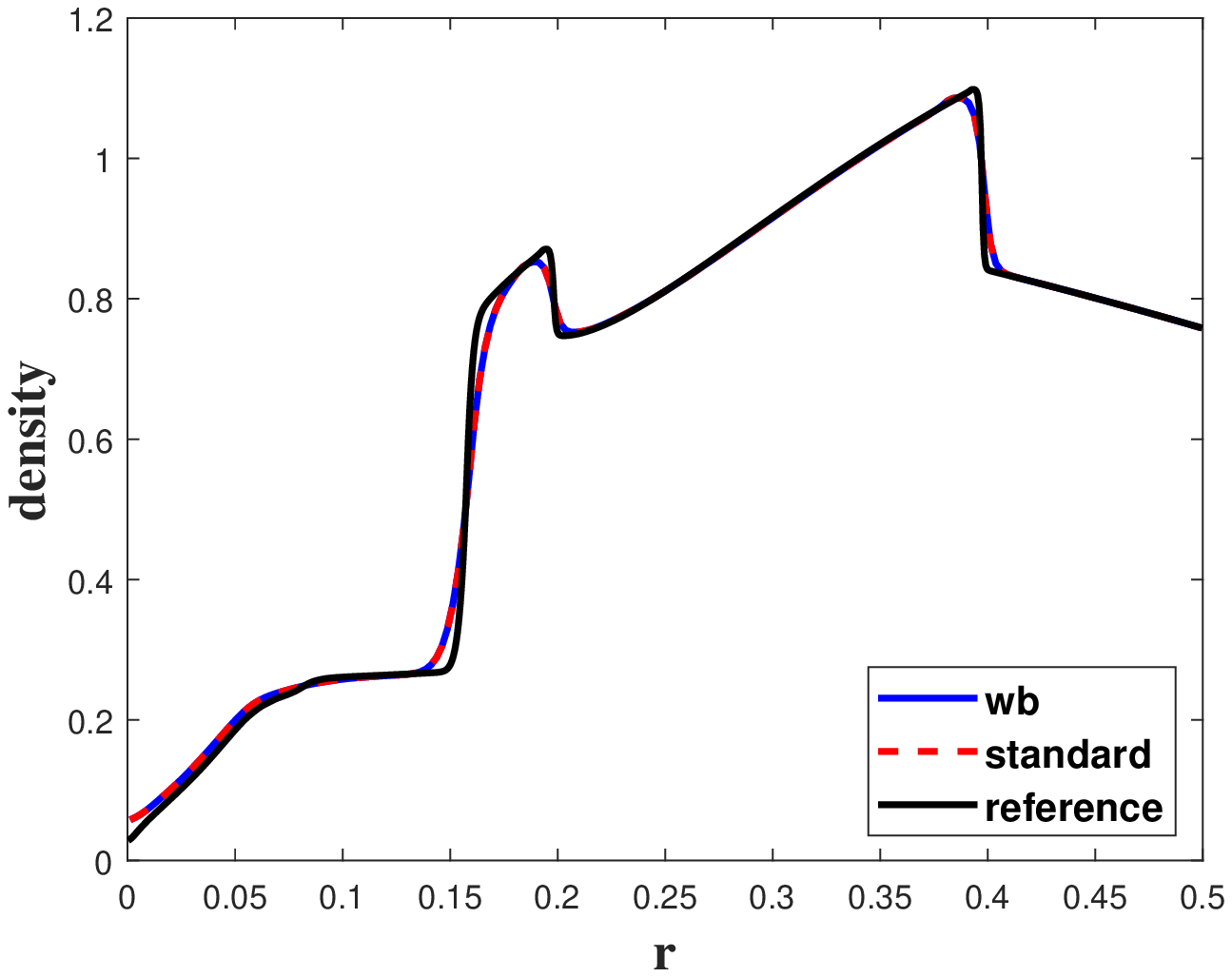}
	\includegraphics[width=0.45\linewidth]{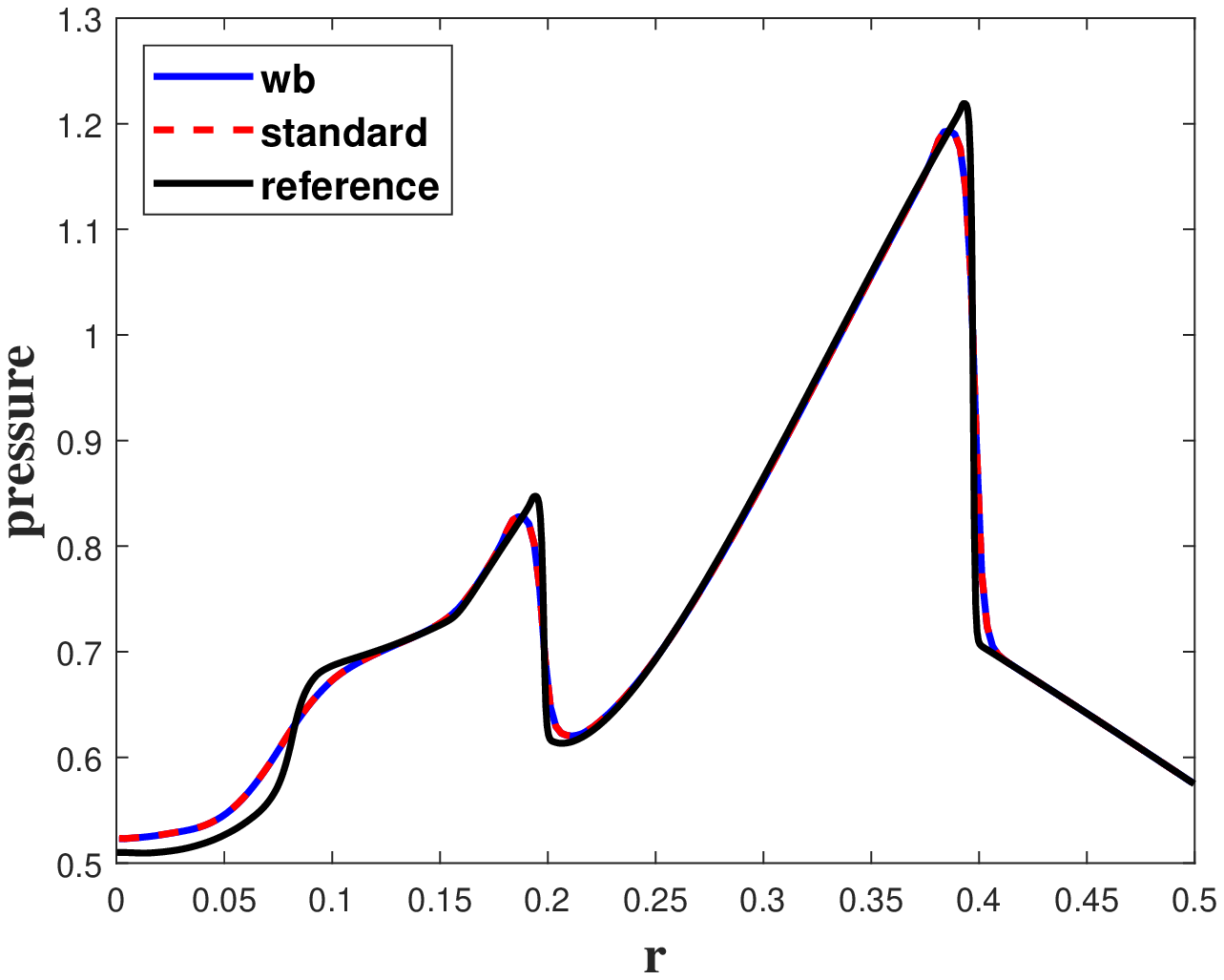} \\
	\includegraphics[width=0.45\linewidth]{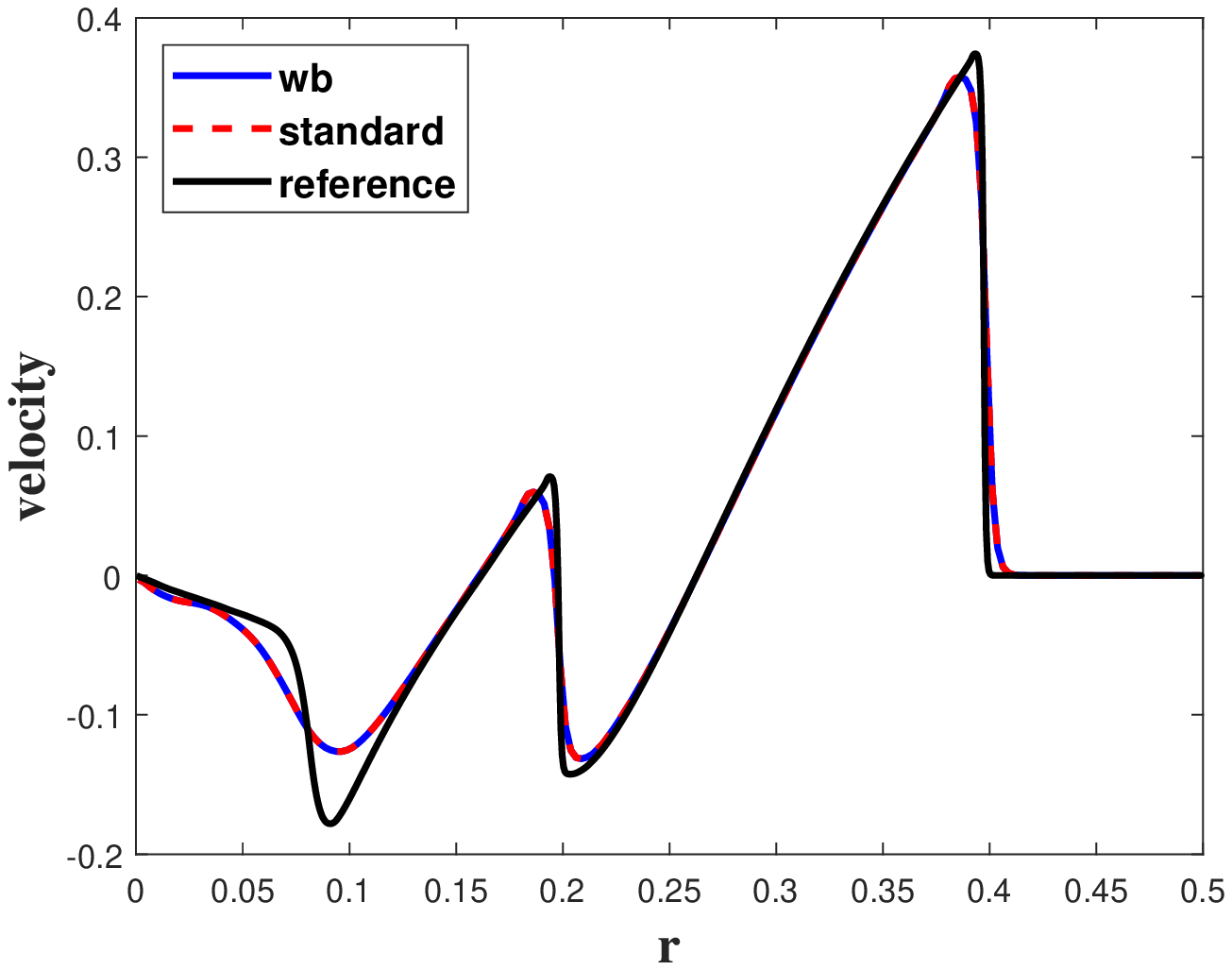}
	\includegraphics[width=0.45\linewidth]{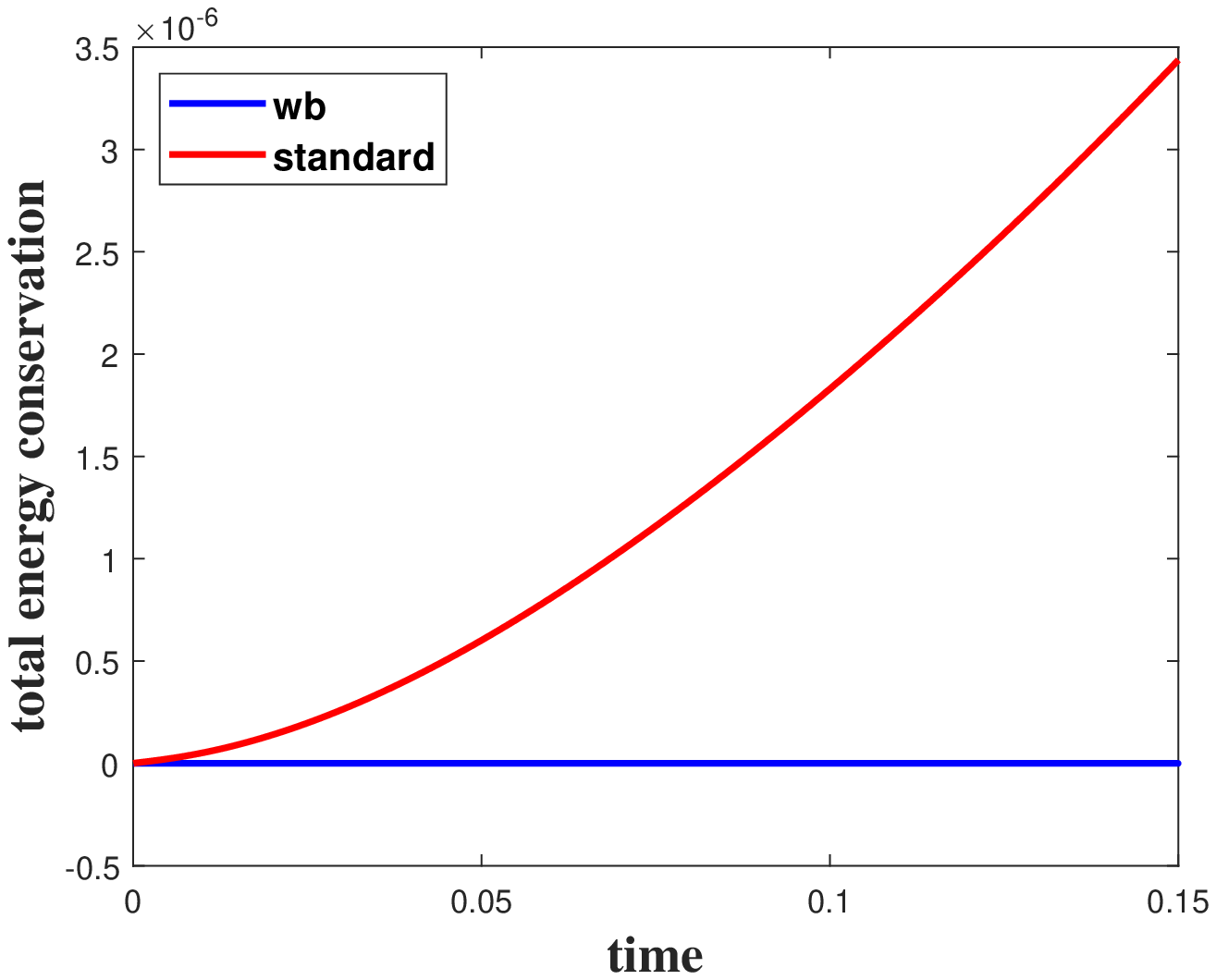}	
	\caption{The solution of well-balanced scheme (blue) and standard DG scheme (red) by using $N=200$ cells, compared with the reference solution (black) produced with $N=800$ cells. From left to right: the numerical solutions of density, velocity, pressure at time $t=0.15$ and the time history of the changes in total energy. The maximum absolute value of the changes in total energy is $8.049\times10^{-15}$ for the proposed scheme.}
	\label{fig:exp2}
\end{figure*}

\subsection{Yahil-Lattimer collapse}
\label{exam_yahil}
In this section, we consider the Yahil-Lattimer collapse test, which involves self-gravity and was studied in \cite{endeve2019thornado}, using standard DG methods. It models the self-similar collapse of a polytropic star, i.e. $p=\kappa\rho^\gamma$. In \cite{yahil1983self}, self-similar solutions to the gravitational collapse problem were constructed for $6/5\leq\gamma<4/3$. With two dimensional parameters in the model (the gravitational constant $G$ and the polytropic constant $\kappa$), the dimensionless similarity variable is
\begin{equation}
	X=\kappa^{-\frac12}G^{(\gamma-1)/2}r(-t)^{\gamma-2},
\end{equation}
where the origin of time is the moment of infinite central density. All the hydrodynamic variables can be expressed as a function of $X$, and the time-dependent Euler equations can be recast as a system of ODEs \citep[see][for details]{yahil1983self}. Therefore, we use these self-similar solutions solved by the ODEs given in \cite{yahil1983self} as a reference solution.

We show some numerical results obtained with $\gamma=1.3$. We set the computational domain to $\Omega=[0,10^{10}]$~cm discretized with $N=256$ cells, and the collapse time to $(-t)=150$~ms. 
We use a geometrically increasing cell spacing
\begin{equation}
	\Delta r_j=r_{j+\frac12}-r_{j-\frac12}=a^{j-1}\Delta r_1,\qquad j=1,...,N,
\end{equation}
with the size of the innermost cell set to $\Delta r_1=1\times10^5$~cm, and increasing at a rate $a=1.03203$.
The size of the last element is about $3\times10^8$ cm. 
The gravitational constant $G$ is set to $6.67430\times10^{-8}$ $\mbox{cm}^{-3}\ \mbox{g}^{-1}\ \mbox{s}^{-2}$. We use the reference solution at time $(-t)=150$~ms to compute the initial density and velocity. The polytropic constant $\kappa=9.54\times10^{14}$ is used to give the initial pressure.  We use the reflecting boundary condition for the inner boundary and zeroth-order extrapolation for the outer boundary.


We simulate collapse until $(-t)=0.5$ ms, and the central density increases from about $10^9$ g $\mbox{cm}^{-3}$ to about $10^{14}$ g $\mbox{cm}^{-3}$. We plot the density $\rho$ and velocity $u$ at different times in Figure \ref{fig:exp_yahil}, and compare the results with the reference solutions obtained in \cite{yahil1983self}. The figures show that our numerical method performs well during collapse. We also compare the total energy conservation property between our proposed scheme and the standard DG scheme. The total energy is defined as $E_{tot}=\int_{\Omega}\left(E+\frac12\,\rho\,\Phi\right)\,r^2\,\mathrm{d}r$. The total energy conservation for RK3 time discretization $\Delta E$ is defined as follows
\begin{align}
	\Delta E(t^{m+1})=&E_{tot}(t^{m+1})-E_{tot}(t^{m})\nonumber\\
	&+4\pi\Delta t\,R^2\frac{\hat{\boldsymbol{f}}^{n,[3]}_{N+\frac12}+\boldsymbol{f}^{(1),[3]}_{N+\frac12}+4\boldsymbol{f}^{(2),[3]}_{N+\frac12}}{6},\nonumber\\
	\Delta E=&\sum_{m=1}^{M}\Delta E(t^{m+1}),\label{eq:correction-energy}
\end{align}
where $R$ is the outer boundary, $N$ is the number of cells and $M$ is the number of time steps. When the time is close to $(-t)=0.5$~ms and the density grow rapidly to $10^{14}$ g $\mbox{cm}^{-3}$, our proposed scheme maintains total energy conservation to round-off error while that of the standard scheme is much larger.
\begin{figure*}
	\centering
	\includegraphics[width=0.45\linewidth]{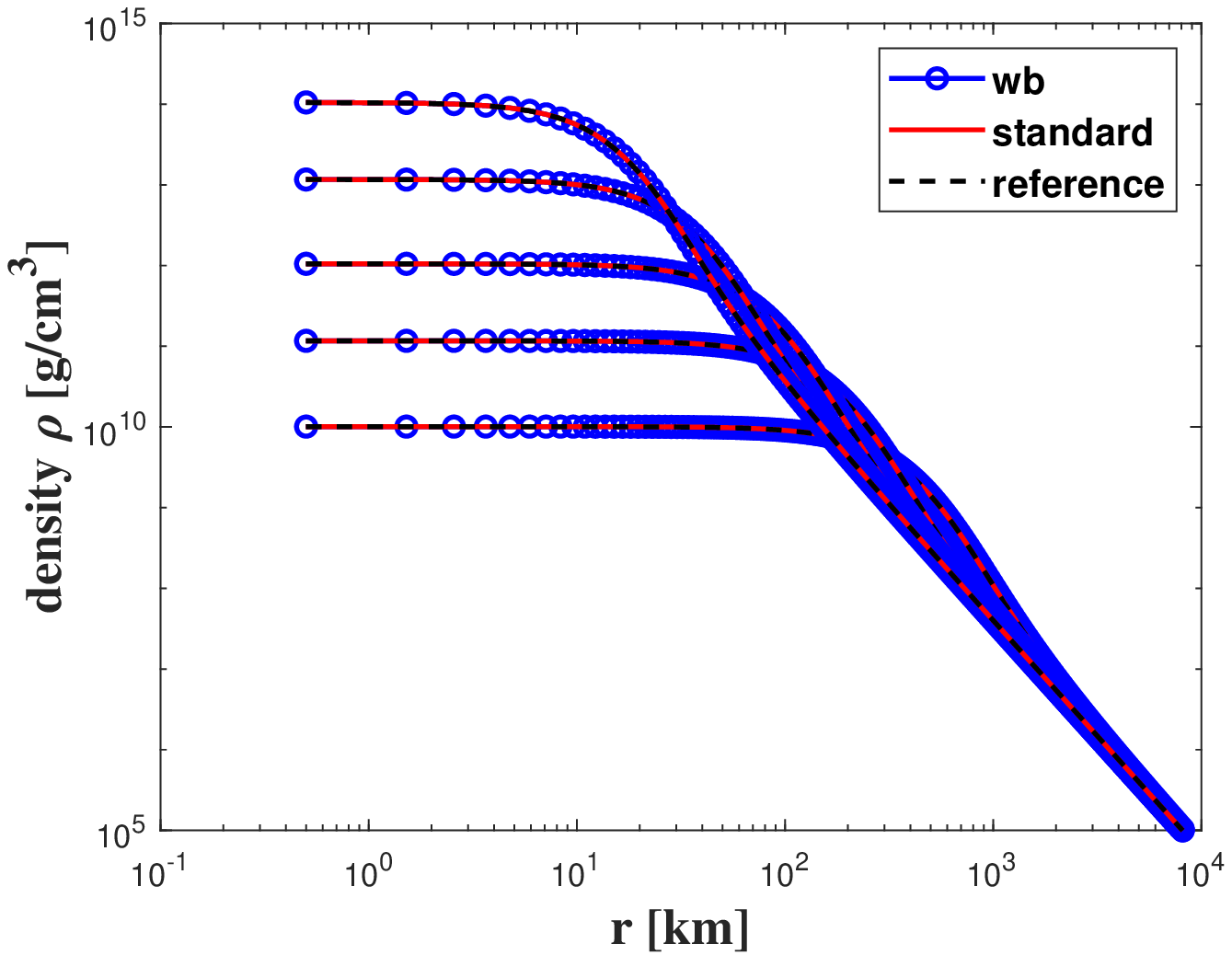}
	\includegraphics[width=0.45\linewidth]{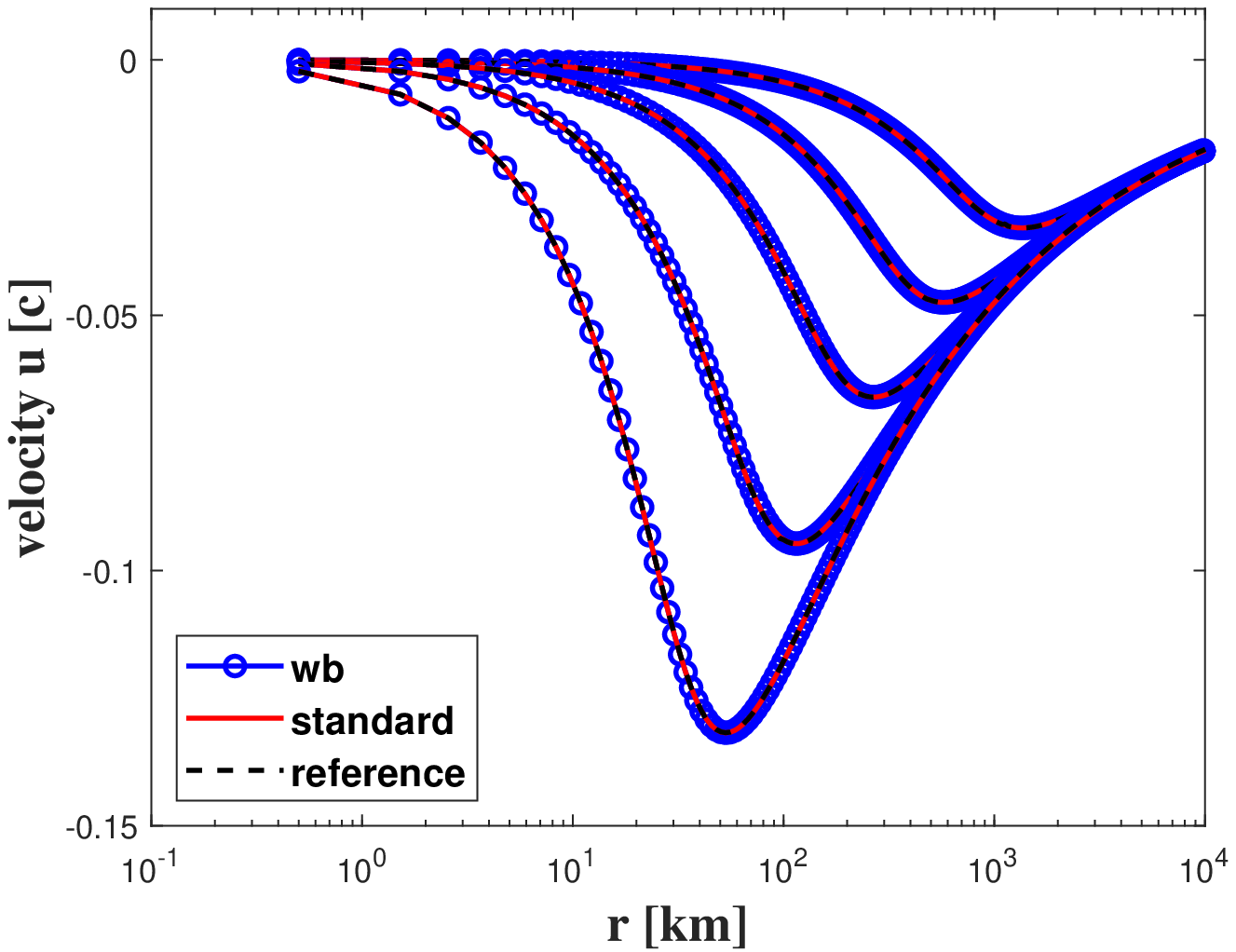} \\
	\includegraphics[width=0.65\linewidth]{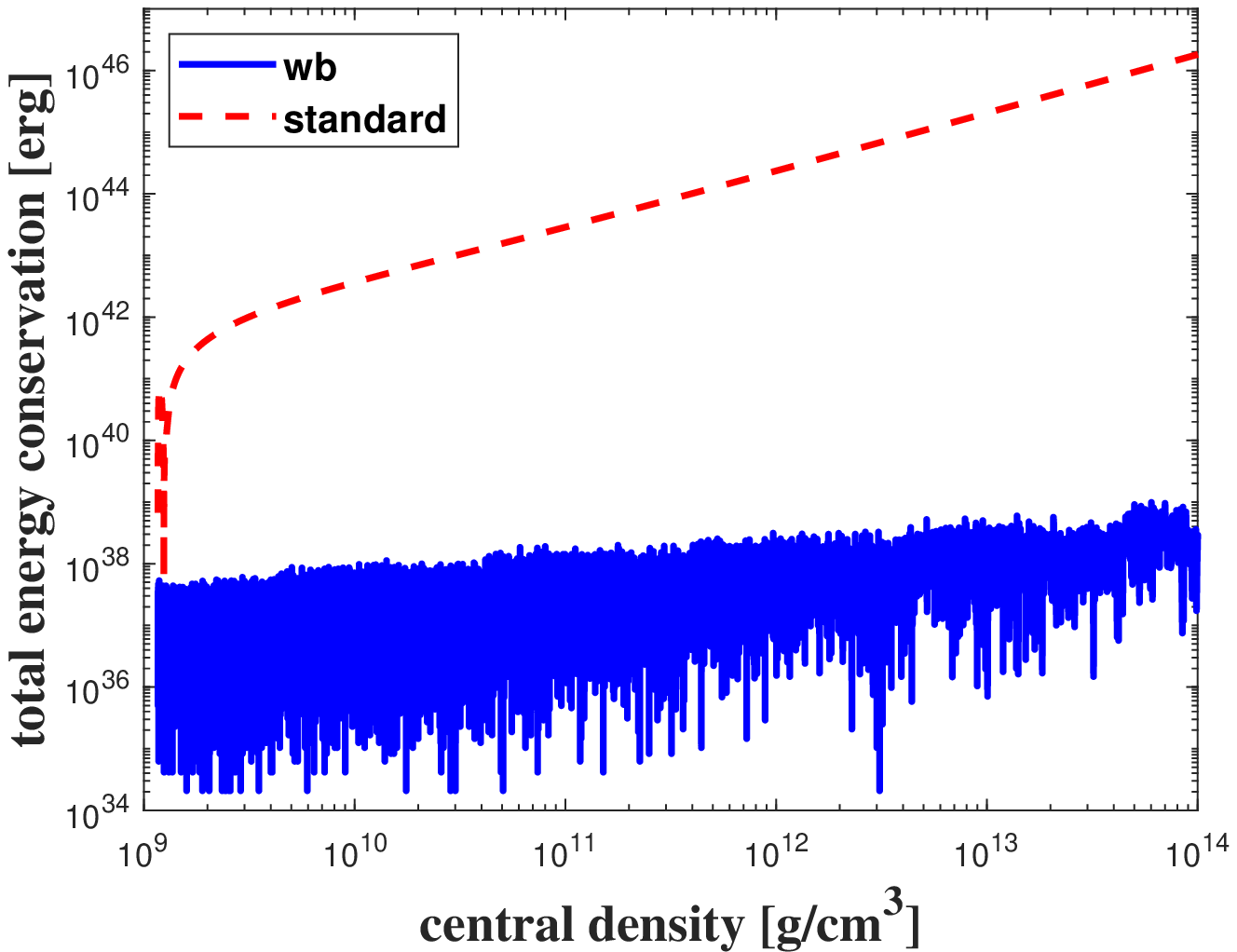}
	\caption{Example \ref{exam_yahil}, the figure of numerical solution (blue) of density $\rho$ (top left) and velocity $u$ (top right) during collapse, compared with the standard scheme (red) and the reference solution (black). We compared the solutions at select central densities, approximately $[10^{10},10^{11},10^{12}, 10^{13}, 10^{14}]$ g cm$^{-3}$, which correspond to $(-t)=[51.0, 15.0, 5.0, 1.5, 0.5]$ ms. Velocity gradually decreases over time. The comparison of the total energy conservation between our proposed scheme and standard scheme versus central density shows in the bottom that when the time is close to $(-t)=0.5$ ms, our proposed scheme has a much smaller total energy conservation than the standard scheme.}
	\label{fig:exp_yahil}
\end{figure*}

\subsection{Toy model of stellar core-collapse, bounce, and shock evolution}\label{exam_toy}

We consider a toy model of core-collapse supernova as considered in \cite{janka1993does,kappeli2016well}. This test simulates the spherically symmetric and adiabatic collapse, bounce, shock evolution, and proto-neutron star formation for a simplified model using a phenomenological EoS.  
This test provides a stringent check on the energy conservation properties of our proposed scheme --- especially during core bounce when core-collapse supernova codes typically exhibit an abrupt change in the total energy \citep[e.g.,][]{skinner_etal_2019,bruenn_etal_2020}.  

The governing equations are given by \eqref{eq:mass}-\eqref{eq:energy} and \eqref{eq:poisson} with a non-ideal EoS. We first set $\gamma=4/3$ and obtain an equilibrium state according to \eqref{eq:equilibrium} and \eqref{eq:polytropic} for a central density $\rho_c=10^{10}\,\text{g/cm}^3$, polytropic constant $\kappa=4.897\times10^{14}$ (in cgs units), and gravitational constant $G=6.67430\times10^{-8}\,\text{cm}^{-3}\text{g}^{-1}\text{s}^{-2}$. We initialize the collapse by reducing the adiabatic index from $\gamma=4/3$ to a slightly smaller value $\gamma_1=1.325$. Then the initial internal energy density is set as $\rho e=\kappa\rho^{\gamma_1}/(\gamma_1-1)$ where the initial density $\rho$ is the equilibrium density for $\gamma=\frac43$ and the initial momentum is set to zero.

The EoS in this test consists of two parts, a polytropic part and a thermal part, taking the form
\begin{align}
	&p=p_{\rm{p}}+p_{\rm{th}},\\
	&\rho e=(\rho e)_{\rm{p}}+(\rho e)_{\rm{th}}.
\end{align}
The polytropic part is given by
\begin{equation}
	p_{\rm{p}}=p_{\rm{p}}(\rho)=\begin{cases}
		\kappa_1\rho^{\gamma_1}, & \rho<\rho_{\rm{nuc}},\\
		\kappa_2\rho^{\gamma_2}, & \rho\ge\rho_{\rm{nuc}},
	\end{cases}
\end{equation}
where $\rho_{\rm{nuc}}=2\times10^{14}\,\text{g/cm}^3$ is the nuclear density parameter and separates two different regimes with different adiabatic indexes, $\gamma_1=1.325$ and $\gamma_2=2.5$ (This mimics the stiffening observed in more realistic EoSs as the matter composition transitions from consisting of nucleons and nuclei to bulk nuclear matter.) The polytropic internal energy density is given by
\begin{equation}
	(\rho e)_{\rm{p}}=(\rho e)_{\rm{p}}(\rho)=\begin{cases}
		E_1\rho^{\gamma_1}, & \rho<\rho_{\rm{nuc}},\\
		E_2\rho^{\gamma_2}+E_3\rho, & \rho\ge\rho_{\rm{nuc}},
	\end{cases}
\end{equation}
where the parameters $E_1,E_2,E_3,\kappa_1,\kappa_2$ are given by
\begin{align}
	&E_1=\frac{\kappa}{\gamma_1-1},\quad \kappa_1=\kappa, \quad \kappa_2=(\gamma_2-1)E_2,\nonumber\\
	&E_2=\frac{\kappa}{\gamma_2-1}\rho_{\rm{nuc}}^{\gamma_1-\gamma_2},\quad E_3=\frac{\gamma_2-\gamma_1}{\gamma_2-1}E_1\rho_{\rm{nuc}}^{\gamma_1-1}.
\end{align}
One can easily check that the polytropic pressure and internal energy density are both continuous across the density $\rho=\rho_{\rm{nuc}}$. The thermal part is given by
\begin{align}
	p_{\rm{th}}=(\gamma_{\rm{th}}-1)(\rho e)_{\rm{th}},\qquad (\rho e)_{\rm{th}}=\rho e-(\rho e)_{\rm{p}},
\end{align}
where $\gamma_{\rm{th}}=1.5$. We note that the initial thermal pressure is zero in this test. Combining the above expressions, we can write the complete EoS in this test as
\begin{align}
	p=p(\rho,e)=\begin{cases}
		(\gamma_{\rm{th}}-1)\rho e+\frac{\gamma_1-\gamma_{\rm{th}}}{\gamma_1-1}\kappa\rho^{\gamma_1}, & \rho<\rho_{\rm{nuc}},\\
		(\gamma_{\rm{th}}-1)\rho e+\frac{\gamma_2-\gamma_{\rm{th}}}{\gamma_2-1}\kappa\rho_{\rm{nuc}}^{\gamma_1-\gamma_2}\rho^{\gamma_2}\\
		\qquad\quad-\frac{(\gamma_{\rm{th}}-1)(\gamma_2-\gamma_1)}{(\gamma_2-1)(\gamma_1-1)}\kappa\rho_{\rm{nuc}}^{\gamma_1-1}\rho, & \rho\ge\rho_{\rm{nuc}}.
	\end{cases}
\end{align}
We note that there may be a different $\gamma$ in different regions of the computational domain ($\gamma_{1}$ versus $\gamma_{2}$) and we use the $\gamma$ of the innermost cell to calculate $n$ and the corresponding numerical solution $\theta_n$ in Section \ref{sec:lane}. 

We set the computational domain as $\Omega=[0,1.5\times10^3]$ km with a geometrically increasing cell spacing
\begin{equation}
	\Delta r_j=r_{j+\frac12}-r_{j-\frac12}=a^{j-1}\Delta r_1,\qquad j=1,...,N,
\end{equation}
such that the mesh can be defined by specifying the size of the innermost cell $\Delta r_1$ and the increasing rate $a$. Different values of $\Delta r_1$ and $a$ have been utilized in the test with values specified in Table \ref{table:3}. 
We use the reflective boundary condition for the inner boundary and zeroth-order extrapolation for the outer boundary. We set $k=2$ and use the third-order RK method \eqref{eq:fully-scheme-3rd-0}-\eqref{eq:fully-scheme-3rd} in this test. The simulation is performed from $t=0$ to $t=0.11$~s. According to the description in \cite{janka1993does,kappeli2016well}, the central density will continue to increase until it exceeds nuclear density $\rho_{\rm{nuc}}$ and the EoS stiffens to form an inner core that eventually settles to a new equilibrium configuration (the proto-neutron star). Due to its inertia, the inner core overshoots its equilibrium and rebounds to form the shock wave. This is the so-called core bounce, and in this paper the time of bounce is set as the time when the average density within the innermost 2~km, which is called central density, reaches its maximum.  
Due to the absence of energy losses in our model (i.e., from deleptonization by neutrinos and dissociation of nuclei below the shock), the shock wave does not stall, but propagates towards the outer boundary of the domain.

We note that the dynamics before bounce is similar to the case discussed in Section~\ref{exam_yahil}.  We refer to the top right panel in Figure~\ref{fig:exp_yahil} for the evolution of the velocity, and the thermal energy ratio $\mathcal{P}_{th}={(\rho e)_{th}}/{(\rho e)}$ is almost zero across the whole computational domain before bounce. To illustrate the dynamics after bounce, we refer to Figure~\ref{fig:after-bounce}, which shows the fluid velocity and thermal energy ratio versus radius for select time slices. We can see the shock forms at bounce at a radius between 10 and 20~km, and then gradually propagates to the outer boundary.  The thermal energy remains very small in the inner core, below the location where the shock formed, while it increases sharply across the shock.  Behind the initial shock, several smaller shocks form and propagate radially as a result of oscillations in the proto-neutron star as it settles to a hydrostatic equilibrium state.
\begin{figure*}
	\centering
	\includegraphics[width=0.48\linewidth]{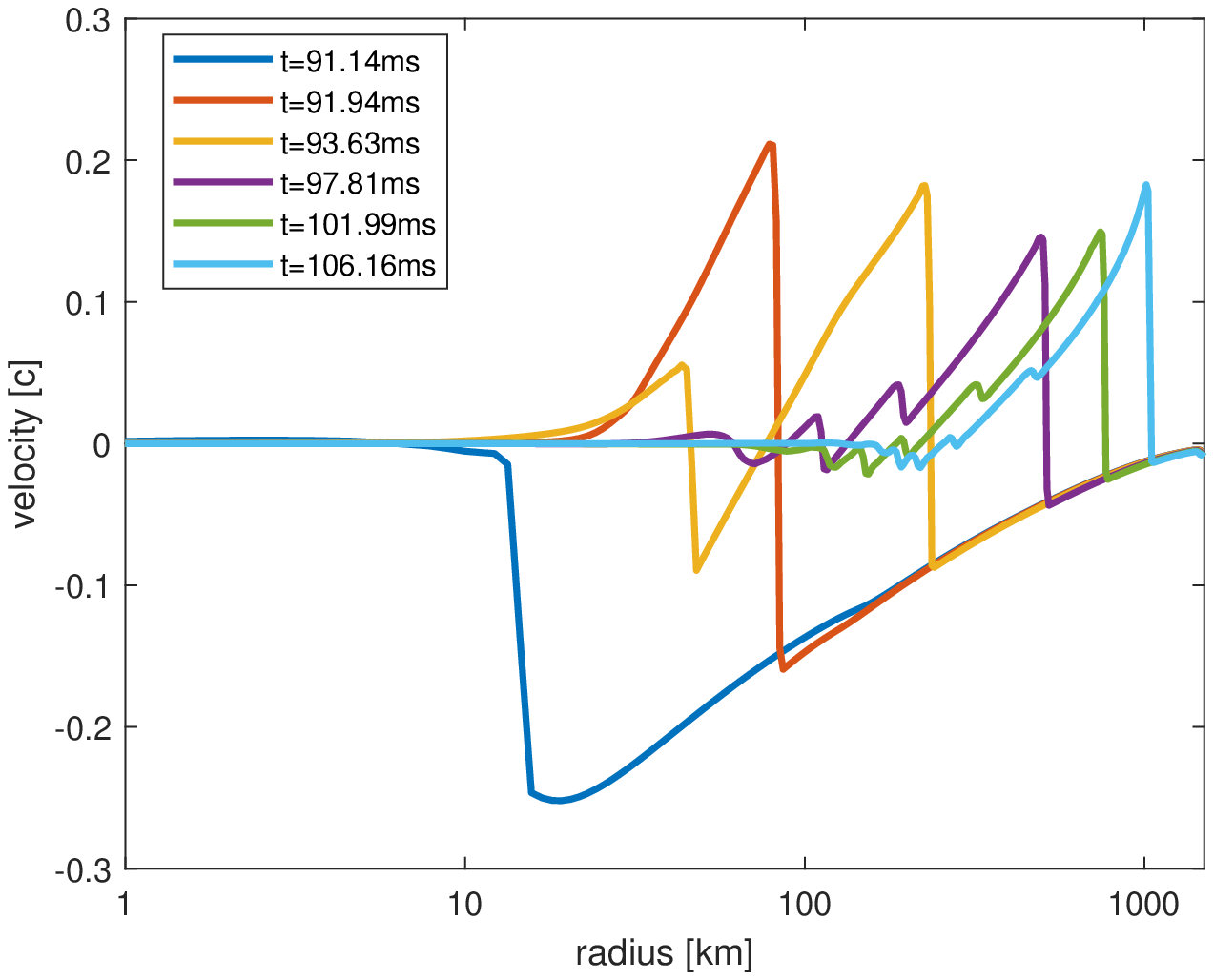}
	\includegraphics[width=0.48\linewidth]{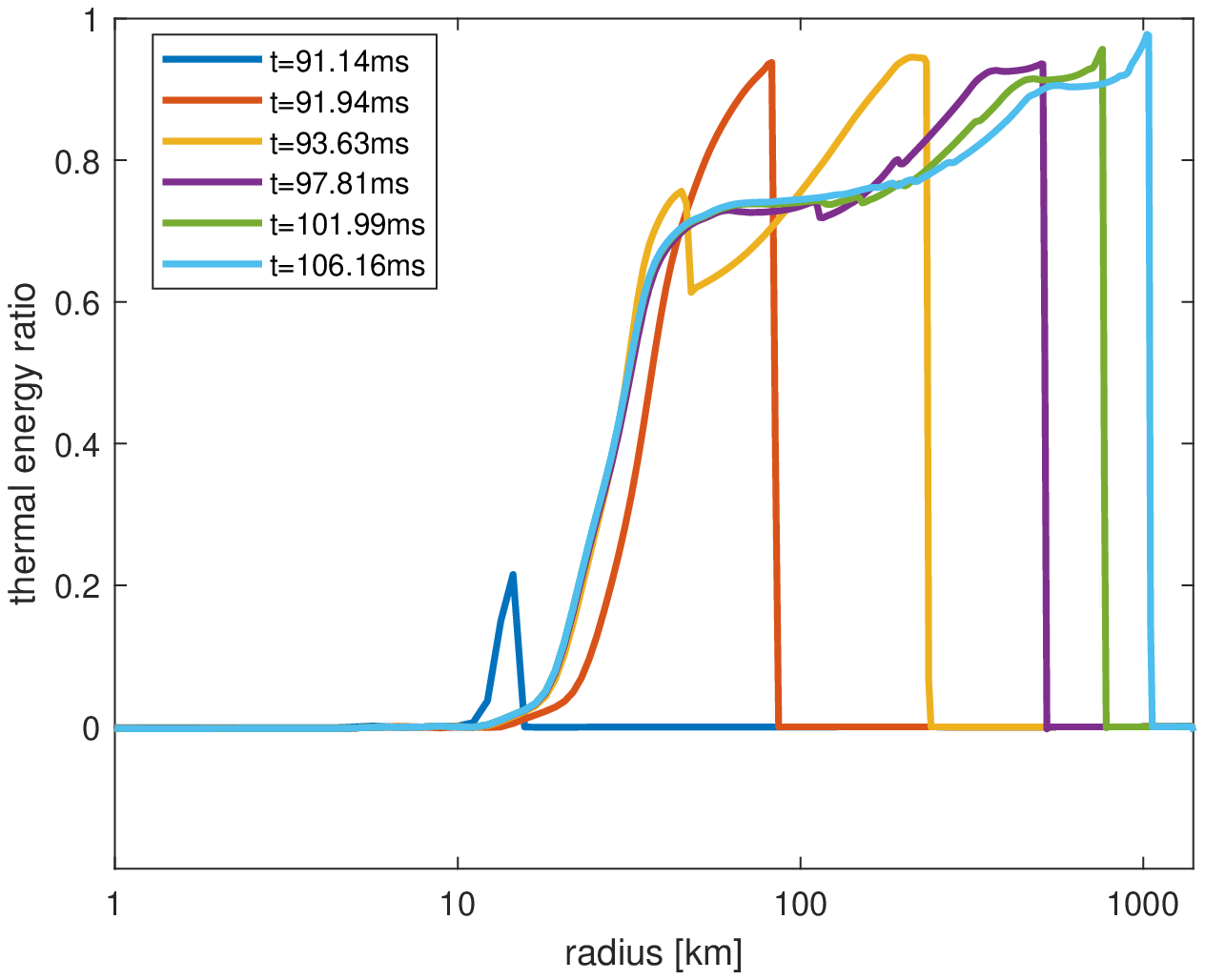}
	\caption{Example \ref{exam_toy}, fluid velocity and thermal energy ratio versus radius after bounce. We use $N=256$ cells and select 6 time slices after the bounce.}
	\label{fig:after-bounce}
\end{figure*}

We test the proposed well-balanced and energy conserving DG method and the standard DG method with different number of grids and present them in Table \ref{table:3}, from which we observe that the time of bounce, the central density of the bounce, and the final central density at $t=110$~ms are very similar for all the cases $N=128,256,512,1024,2048$. We show the central density as a function of time in Figure \ref{fig:central_density}. Both the proposed and standard DG schemes simulate this test well. In the zoom-in figure, the proposed scheme is shown to be slightly better than the standard scheme for $N=256$ and $t\in[91,94]$. In Figure \ref{fig:final_density}, we show the density versus radius at $t=0.11$~s for the case $N=128,256$. We can observe that there are small shocks at the region $r\in[200,1100]$, and our proposed scheme performs much better than standard scheme in capturing these shocks (when compare with the high resolution reference simulation), especially for the case with $N=256$.

At last, we define the energies as follows
\begin{align}
	&E_{\rm int}=\int_{\Omega}\rho e\,4\pi r^2\,\mathrm{d}r,~~
	E_{\rm kin}=\int_{\Omega}\frac12\rho u^2\,4\pi r^2\,\mathrm{d}r,\nonumber\\
	&E_{\rm grav}=\int_{\Omega}\frac12\rho\,\Phi\,4\pi r^2\,\mathrm{d}r,
\end{align}
where $E_{\rm int}$, $E_{\rm kin}$, and $E_{\rm grav}$ denote the internal energy, kinetic energy, and gravitational energy, respectively. We list these three energies $E_{\rm int}$, $E_{\rm kin}$, $-E_{\rm grav}$, and the total energy conservation $\Delta E$ in \eqref{eq:correction-energy} for different number of cells $N$ at time $t=0.11$ s in Table~\ref{table:4}. Our objective is to study how different schemes and limiters affect the total energy conservation $\Delta E$. Three different cases are considered in this table: our well-balanced and total-energy-conserving scheme, the standard RKDG scheme, and the standard scheme with the new limiter correction \eqref{eq:tvd2} (results for this latter scheme are also plotted in the bottom panels in Figure~\ref{fig:central_density}).  The reason for including the standard scheme with the correction is motivated by results from \citet{pochik2021thornado}, which suggest that limiters may negatively impact the energy conservation properties of the standard DG scheme for the Euler--Poisson system. From Table~\ref{table:4} (rightmost column), we can see that the well-balanced scheme can maintain the total energy conservation to round-off errors. For the standard scheme, neither the case with the standard limiter or the case with the correction term can maintain the round-off errors.  However, we note that the standard scheme with the correction is substantially better than the standard scheme with the standard limiter. We plot $E_{\rm int}$, $E_{\rm kin}$, $-E_{\rm grav}$, and total energy conservation $\Delta E$ versus time in Figure~\ref{fig:energy} for the simulations with $N=128$ and $N=256$. We can see that the total energy conservation for the standard scheme increases rapidly near bounce, and remains relatively constant thereafter, while for our proposed scheme the change in the total energy remains small and is not affected by core bounce.  

\begin{table*}
	\centering
	\caption{Example \ref{exam_toy}, the time of bounce, central density at the bounce time, and central density at the final time for different number of cells. The left and right columns below each label represent the result of the proposed scheme and standard scheme, respectively.}
\begin{tabular}{c c c c c c c c c}
	\toprule
	$N$ & $\Delta r_1$ [km] & $a-1$ & \multicolumn{2}{c}{$t_b$ [ms]} & \multicolumn{2}{c}{$\rho_b$ [$10^{14}\,\text{g/cm}^3$]} & \multicolumn{2}{c}{$\rho_f$ [$10^{14}\,\text{g/cm}^3$]}\\
	\midrule
	128 & 2 & $2.292\times10^{-2}$ & 91.10 & 91.09 & 3.65 & 3.66 & 2.87 & 2.81 \\
	\cmidrule(r){4-5} \cmidrule(r){6-7} \cmidrule(r){8-9}
	256 & 1 & $1.136\times10^{-2}$ & 91.13 & 91.13 & 3.68 & 3.68 & 2.81 & 2.79 \\
	\cmidrule(r){4-5} \cmidrule(r){6-7} \cmidrule(r){8-9}
	512 & 0.5 & $5.659\times10^{-3}$ & 91.16 & 91.16 & 3.65 & 3.63 & 2.81 & 2.80 \\
	\cmidrule(r){4-5} \cmidrule(r){6-7} \cmidrule(r){8-9}
	1024 & 0.25 & $2.823\times10^{-3}$ & 91.16 & 91.16 & 3.63 & 3.63 & 2.81 & 2.80 \\
	\cmidrule(r){4-5} \cmidrule(r){6-7} \cmidrule(r){8-9}
	2048 & 0.125 & $1.410\times10^{-3}$ & 91.17 & 91.17 & 3.62 & 3.62 & 2.81 & 2.80 \\
	\bottomrule
\end{tabular}
	\label{table:3}
\end{table*}

\begin{table*}
	\centering
	\caption{Example \ref{exam_toy}, four energies at time $t=0.11$ s. We compare the results of three schemes in this table for different number of cells $N$: the well-balanced and total-energy-conserving scheme, the standard scheme, the standard scheme with the new limiter correction \eqref{eq:tvd2}.}
\begin{tabular}{c c c c c c}
	\toprule
	$N$ & Case & $E_{\rm int}\,[10^{51}\,\text{erg}]$ & $E_{\rm kin}\,[10^{51}\,\text{erg}]$ & $-E_{\rm grav}\,[10^{51}\,\text{erg}]$ & $\Delta E\,[10^{51}\,\text{erg}]$\\
	\midrule
	\multirow{3}{*}{128} & wb & 120.0 & 3.658 & 122.6 & 4.386$\times10^{-11}$ \\
						 & standard & 117.7 & 4.091 & 119.1 & 1.269 \\
						 & standard with correction & 119.0 & 3.838 & 121.0 & 4.219$\times10^{-2}$ \\
						 \midrule
	\multirow{3}{*}{256} & wb & 117.7 & 3.452 & 120.0 & 2.886$\times10^{-10}$ \\
						 & standard & 116.8 & 3.681 & 118.8 & 0.425 \\
						 & standard with correction & 117.3 & 3.543 & 119.6 & 5.976$\times10^{-3}$ \\
						 \midrule
	\multirow{3}{*}{512} & wb & 117.2 & 3.509 & 119.7 & 2.395$\times10^{-10}$ \\
						 & standard & 116.9 & 3.602 & 119.2 & 0.170 \\
						 & standard with correction & 117.1 & 3.546 & 119.5 & 1.448$\times10^{-3}$ \\
						 \midrule
   \multirow{3}{*}{1024} & wb & 117.2 & 3.542 & 119.7 & 5.404$\times10^{-10}$ \\
						 & standard & 117.1 & 3.584 & 119.5 & 0.112 \\
						 & standard with correction & 117.1 & 3.559 & 119.6 & 3.545$\times10^{-4}$ \\
						 \midrule
   \multirow{3}{*}{2048} & wb & 117.2 & 3.556 & 119.7 & 1.466$\times10^{-9}$ \\
						 & standard & 117.1 & 3.578 & 119.6 & 0.038 \\
                         & standard with correction & 117.1 & 3.566 & 119.7 & 4.610$\times10^{-5}$ \\
                         \bottomrule
\end{tabular}
	\label{table:4}
\end{table*}

\begin{figure*}
	\centering
	\includegraphics[width=0.48\linewidth]{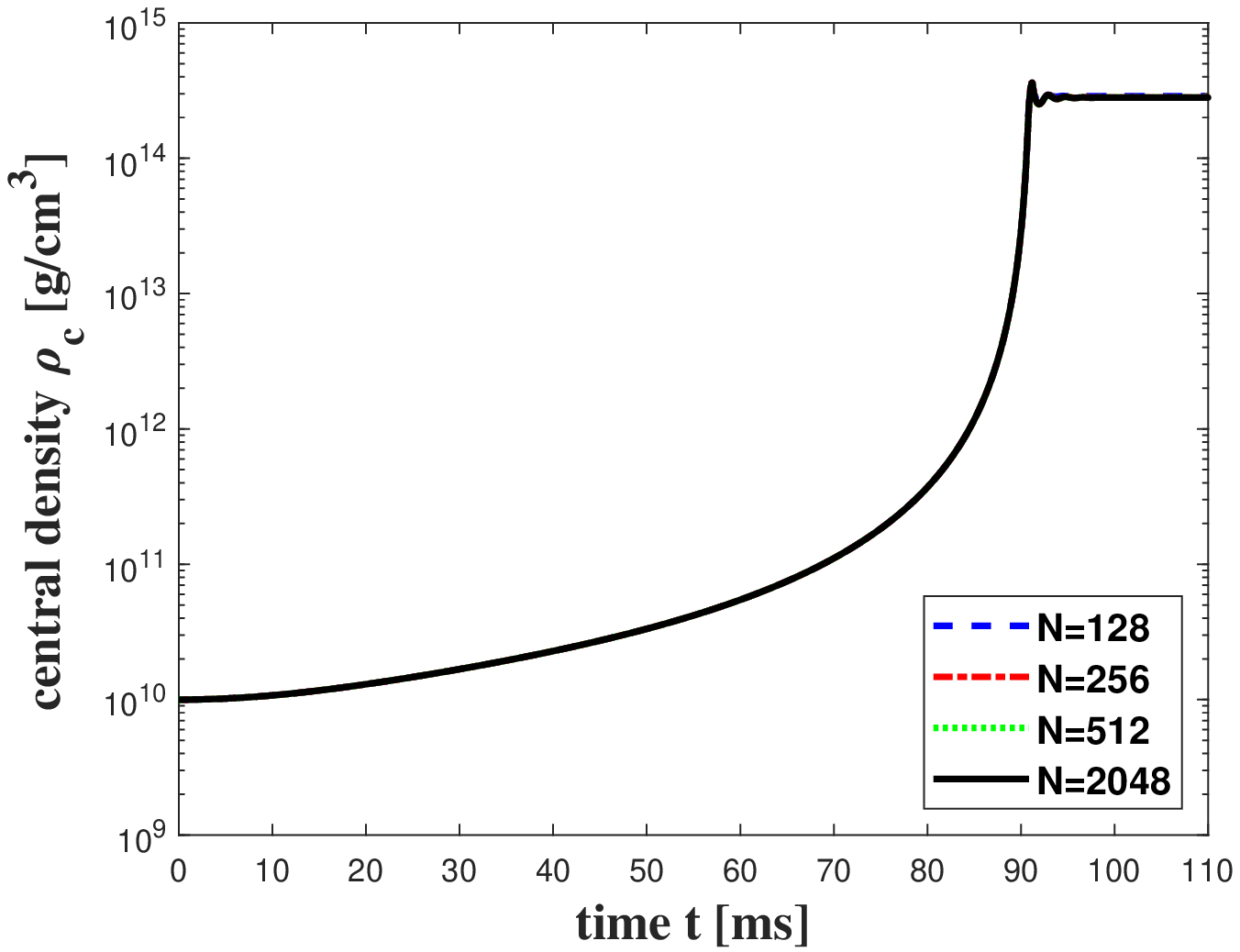}
	\includegraphics[width=0.48\linewidth]{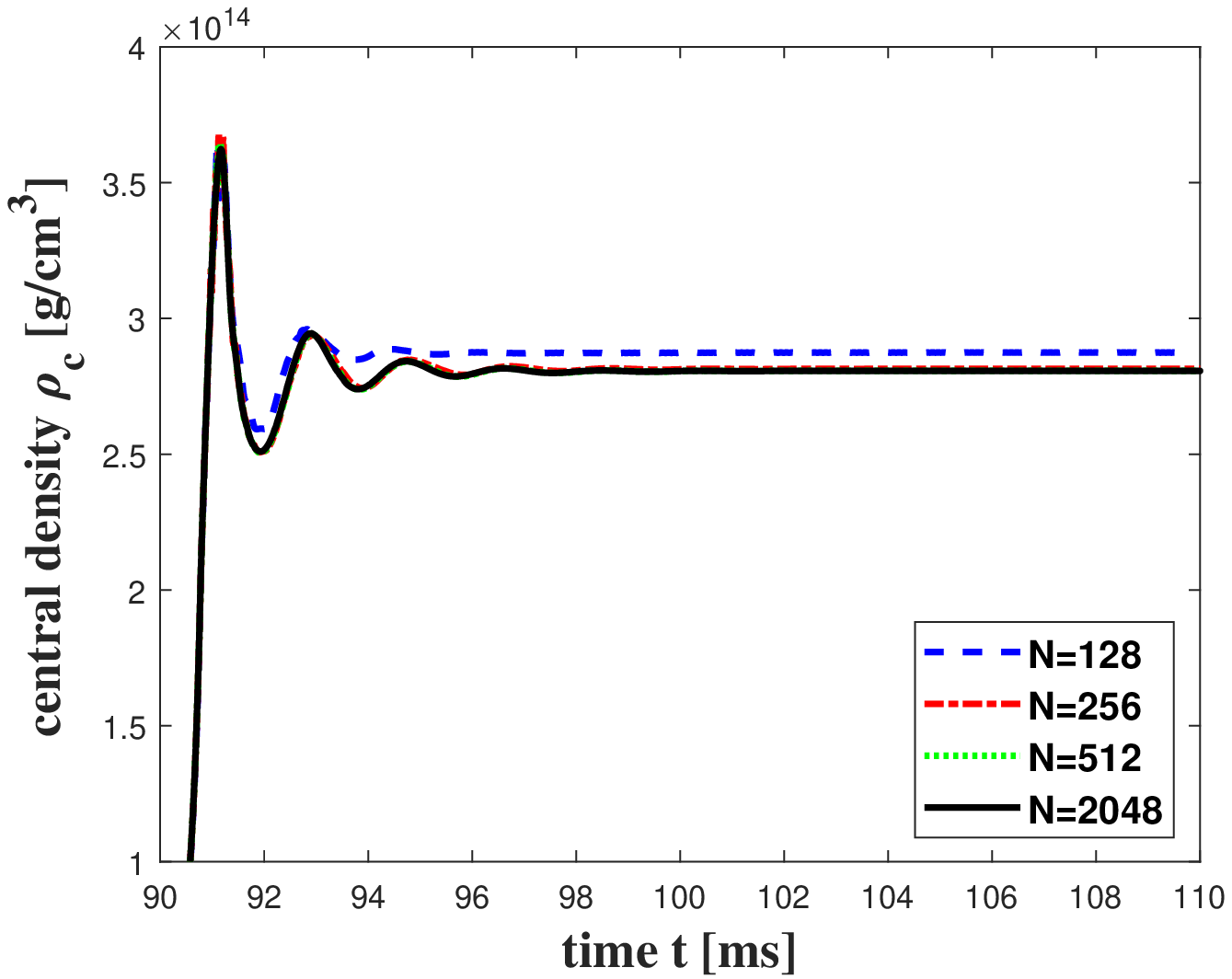}
	\includegraphics[width=0.48\linewidth]{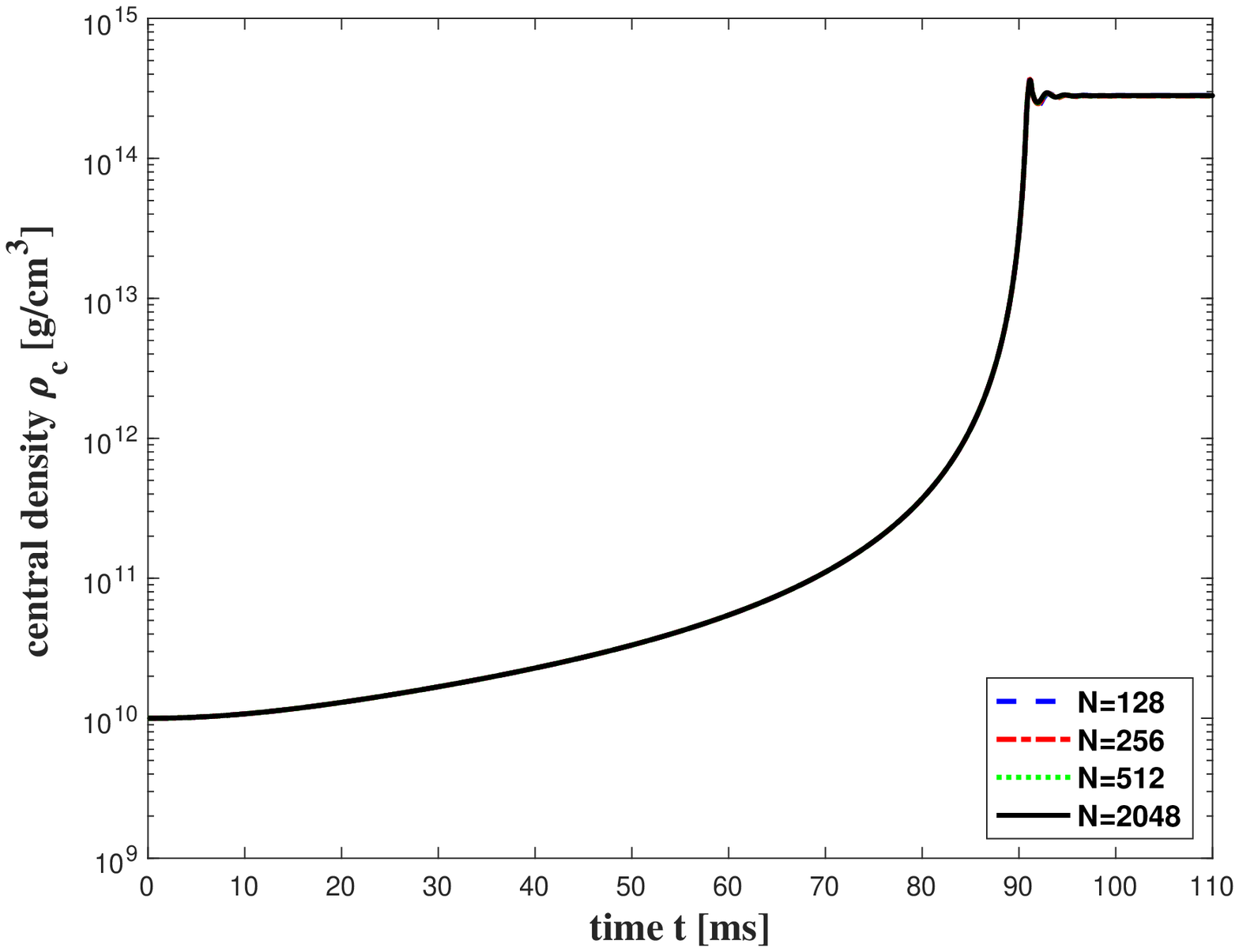}
	\includegraphics[width=0.48\linewidth]{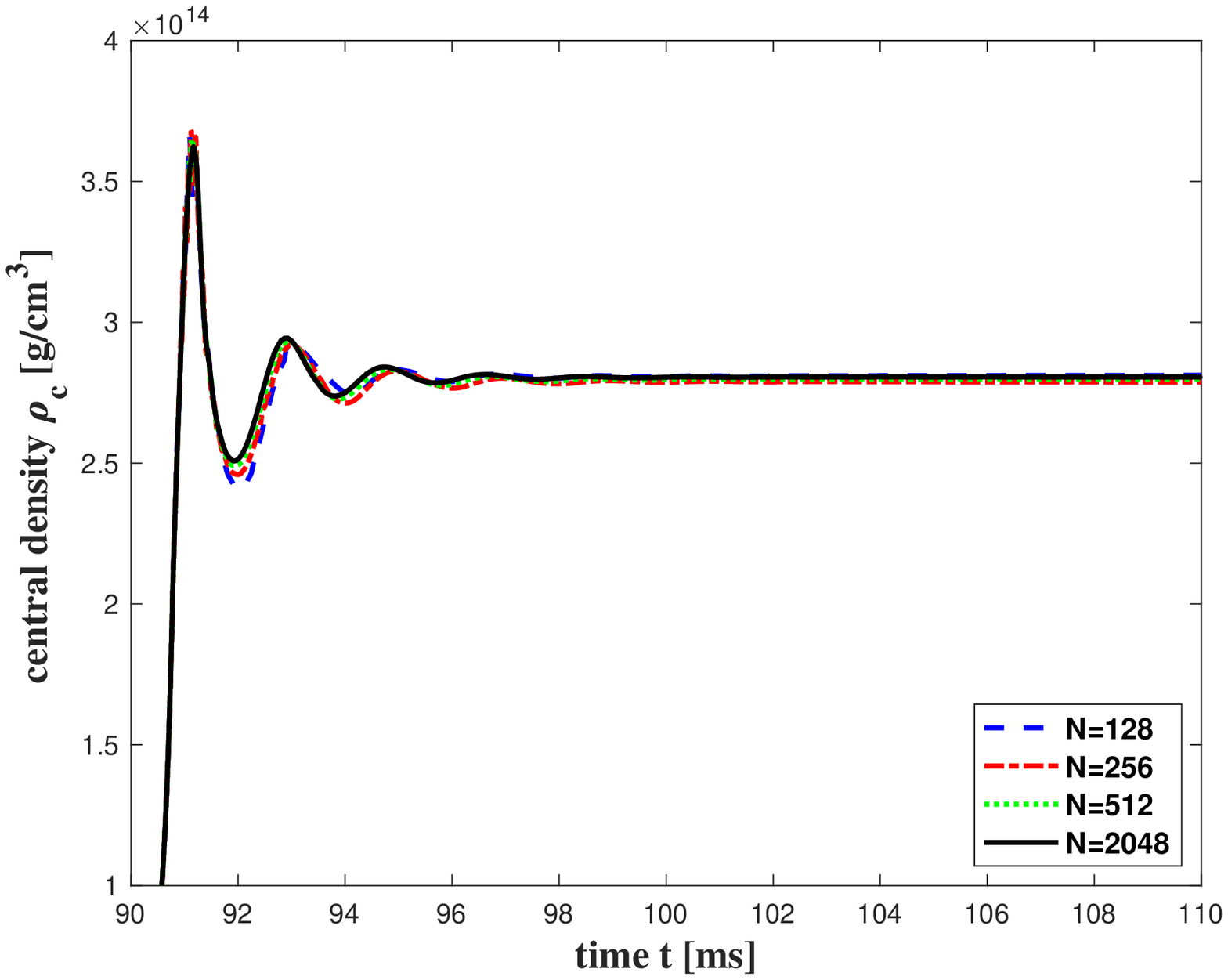}
	\includegraphics[width=0.48\linewidth]{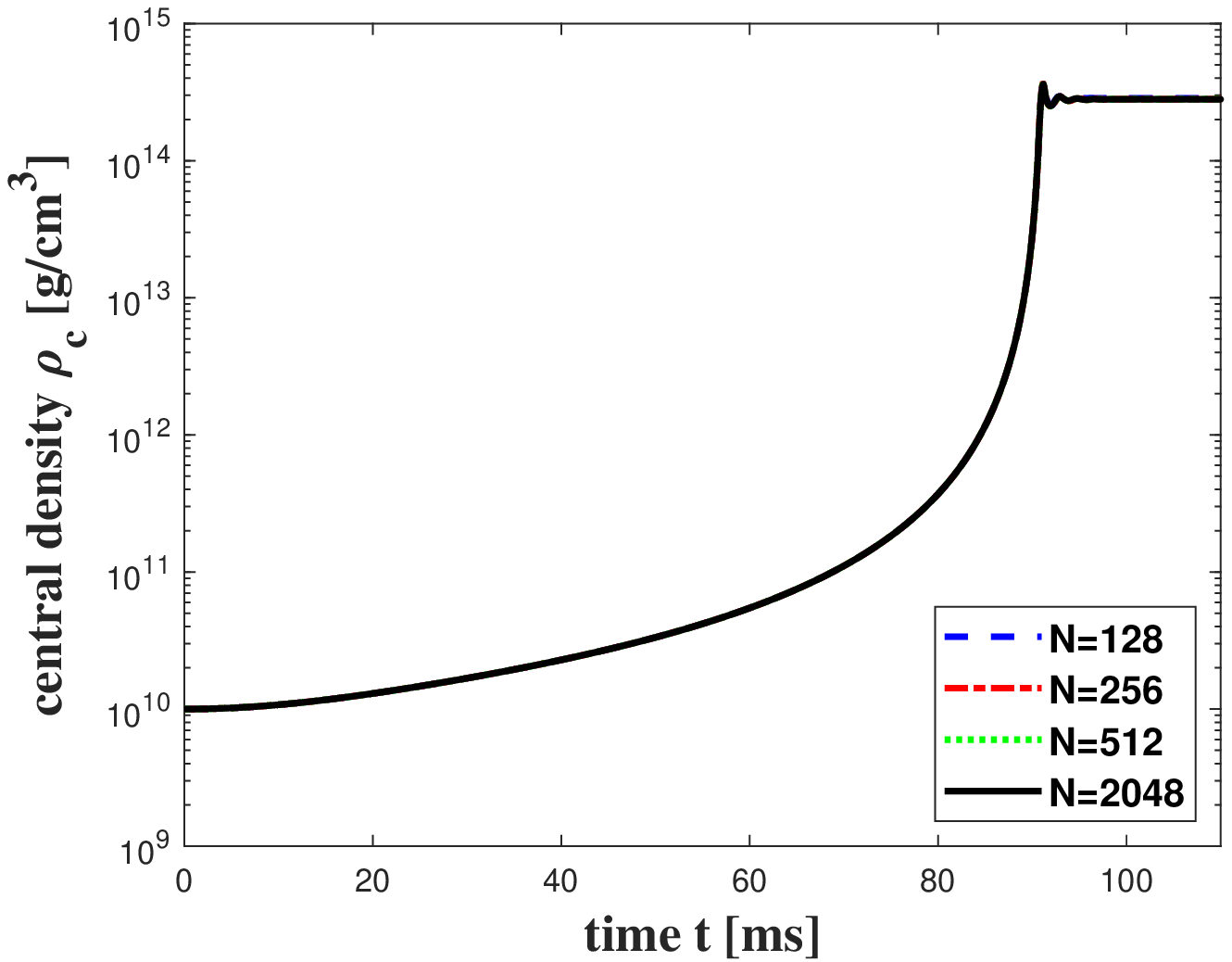}
	\includegraphics[width=0.48\linewidth]{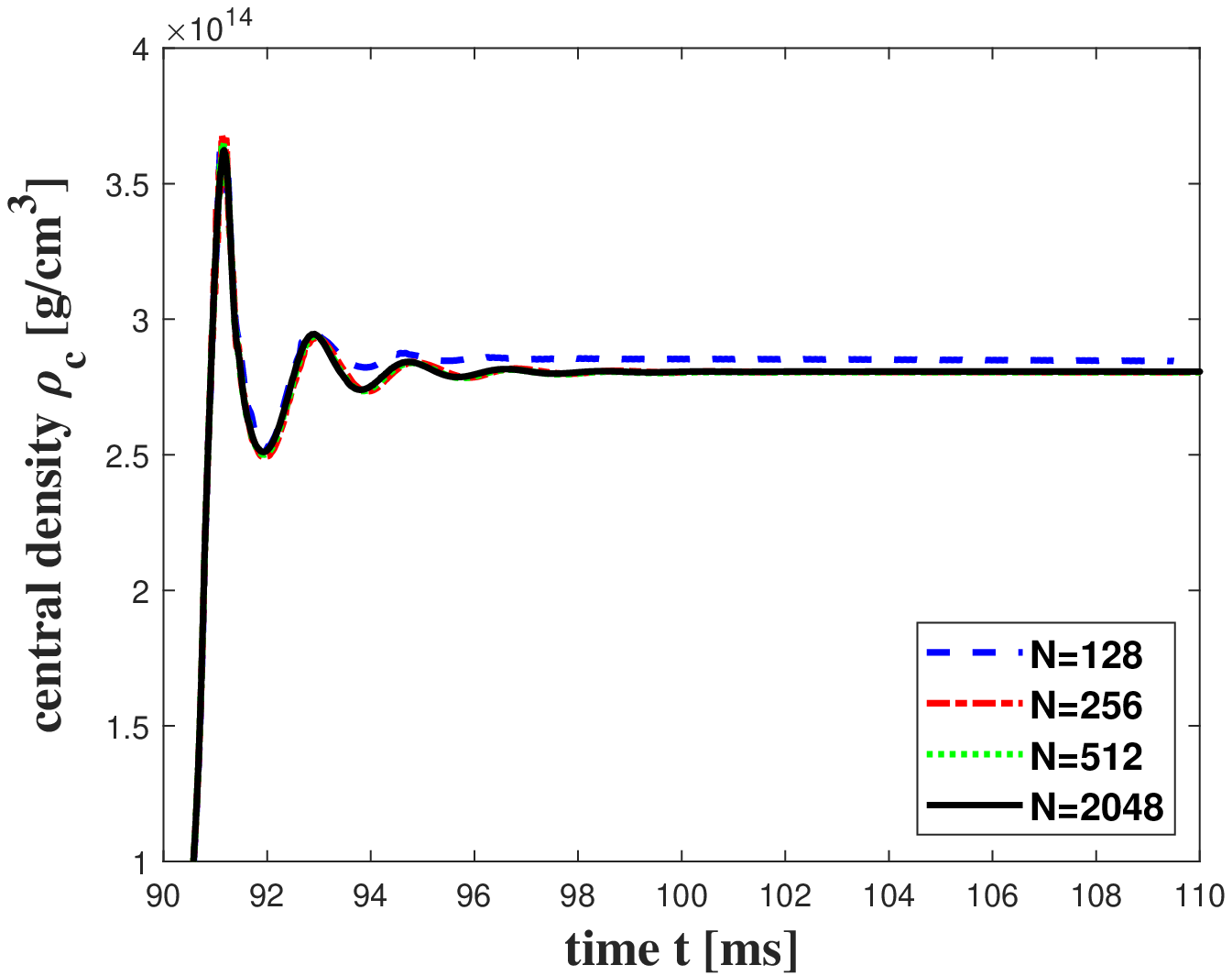}
	\caption{Example \ref{exam_toy}, central density as a function of time for the proposed (top two), the standard (mid two), and the standard with correction \eqref{eq:tvd2} (bottom two) DG schemes with $N$=128 (blue dashed), 256 (red dash-dotted), 512 (green dotted) and 2048 (black solid). The right figures represent zoomed-in versions for $t\in[90,110]$.}
	\label{fig:central_density}
\end{figure*}

\begin{figure*}
	\centering
	\includegraphics[width=0.48\linewidth]{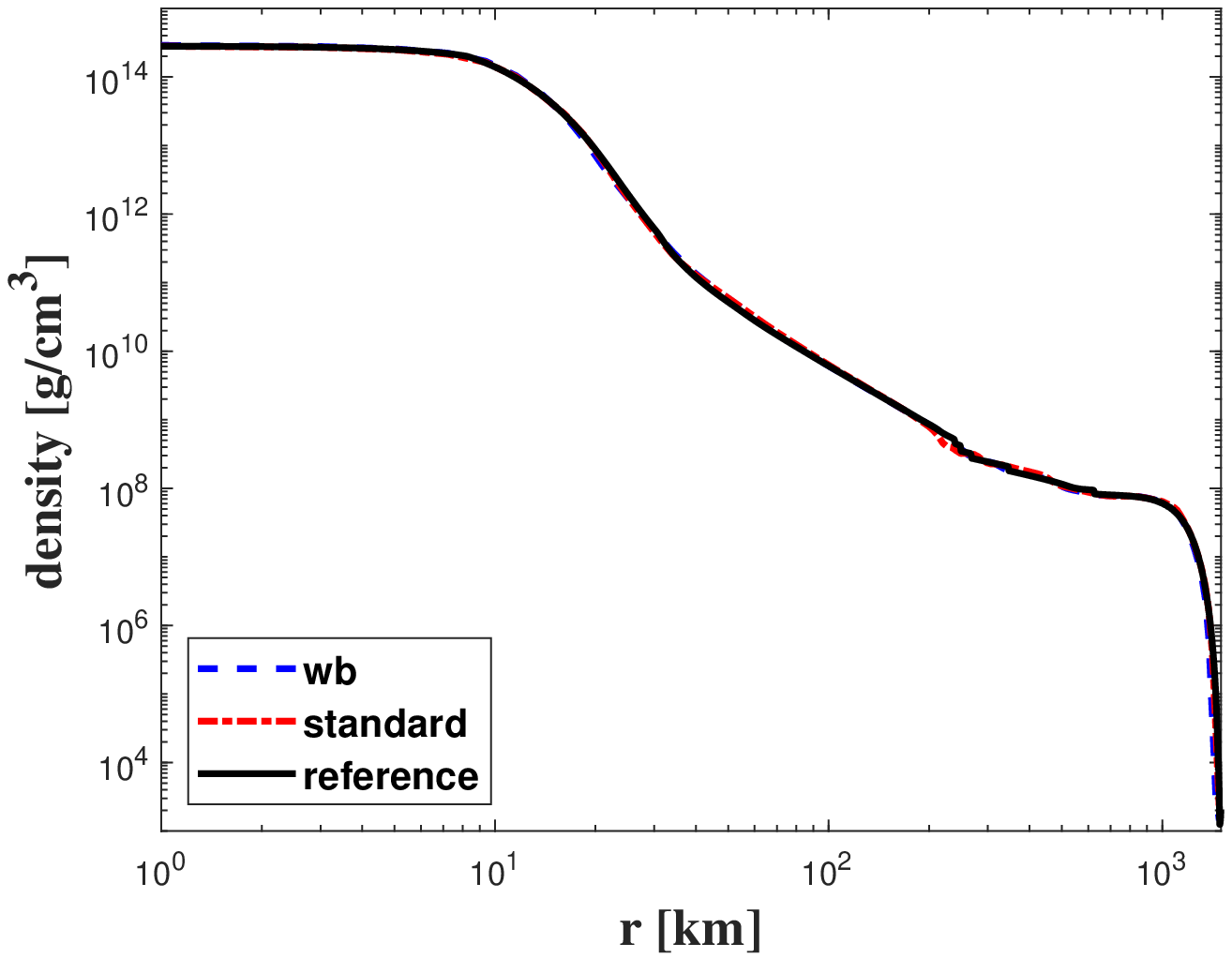}
	\includegraphics[width=0.48\linewidth]{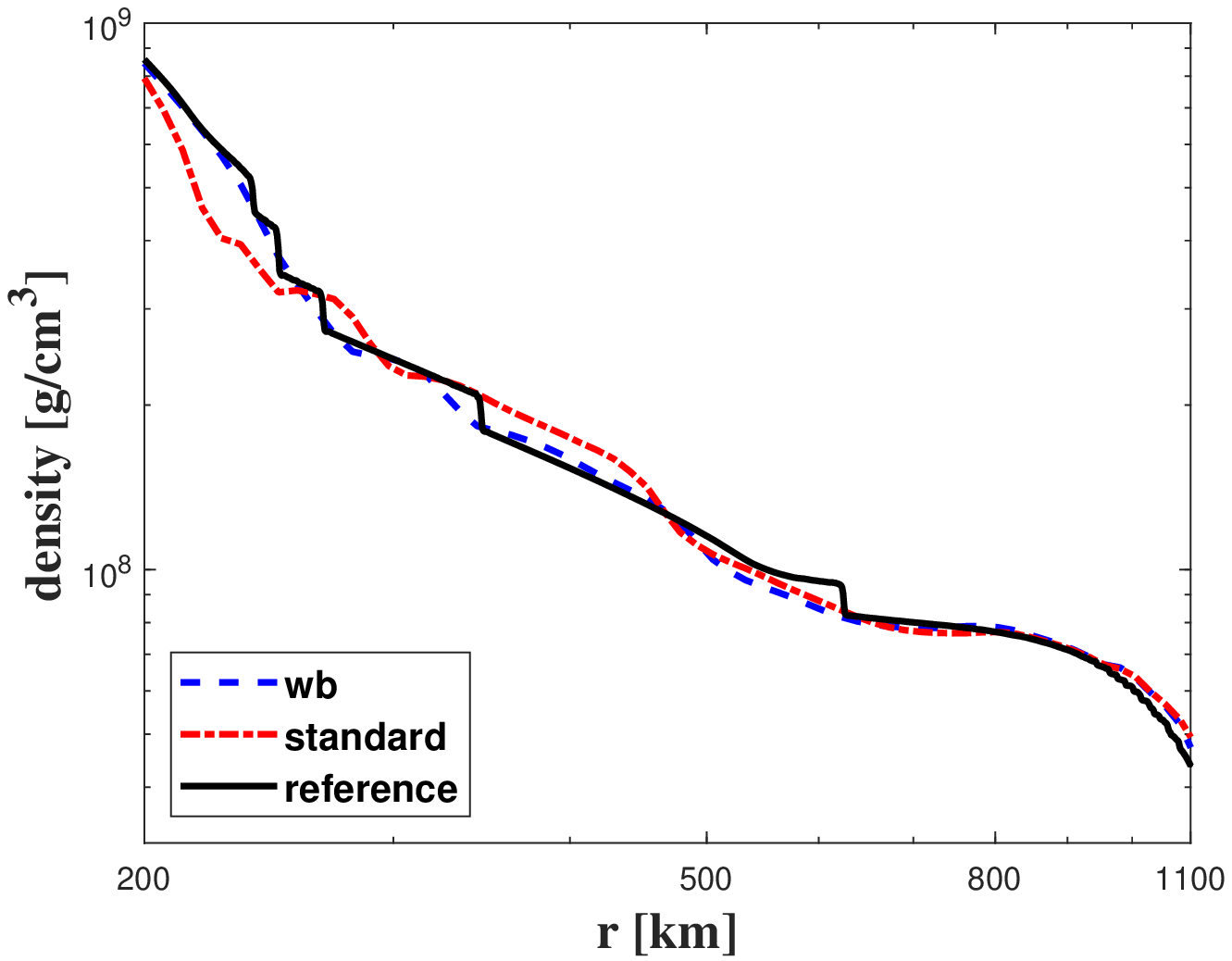}
	\includegraphics[width=0.48\linewidth]{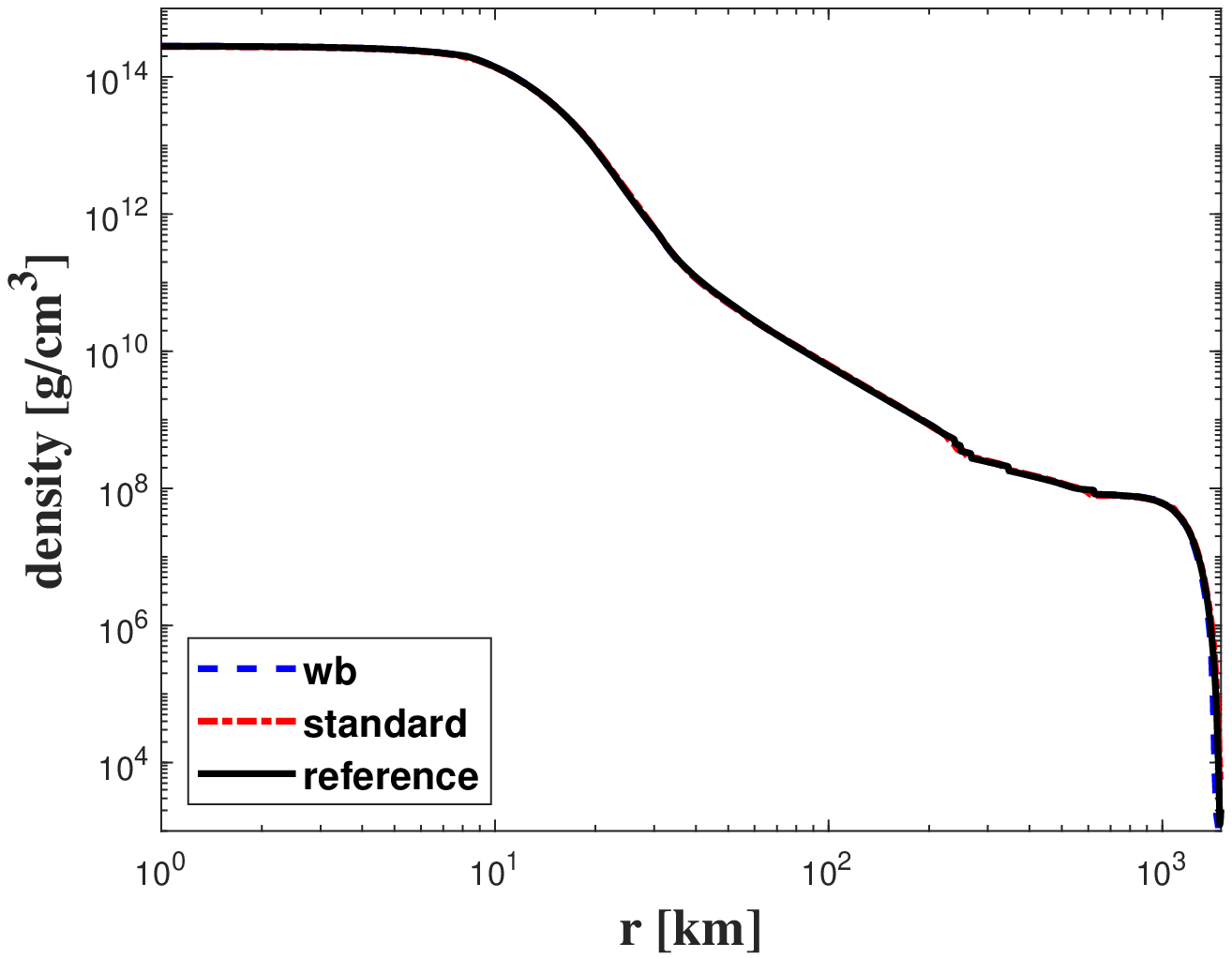}
	\includegraphics[width=0.48\linewidth]{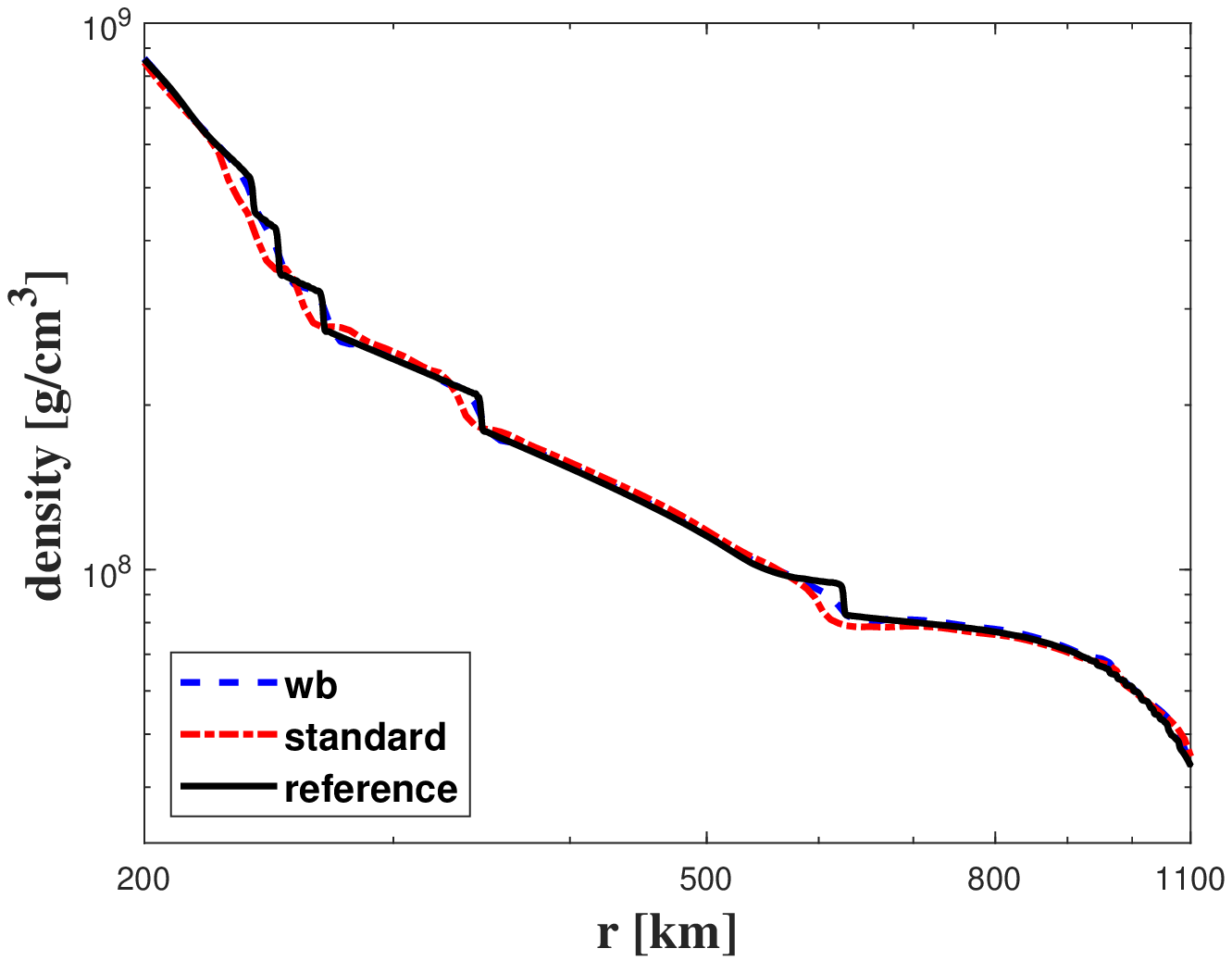}
	\caption{Example \ref{exam_toy}, the mass density versus radius at $t=0.11$ s of the proposed (blue dashed) and the standard DG scheme (red dash-dotted) with $N=128$ (top two), $256$ (bottom two) compared with a reference solution of $N=2048$ (black solid). The right two figures represent the zoom-in version at $r\in[200,1100]$~km.}
	\label{fig:final_density}
\end{figure*}

\begin{figure*}
	\centering
	\includegraphics[width=0.48\linewidth]{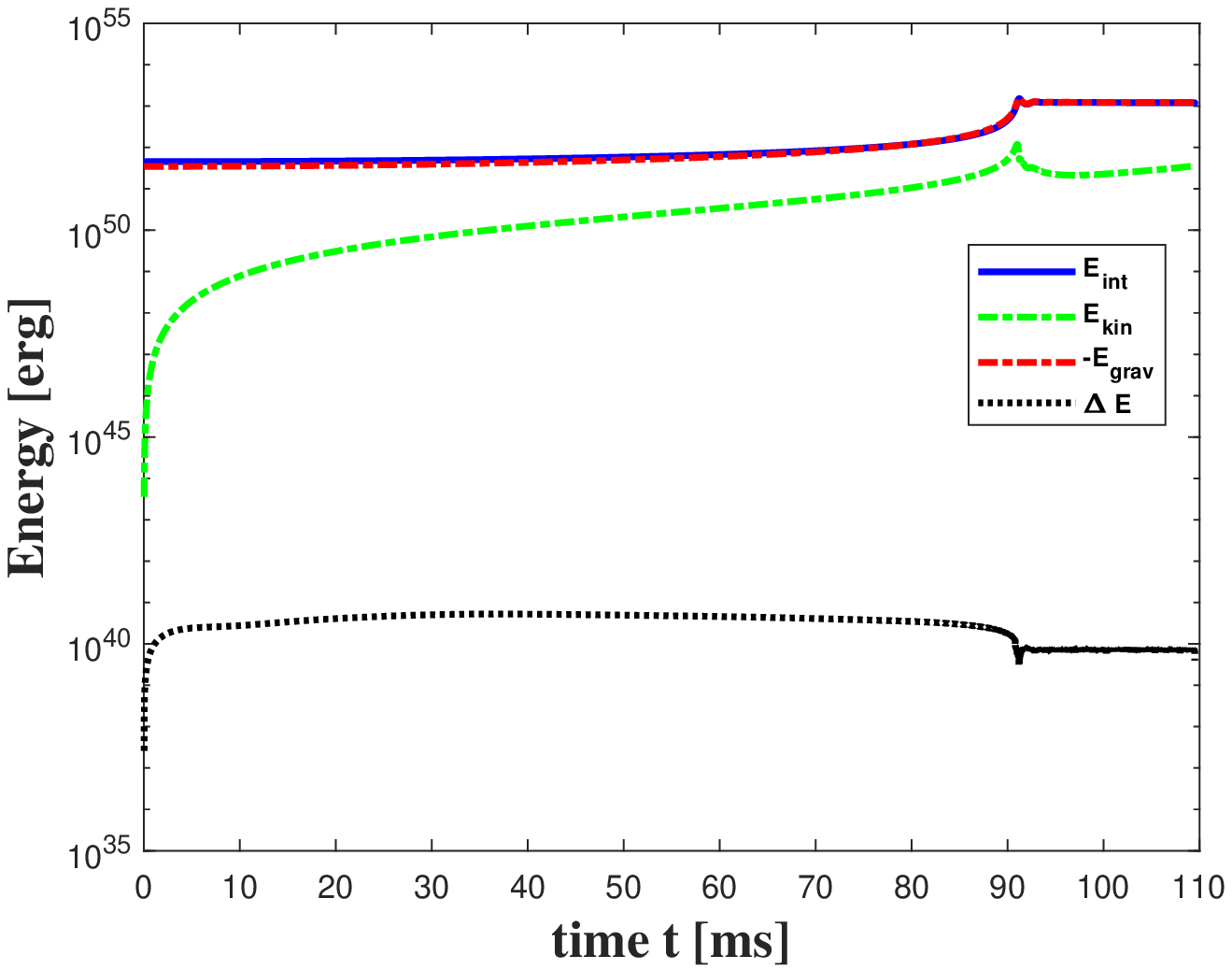}
	\includegraphics[width=0.48\linewidth]{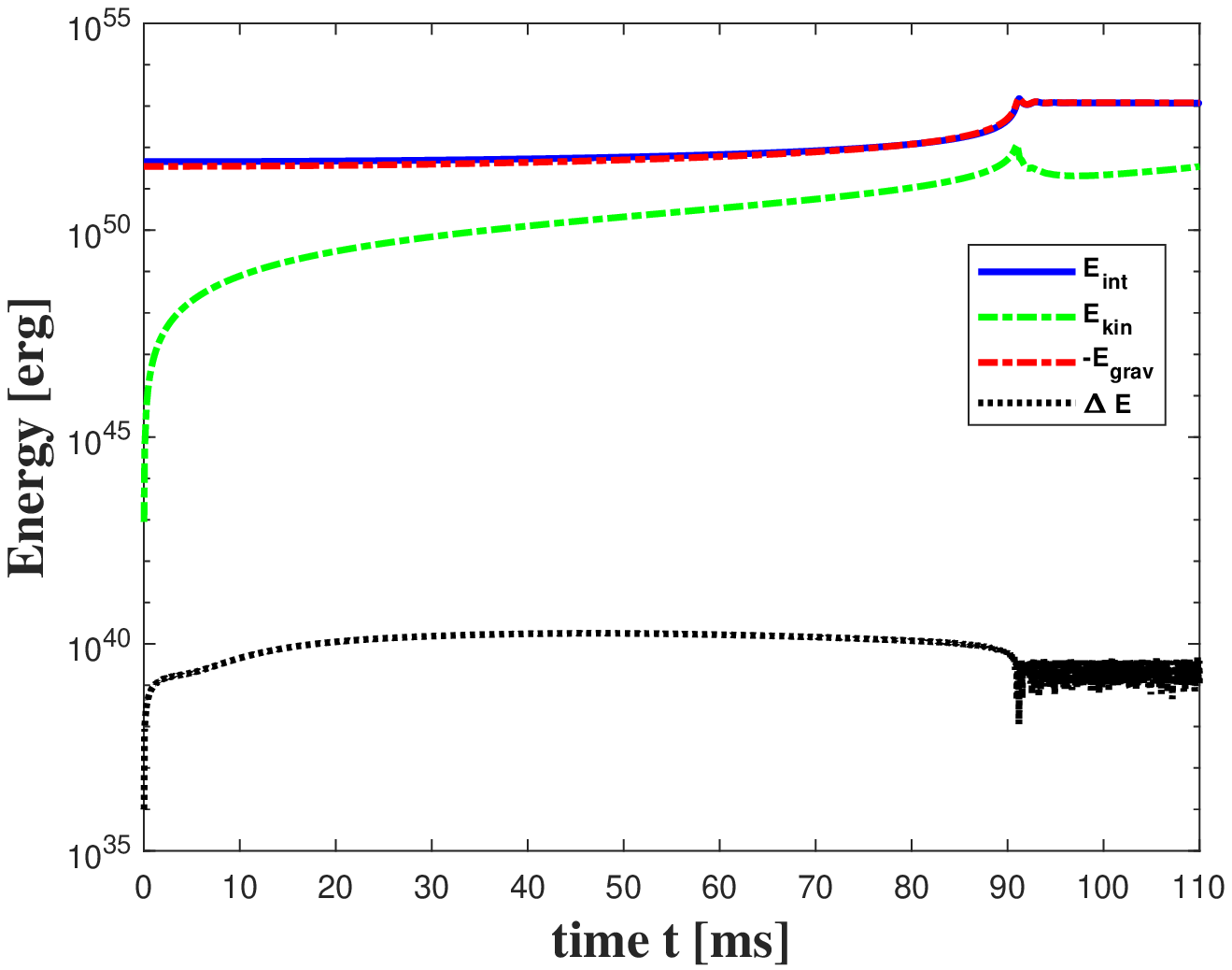}
	\includegraphics[width=0.48\linewidth]{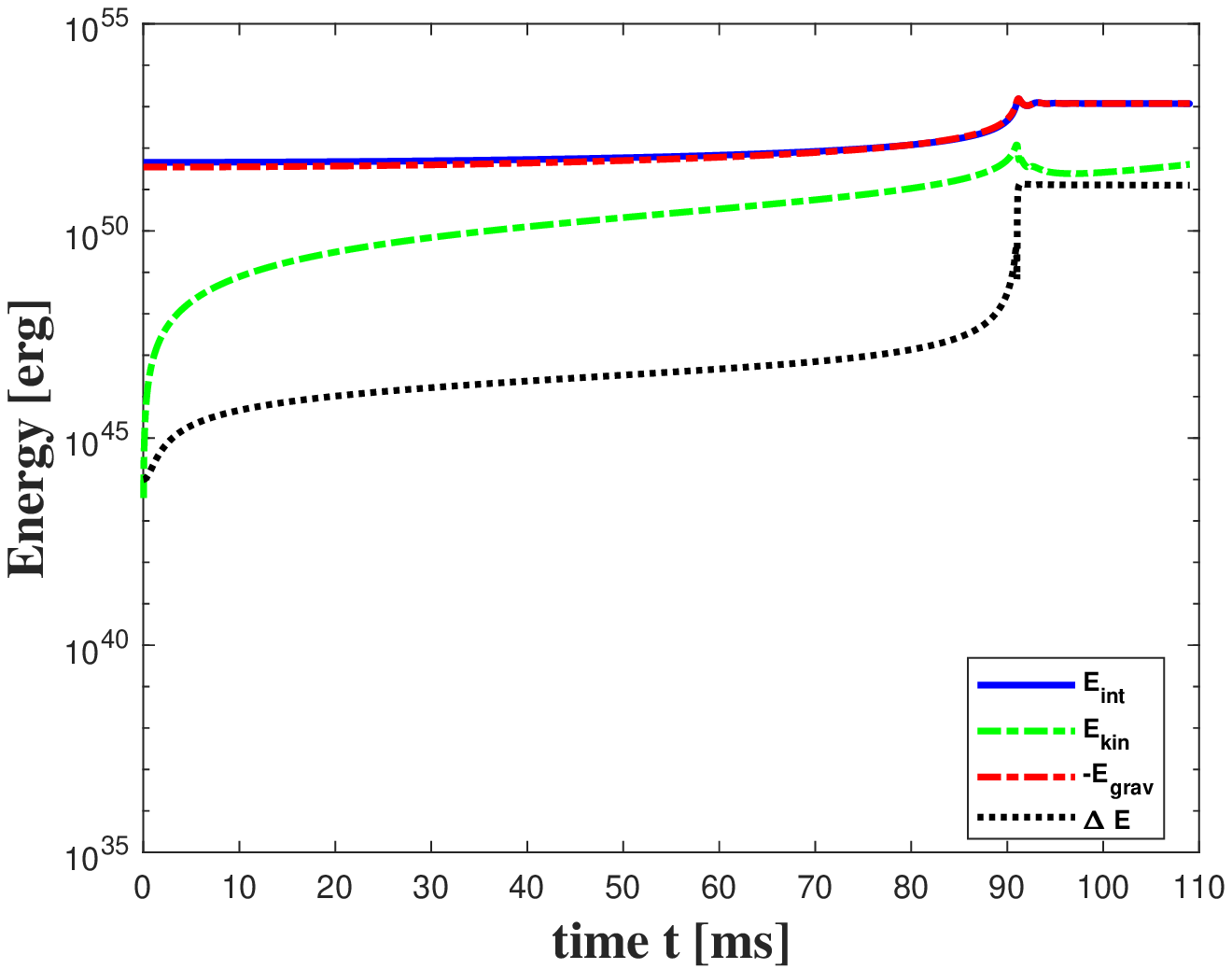}
	\includegraphics[width=0.48\linewidth]{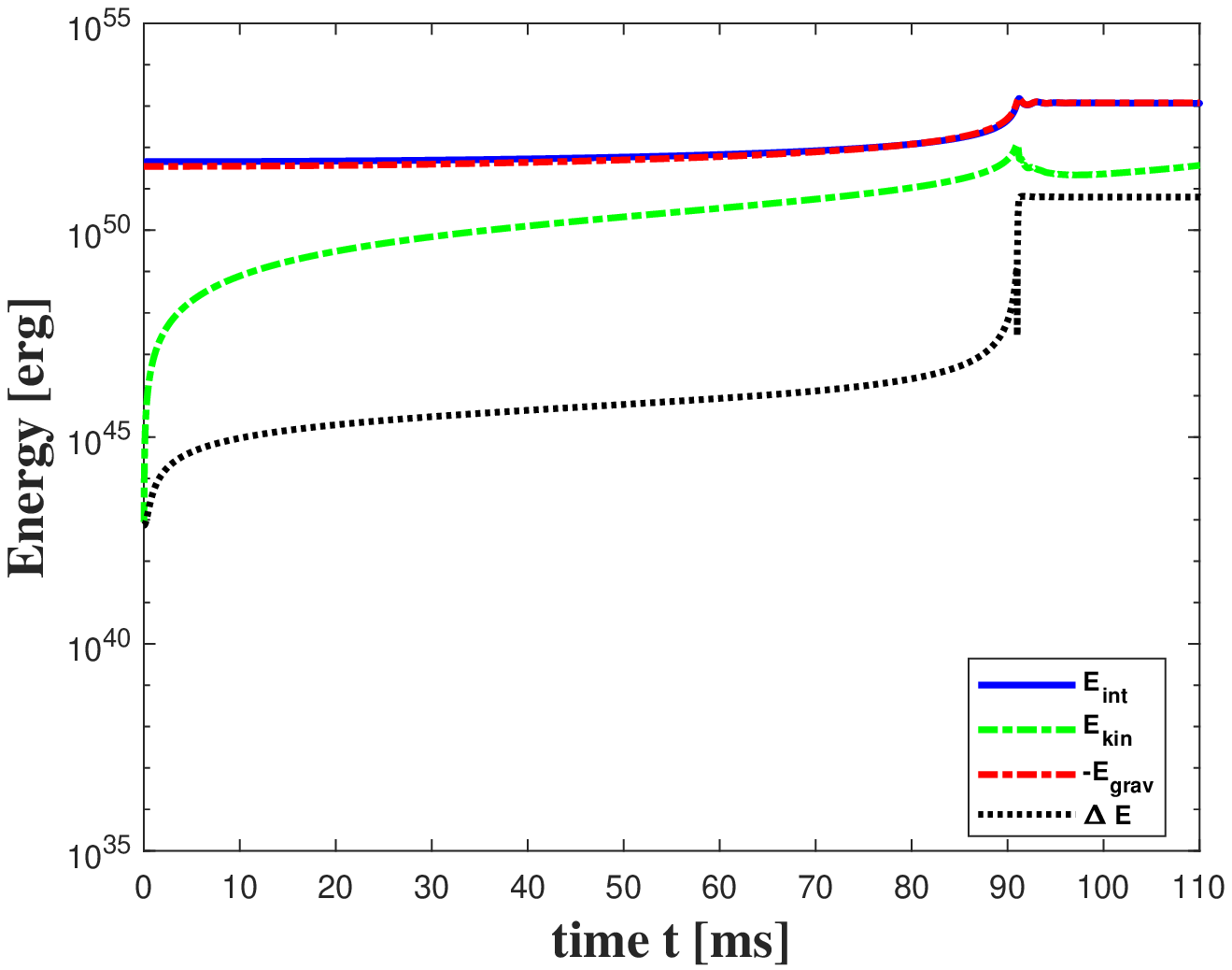}
	\includegraphics[width=0.48\linewidth]{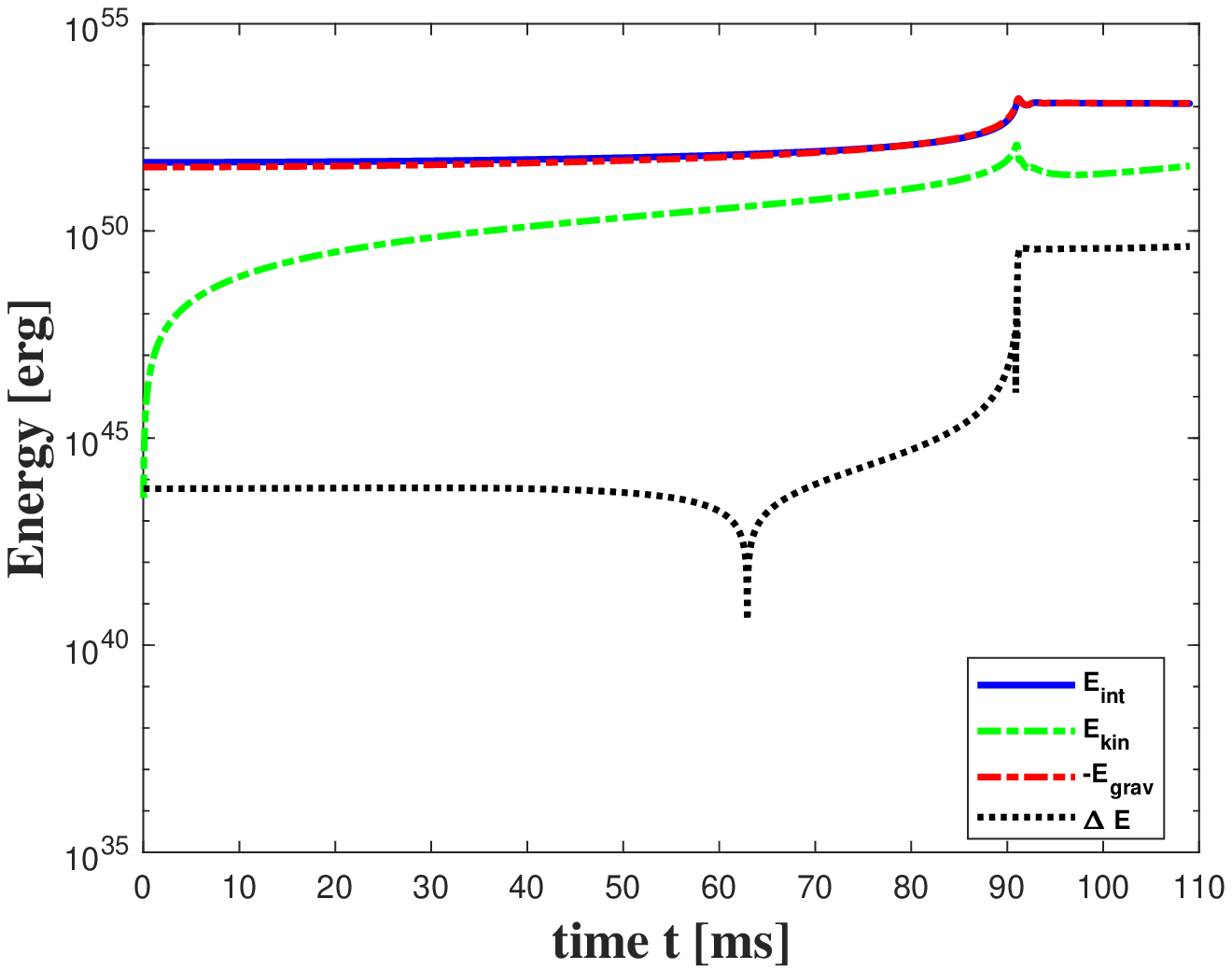}
	\includegraphics[width=0.48\linewidth]{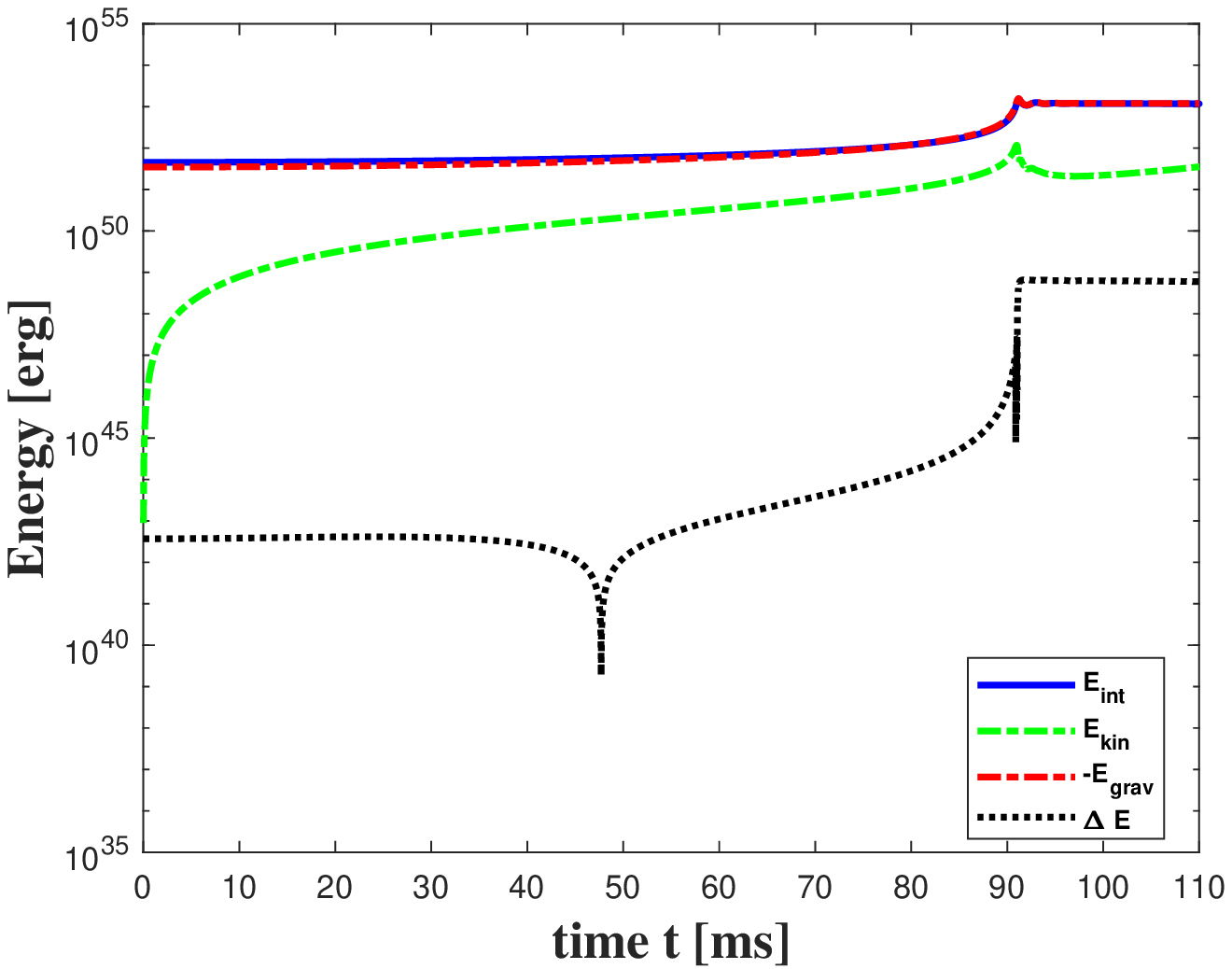}
	\caption{Example \ref{exam_toy}, the time history of the internal energy $E_{\rm int}$ (blue solid), kinetic energy $E_{\rm kin}$ (green dashed), negative gravitational energy $-E_{\rm grav}$ (red dash-dotted) and change in total energy $\Delta E$ (black dotted), with $N=128$ (left figures) and $N=256$ (right figures). We compared the solutions of our proposed scheme (in the top figures) and the standard DG scheme (in the mid figures) and standard DG scheme with correction term \eqref{eq:tvd2} (in the bottom figures).}
	\label{fig:energy}
\end{figure*}

\section{Summary and Conclusion}\label{sec:conclusion}

We have developed high-order, total-energy-conserving, and well-balanced discontinuous Galerkin (DG) methods for solving the Euler--Poisson equations in spherical symmetry. Our proposed scheme can preserve polytropic steady states and the total energy up to round-off errors. Key to these properties are the new way of recovering the steady states, the well-balanced numerical flux, the novel source term approximations (the well-balanced and total energy conserving parts), the total energy correction term for the limiter, and the newly defined time discretization. We have compared the performance of our proposed scheme with the standard scheme in several different situations, which all demonstrate the benefits of our proposed scheme. In all these examples, we can observe the round-off errors for the steady state solutions and total energy conservation, while the standard scheme can not.
In our opinion, the properties of our proposed scheme may be advantageous for simulating CCSNe in the context of non-relativistic, self-gravitating hydrodynamics. 

There are still challenges that remain to be solved in future works.  Importantly, CCSNe, and related systems where the methods developed here could be applicable, are inherently multidimensional due to, e.g., rotation, hydrodynamic instabilities, and magnetic fields \citep[][]{muller2020review}. The steady states considered in this work are valid only in spherical symmetry, and it will likely become much more complicated to generalize the well-balanced property to multiple spatial dimensions, which is the main reason we did not consider multidimensional methods in this paper. For extensions to multiple spatial dimensions, the main difficulty relates to how the desired steady states are characterized.  However, for problems that can be characterized as being nearly spherically symmetric (i.e., where the gravitational potential is dominated by the monopole component), such as CCSNe originating from slowly rotating stars, the methods developed here may potentially still be beneficial, but this remains to be investigated.  The extension of the energy conservation property to multiple spatial dimensions appears to be more straightforward, and will be considered in a future study.  Another topic to consider in a future work is the generalization of the well-balanced property to tabulated nuclear EoSs needed for more physically realistic models.  

\section*{Acknowledgements}
This work was carried out when W.~Zhang was visiting Department of Mathematics, The Ohio State University under the support of the China Scholarship Council (CSC NO. 201906340196).  
The work of Y.~Xing was partially supported by the NSF grant DMS-1753581. 
E.~Endeve acknowledges support from the NSF Gravitational Physics Theory Program (NSF PHY 1806692 and 2110177) and the Exascale Computing Project (17-SC-20-SC), a collaborative effort of the US Department of Energy Office of Science and the National Nuclear Security Administration.  

\section*{Data Availability Statements}
The data underlying this article will be shared on reasonable request to the corresponding author.

\bibliographystyle{mnras}
\bibliography{euler-self-gravity}


	\bsp	
	\label{lastpage}
\end{document}